\documentclass[egregdoesnotlikesansseriftitles,a4paper, 11pt,leqno,bibliography=totoc]{scrartcl}

\usepackage[english]{babel}			
\usepackage[utf8]{inputenc} 		
\usepackage[T1]{fontenc}		
\usepackage{lastpage}
\usepackage{hyperref}						
\usepackage[dvipdf, pdftex]{graphicx}		
\usepackage[margin=2ex,font=small,labelfont=bf]{caption}		
\usepackage{amsmath}
\usepackage{amsthm}
\usepackage{amssymb}
\usepackage{stmaryrd}						
\usepackage{mathtools}
\usepackage{stix}
\usepackage{faktor}
\usepackage{cite}
\usepackage{authblk}
\usepackage{abstract}
\mathtoolsset{showonlyrefs}

\usepackage{pgfplots}
\usetikzlibrary{decorations.markings}
\tikzset{
	set arrow inside/.code={\pgfqkeys{/tikz/arrow inside}{#1}},
	set arrow inside={end/.initial=>, opt/.initial=},
	/pgf/decoration/Mark/.style={
		mark/.expanded=at position #1 with
		{
			\noexpand\arrow[\pgfkeysvalueof{/tikz/arrow inside/opt}]{\pgfkeysvalueof{/tikz/arrow inside/end}}
		}
	},
	arrow inside/.style 2 args={
		set arrow inside={#1},
		postaction={
			decorate,decoration={
				markings,Mark/.list={#2}
			}
		}
	},
}

\mathtoolsset{centercolon,showonlyrefs,showmanualtags}

\newcommand{\intB}{\int_{B_1(0)}}
\addto\captionsenglish{\renewcommand{\proofname}[2][]{Proof{#1}{#2}. }}
\renewenvironment{proof}[1][]{{\noindent\textbf{\proofname{#1}}}}{\qed}

\newcommand{\eps}{\varepsilon}
\theoremstyle{plain}
\newtheorem{theorem}{Theorem}[section]
\newtheorem*{theorem*}{Theorem}
\newtheorem{lemma}[theorem]{Lemma}
\newtheorem*{lemma*}{Lemma}
\newtheorem{proposition}[theorem]{Proposition}
\newtheorem*{proposition*}{Proposition}
\newtheorem{corollary}[theorem]{Corollary}
\newtheorem{remark}[theorem]{Remark}

\newcommand{\laplace}{\Delta}
\newcommand{\grad}{\nabla}
\renewcommand{\L}{\mathcal{L}}
\newcommand{\la}{\langle}
\newcommand{\ra}{\rangle}
\newcommand{\mel}{\MoveEqLeft}

\newcommand{\N}{\mathbb{N}}
\newcommand{\R}{\mathbb{R}}
\renewcommand{\S}{\mathbb{S}}

\theoremstyle{definition}

\newtheorem*{definition*}{Definition}

\newtheorem*{beispiel*}{Beispiel}

\newtheorem*{bemerkung*}{Bemerkung}
\newtheorem*{erinnerung*}{Erinnerung}

\newtheorem*{assumption*}{Assumption}

\DeclareMathOperator\dist{dist}
\DeclareMathOperator\Span{span}
\DeclareMathOperator\supp{supp}
\DeclareMathOperator{\Lip}{Lip}
\DeclareMathOperator{\huelle}{span}
\DeclareMathOperator\const{\mathit{const}}

\newcommand{\vertiii}[1]{{\left\vert\kern-0.25ex\left\vert\kern-0.25ex\left\vert #1 
		\right\vert\kern-0.25ex\right\vert\kern-0.25ex\right\vert}}

\begin{document}
	\title{Invariant Manifolds for the Thin Film Equation}
	
	\author{
		Christian Seis \qquad Dominik Winkler}
	\affil{\textit{Institut f\"ur Analysis und Numerik, Westf\"alische Wilhelms-Universit\"at M\"unster}}
	\date{\today}
	
	\maketitle
	
	\begin{abstract}
	The large-time  behavior of solutions to the thin film equation with linear mobility in the complete wetting regime on  $\mathbb{R}^N$ is examined: We investigate the higher order asymptotics of solutions converging towards self-similar Smyth--Hill solutions under certain symmetry assumptions on the initial data. The analysis is based on a construction of  finite-dimensional invariant manifolds that solutions approximate to an arbitrarily prescribed order. 
	\end{abstract}
	
	\textbf{Statements and Declarations:} The  authors state that there is no conflict of interest. Data sharing is not applicable to this article as no data sets were generated or analyzed during the current study.

\section{Introduction}

In this work, we investigate the thin film equation with linear mobility  in arbitrary space dimensions, that is, the partial differential equation
\begin{equation}\label{TFE}
\partial_\tau u+\nabla \cdot \left(u\nabla\Delta u\right)=0
\end{equation}
in the whole space $\mathbb{R}^N$. This equation models the flow of an $N+1$ dimensional viscous fluid with high surface tension over a flat substrate, and thus, the real physical three-dimensional  setting corresponds to the case $N=2$. 
The evolving scalar variable $u=u(\tau ,y)$ in \eqref{TFE} represents the height of the liquid film, and is assumed to be nonnegative \cite{RevModPhys.69.931,MR1642807}. In the $1+1$ dimensional case, equation \eqref{TFE} can also be seen as the lubrication approximation in a two-dimensional  Hele--Shaw cell \cite{GiacomelliOtto03}.

The thin film equation is degenerate parabolic in the sense that the  diffusion flux decreases to  zero where $u$ vanishes. It follows that the speed of propagation is finite and  thus droplet configurations stay compactly supported for all times. 
On a mathematical level, we are thus concerned with a free boundary problem. We will be focusing on a setting in which droplet solutions are slowly spreading over the full space, a regime that is commonly referred to as \emph{complete wetting.} This is obtained mathematically by prescribing the contact angle at the droplet boundary $\partial\{u>0\}$ to be zero, that is, $\grad u=0$. 

A reference spreading droplet configuration is given by Smyth and Hill's 
self-similar solution \cite{SmythHill,MR1148286,MR1479525}
\begin{equation}
\label{smythhillsolution}
u_*(\tau ,y) = \frac{1}{\tau^{\frac{N}{N+4}}}\alpha_N\left(\sigma_M-\frac{\vert y\vert^2}{\tau^{\frac{2}{N+4}}}\right)^2_+,
\end{equation}
where $\alpha_N=\frac{1}{8(N+4)(N+2)}$ and $\sigma_M$ is a positive constant only depending on the mass constraint
\[
\int\limits_{\R^N} u_*dy=M.
\]
Moreover, we write $(s)_+$ for the positive part $ \max\!\left\{0,s\right\}\!$ of a quantity $s$. These source-type solutions \eqref{smythhillsolution} play a distinguished role in the theory since they are, similar to related parabolic problems, believed to describe the large time asymptotic behavior of \emph{any} solution of mass $M$ to the thin film equation, i.e.,
\begin{equation}
\label{119}
u(\tau ,y)\approx u_*(\tau ,y)\quad \mbox{for any }\tau \gg1.
\end{equation}
This convergence has been  proved for strong solutions  in the one-dimensional setting ($N=1$) via entropy methods by Carrillo and Toscani \cite{CarrilloToscani02} and for  minimizing movement solutions in arbitrary dimensions via gradient flow techniques by Matthes, McCann and Savar\'e \cite{MatthesMcCannSavare09}. Both contributions provide sharp rates of convergence and exploit the intimate relation between the thin film equation \eqref{TFE} and the porous medium equation
\begin{equation}
\label{118}
\partial_{\tau} u - \laplace u^{m}=0
\end{equation}
in the case $m=3/2$. In fact, up to a suitable rescaling, the Smyth--Hill solutions \eqref{smythhillsolution} coincide with the self-similar Barenblatt solutions \cite{Zelcprimedovic1950,Barenblatt1952,Pattle1959} of the porous medium equation \eqref{118}, and the surface energy, which is dissipated by the thin film equation \eqref{TFE}, coincides with the rate of dissipation of the Tsallis entropy under the  porous medium flow \eqref{118}. See \cite{MatthesMcCannSavare09,McCannSeis15} for a clean formulation of this entropy-information relation from a gradient flow perspective.

The link between the two equations can be further exploited in order to get  deeper insights into the large time behavior of solutions to \eqref{TFE}: When linearizing both equations about the self-similar solutions, it turns out that the linear porous medium operator $\L$ translates into the linear thin film operator in a simple algebraic way, namely $\L^2 +N\L$ \cite{McCannSeis15}. It immediately follows that the eigenfunctions of both operators agree, while the transformation of the eigenvalues from the porous medium setting to the thin film setting obeys the same algebraic formula. The operator $\L$ was diagonalized in \cite{ZeldovicBarenblatt58,Seis14}, and thus, the full spectral information is also available for the thin film equation \cite{McCannSeis15}, see Theorem \ref{spektrum} below. The spectrum of the one-dimensional operator was computed earlier in \cite{BernoffWitelski02}.

The knowledge of the complete spectrum does not  only give information on the sharp rate of convergence (for which information about the spectral gap would be sufficient), but also on the geometry of all modes through the knowledge of all eigenfunctions. One may thus analyze in detail the role played by affine symmetries such as dilations, rotations or shears, and we will in this paper  obtain improved rates of convergence for the thin film equation \eqref{TFE} by quoting out such symmetries. Further details on the large time asymptotics can be formulated after a suitable change of variables. 

Higher order large time asymptotics for the porous medium equation  \eqref{118} with $m>1$  were obtained in one dimension by Angenent \cite{Angenent1988}, building up on the spectral information in \cite{ZeldovicBarenblatt58} and, more recently, in any dimension by the first author \cite{Seis15}, building up on \cite{Seis14}. While Angenent derived fine series expansions around the limiting solution, the later multidimensional contribution takes a geometric point of view by constructing  finite-dimensional invariant manifolds that the solutions approximate   to any given order. In the present work, we will derive a parallel theory for the thin film equation.
Invariant manifold studies  can be found in numerous applications in the field of nonlinear partial differential equations,  for instance, \cite{HirschPughShub77,Carr83,FoiasSaut84,ChowLu88,FoiasSellTemam88,
	ConstantinFoiasNicolaenko89,EckmannWayne91,ChowLinLu91,
	VanderbauwhedeIooss92,Wayne97,GallayWayne02}.
What is particularly challenging in \cite{Seis15} and the present paper is the moving free boundary at which  solutions cease to be smooth.

What is needed for linking the spectrum of the linear operators to the nonlinear dynamics \eqref{TFE} or \eqref{118} is a regularity framework in which solutions depend differentiably on the initial configuration. This is necessary since the precise rate of convergence in the limit \eqref{119} is dictated by the particular choice of the initial datum. Identifying such a framework is far from being trivial. A crucial first step is a nonlinear change of variables that transforms the free boundary problem into an evolution equation on a fixed domain, which can be chosen as the unit ball. The linear leading order part of the equation can then be seen as a degenerate parabolic equation, whose degeneracy can be cured by interpreting the the dynamics as a fourth-order heat flow on a weighted Riemannian manifold.
For the porous medium equation, such this setting was proposed by Koch in his habilitation thesis \cite{KochHabilitation},  further refined in the work of Kienzler \cite{Kienzler16} and then adapted in \cite{Seis15}. An analogous theory for the thin film equation was derived by John in \cite{John15} and later adapted by the first author in \cite{SeisTFE}. After some necessary refinements, the latter  will be the starting point for the present study.

We also like to mention the related studies  by Denzler, Koch and McCann \cite{DenzlerMcCann05,DenzlerKochMcCann15} and Choi, McCann and the first author \cite{ChoiMcCannSeis22},  who derived some improved large time asymptotics for the fast diffusion equation, i.e., \eqref{118} with $m<1$, in the full space and a bounded domain, respectively. The full space setting is particularly challenging due to the occurrence of continuous spectrum, which arises from the fact that the associated Barenblatt profile possesses a finite number of moments, while in a bounded domain, in which solutions extinct in finite time, negative (unstable) eigenvalues challenge the leading order asymptotics \cite{BonforteFigalli21}. 

Before giving in the next section a specific description of our setting and of our main results,   we want to finish this  introductory section with a brief discussion about the state of the art in the mathematical theory for thin film equation.
Existence of nonnegative weak solutions was established with the help of compactness arguments and estimates on the free surface energy by Bernis and Friedman \cite{MR1031383}. This approach is not adequate to prove a general uniqueness result even though the regularity of these solutions could be improved, see \cite{MR1328475,MR1371925,MR1616558}. In a neighborhood of  stationary solutions (of infinite mass), well-posedness and regularity of one-dimensional solutions could be established in  a weighted Sobolev setting \cite{GiacomelliKnupferOtto08} and in H\"older spaces \cite{GiacomelliKnupfer10}. Moreover, the aforementioned work \cite{John15} deals with the multidimensional case and lowers the regularity requirements to Lipschitz norms and Carleson-type measures. The latter approach was adapted to neighborhoods of the  Smyth--Hill self-similar solution in \cite{SeisTFE}. The one-dimensional setting was also considered in \cite{Gnann15} using  weighted Hilbert spaces.    We finally remark that for nonlinear mobilities, solutions are in general not smooth, see \cite{GGKH14,Gnann16}.

\medskip

\textbf{Organization of the paper.} In the following section, we state and discuss our results on the large time asymptotics in self-similar variables. In Section \ref{newvariables}, we rewrite the thin film equation as a perturbation equation around the self-similar Smyth--Hill solution and present our main theorems of this paper, including the Invariant Manifold Theorem. We will describe in Section \ref{S4} how these results for the perturbation equation translate into the large-time asymptotics for the thin film equation. Section \ref{theoryforperturbationequation} collects information on the well-posedness of the perturbation equation and improves on known regularity estimates. The subsequent Section \ref{S5} deals with a truncated version of the perturbation equation. Well-posedness and regularity estimates are provided. Moreover, we introduce and discuss the time-one mapping that will be our main object of consideration in our construction of invariant manifolds in Section \ref{dynamicalsystemarguments}. The final Section \ref{applicationinvariantmanifolds} exploits the invariant manifold theory to prove the large-time asymptotic expansions for the perturbation equation. We conclude with two appendices, one with a derivation of the perturbation equation, one with  inequalities for weighted Sobolev spaces.

\section{Higher order asymptotics for the thin film equation}\label{section2}

In order to study the convergence towards self-similar solutions, it is customary to perform a self-similar change of variables. In view of the particular form of the Smyth--Hill solution \eqref{smythhillsolution}, we choose
\begin{align}\label{121}
	x = \frac{1}{\sqrt{\sigma_M}}\frac{1}{\tau^{\frac{1}{N+4}}}y, \quad t= \gamma^{-1} \log\left(\tau^{\frac{1}{N+4}}\right) \quad \text{ and } \quad v = \frac{(N+4)\gamma}{\sigma_M^2}\tau^{\frac{N}{N+4}}u,
\end{align} 
where $\gamma = 2(N+2)$, which transforms equation \eqref{TFE} into the \emph{confined} thin film equation
\begin{align}\label{FPE}
	\partial_t v +  \nabla \cdot \left(v\grad\laplace v\right)- \gamma\grad\cdot \left(xv\right)=0,
\end{align}
and turns the self-similar solution \eqref{smythhillsolution} into a stationary one,
\begin{align}\label{123}
	v_*(x) = \frac{1}{4}\left(1-\vert x \vert^2\right)_+^2.
\end{align}
We remark that under this  change of variables, the initial time will be transferred from $\tau =0$ to $t=-\infty$. As we are interested into the solutions' large time behavior only, we will hereafter treat $0$ as the initial time for the transformed equation. Moreover, the rescaling incorporates the total mass $M$ through $\sigma_M$ in such a way that the stationary $v_*$ is the limiting solution only if $v$ and $v_*$ have the same total mass. In what follows, we will assume that this is always the case be requiring that
\begin{equation}
\label{120}
\int\limits_{\R^N} v_0\, dx = \int\limits_{\R^N} v_*\, dx,
\end{equation}
if $v_0$ is the initial configuration for the evolution in \eqref{FPE}.

The theory in \cite{SeisTFE} guarantees that the confined thin film equation \eqref{FPE} has a unique regular solution provided that $v_0$ and $v_*$ are sufficiently close in the sense
\begin{equation}
\label{o1}
\|\sqrt{v_0} -  V_* \|_{W^{1,\infty}(\supp v_0)} \ll1,
\end{equation}
where $V_*(  x) = \frac12(1-|  x|^2)$ is the  (unsigned) extension of $\sqrt{ v_*}$ to $\R^N$. This condition actually yields strong estimates between $v_0$ and and the exact stationary solution $v_*$ as will be explained in the following remark.

\begin{remark}\label{R10}
Choosing the globally decaying $V_*$ over $\sqrt{v_*}$ in \eqref{o1} has the advantage that we can infer from it simultaneously  an information on the support of $v_0$,  a global estimate on the difference of $v_0$ and $v_*$, and  a bound on the slope of $v_0$.

Indeed, regarding the first, restricting to the boundary of the support, where $v_0$ vanishes, and noticing that $V_*(x)\sim \dist(x,\partial B_1(0))$, we directly deduce 
	\begin{align}\label{301}
		\sup \limits_{x\in \partial \supp v_0}\dist(x, \partial B_1(0)) \ll1.
	\end{align}

Next, we observe that   $V_*=\sqrt{v_*}$ inside the ball $B_1(0)$. Outside of $B_1(0)$   it holds that $0\le  \sqrt{v_0} -  \sqrt{v_*}  = \sqrt{v_0} \le \sqrt{v_0}-V_*$, and thus, we find 
	\begin{equation}\label{303}
	\| \sqrt{v_0} - \sqrt{v_*}\|_{L^{\infty}(A)} \ll1 
	\end{equation}
	on the set $A = \supp v_0$ as a consequence of \eqref{o1}. Moreover, since $V_*-  \sqrt{v_0}= V_* =  \sqrt{v_*}$ on $B_1(0)\cap \partial \supp v_0$ and since $v_*$ is decaying towards the boundary, the  estimate \eqref{303} holds true on also on $A = B_1(0)  \setminus \supp v_0 $. It remains to notice that $v_0 = v_*=0$ on the remaining set $A=B_1(0)^c\cap (\supp v_0)^c$, and thus \eqref{303} is proved to be true with $A=\R^N$.	
We immediately deduce that
	\begin{align}\label{304}
		\left\|v_0-v_*\right\|_{L^\infty\left(\R^N\right)}\ll1,
	\end{align} 
because $|v_0-v_*| = |\sqrt{v_0}-\sqrt{v_*}|(\sqrt{v_0}+\sqrt{v_*}) \lesssim|\sqrt{v_0}-\sqrt{v_*}|$	where the last identity is true because $v_0$ and $v_*$ are bounded. 
		
	Finally, we can also extract a condition on the slope of $v_0$, namely
	\begin{align}\label{305}
		\left\|\nabla v_0 + 2x \sqrt{v_0}\right\|_{L^\infty\left(\R^N\right)} \ll1.
	\end{align}
	To establish \eqref{305}, we first note that the left-hand side vanishes provided that $x$ does not lie in the support of $v_0$. Inside $ \supp v_0$, we have
	$|\nabla v_0+ 2x\sqrt{v_0}| =  2\sqrt{v_0}|\nabla \sqrt{v_0}- \nabla V_*|\lesssim |\nabla \sqrt{v_0}- \nabla V_*|$, because $v_0$ is bounded. Condition \eqref{o1} then yields the claim. Notice that the left-hand side in \eqref{305} vanishes precisely for $v_0=v_*$ (under the mass constraint \eqref{120}).
\end{remark}
The main results of the referred work \cite{SeisTFE} are repeated in more details later in Section \ref{theoryforperturbationequation}. This section also contains the main results of the present work. At this stage, we present some consequences of that general theory for the confined thin film equation \eqref{FPE}, which provide  exemplary improved convergence rates towards equilibrium by quoting out symmetries.

The rates of relaxation being intimately related to the spectrum of the linear operator $\L^2+N\L$ associated to the confined equation \eqref{FPE}, see Section \ref{newvariables} below, for a better understanding of our results presented in the sequel, we recall the findings of the spectral analysis from the literature.

\begin{theorem}[\cite{BernoffWitelski02,McCannSeis15}]\label{spektrum}
	The operator $\mathcal{L}^2+N\mathcal{L}$ has a purely discrete spectrum consisting of the eigenvalues
	\begin{align}
	\mu_{l,k}= \lambda_{l,k}^2+N\lambda_{l,k},
	\end{align}
	where the $\lambda_{l,k}$ are the eigenvalues of $\mathcal{L}$. They are given by 
	\begin{align}
	\lambda_{l,k}= 2\left(l+2k\right)+2k\left(k+l+\frac{N}{2}-1\right),
	\end{align}
	for $(l,k)\in \mathbb{N}_0\times \mathbb{N}_0$ if $N\geq2$ and $(l,k) \in \{0,1\}\times \mathbb{N}_0$ if $N=1$.
	The corresponding eigenfunctions are polynomials 	of degree $l+2k$, namely
	\begin{align}
	\psi_{l,n,k}(x) = {}_2F_1\left(-k,1+l+\frac{N}{2}+k;l+\frac{N}{2};|x|^2\right)Y_{l,n}\left(\frac{x}{|x|}\right)|x|^{l},
	\end{align}
where $n\in \{1,\dots,N_l\}$ with $N_0=1$ or $N_1=N$ and $N_l=\frac{\left(N+l-3\right)!\left(N+2l-2\right)}{l!\left(N-2\right)!}$ if $l\geq 2$. Besides, ${}_2F_1(a,b;c;d)$ is a hypergeometric function and $Y_{l,n}$ is a spherical harmonic (of degree $l$) if $N\geq 2$, corresponding to the eigenvalue $l(l+N-2)$ of $-\Delta_{\S^{N-1}}$ with multiplicity $N_l$. If $N=1$ it is $Y_{l,n}\left(\pm1\right)= \left(\pm1\right)^l.$
\end{theorem}

The computation of the linear operator in \cite{McCannSeis15} was rather formal and was derived from the gradient flow interpretation of \eqref{FPE} with respect to the Wasserstein metric tensor \cite{Otto98,MR1865003,MatthesMcCannSavare09}. It occurs naturally after suitable rescaling inn the perturbation equation \eqref{perturbationequation}

In the statement of the theorem, the linear operator is analyzed with respect to the Hilbert space introduced in \eqref{124} below, and the eigenfunctions $\psi_{l,n,k}$ give rise to  an orthogonal basis of that Hilbert space.

We recall that hypergeometric functions can be written as power series of the form
\begin{align}
{}_2F_1(a,b;c;z)=\sum\limits_{j=0}^{\infty} \frac{(a)_j(b)_j}{(c)_jj!}z^j,
\end{align}
where $a,b,c,z \in \mathbb{R}$ and $c$ is not an non-positive integer, see, e.g.~\cite{Rainville71}. The definition uses extended factorials, also known as Pochhammer symbols,
\begin{align}
(s)_j=s(s+1)\cdots (s+j-1), \quad \text{ for } j\geq 1 \text{ and } (s)_0=1.
\end{align}
The hypergeometric functions with $z = |x|^2$   reduce to a polynomial of degree $2k$ if we plug in $-k$ for $a$. In this case, they can be expressed as Jacobi polynomials.

\begin{figure}[t]\begin{center}
\includegraphics[width=1\textwidth]{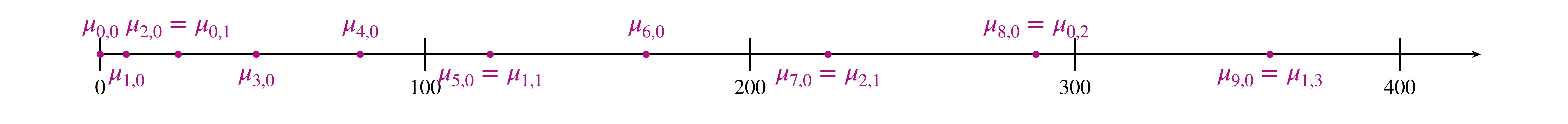}
\end{center}\caption{The spectrum of the linear operator with multiples in the range $[0,400]$ in the $2+1$ dimensional setting ($N=2$). }\label{fig1}
\end{figure}
In the one-dimensional setting, all eigenvalues have multiplicity one. In higher dimensions, all eigenvalues with $l\ge 2$ have a dimension dependent multiplicity that stems from the multiplicity of the eigenvalue $l(l+N-2)$ associated with  the spherical harmonics, i.e., the eigenfunctions of the Laplace--Beltrami operator $\laplace_{\S^{N-1}}$. In addition, there are certain intersections between the eigenvalues $\mu_{\cdot, k}$ and $\mu_{\cdot, k+n}$. For instance, in two dimensions, it holds $\mu_{l,k} = \mu_{l(k+1) + k(k+2),0}$ for any $k,l$, see Figure \ref{fig1}.

\subsection{Leading order asymptotics}

Apparently, $\mu_{0,0}=0$ is the smallest eigenvalue. It corresponds to a situation in which the convergence in \eqref{119} fails, which is precisely the case if the equal mass condition \eqref{120} is not satisfied. Conversely, by requiring that \eqref{120} holds, this eigenvalue is automatically eliminated. The exact leading order asymptotics are then governed by the second smallest eigenvalue $\mu_{1,0}=4+2N$, which is  
our first result for solutions to the confined thin film equation.  We will derive it from a more general statement in Theorem \ref{Whoeheremoden} in Section \ref{newvariables} and present it thus as a corollary here.
 
 \begin{corollary}[Exact leading order asymptotics]\label{exactleadingorder}
 	Let $v$ be the solution to \eqref{FPE} with initial data $v_0$ satisfying the mass constraint \eqref{120} and being sufficiently  close to $v_*$ in the sense of \eqref{o1}. Then it holds that
 	\begin{align}
 		\left\|\sqrt{v(t)}-V_*\right\|_{W^{1,\infty}\left(\supp v(t)\right)} &\lesssim e^{-(4+2N) t} \quad \text{ for all } t\geq 0.
 	\end{align}
 \end{corollary} 

The result entails the convergence of  $v(t)$ towards $v_*$ as outlined in Remark \ref{R10}.

The same rate of convergence was established earlier in terms of the relative Tsallis entropy and the $L^1$ norm by Carrillo and Toscani \cite{CarrilloToscani02} in the one-dimensional setting and by Matthes, McCann and Savar\'e in any dimension (if one takes into account the difference in the time scaling that we introduced in \eqref{FPE} through the $\gamma^{-1}$ factor). It corresponds to an  $O(\tau^{-(N+1)/(N+4)})$ convergence in the limit \eqref{119} for the original thin-film equation \eqref{TFE}.

The convergence rate in this theorem is sharp and  is saturated by spatial translations of the stationary solution $v_*$. Indeed, for every  vector  $b\in \mathbb{R}^N$, the function $v(t,x)=v_*(x-e^{-\gamma t}b)$ solves the confined thin film equation exactly   and approaches $v_*(x)$ with exponential rate $\gamma  =4+2N$, as can be readily checked via Taylor expansion.
However, because the original  equation \eqref{TFE} is invariant under spatial translations, the convergence in \eqref{119} with rate $O(\tau^{-(N+1)/(N+4)})$ remains true for any shifted version of the Smyth--Hill solution, i.e., $u(\tau,y)\approx u_*(\tau,y-b)$, and the significance of this rate is thus an artifact of this symmetry. Indeed, the above arguing shows that the convergence in Corollary \ref{exactleadingorder} is sharp only if we are not willing to pick the ``correctly'' centered Smyth--Hill solution. We may equivalently adjust the initial datum by a suitable translation in $\R^N$. As we will see, the ``correct'' choice for  $b$ is the  center of mass, which is preserved under the original evolution \eqref{TFE} and pushed towards the origin by the confined equation,
\begin{equation}\label{motioncenterofmass}
\int xv(t,x)dx = e^{- \gamma t}b_0,\quad b_0=\int x v_0(x)dx,
\end{equation}
for all $t\geq 0$, because our rescaling \eqref{121} has eliminated the translation invariance. Supposing that $v_0$ is centered at the origin, $b_0=0$,  the eigenvalue $\mu_{1,0}$ drops out of the spectrum and we obtain a better rate of convergence, namely by the next smallest eigenvalue, which is $\mu_{0,1}=30$ if $N=1$, and $\mu_{2,0} = 16+4N$ if $N\ge 2$. 

\begin{corollary}\label{firstordercorrection}
	Let $v$ be as in Corollary \ref{exactleadingorder} and assume in addition that $v_0$ is centered at the origin, i.e., $b_0=0$ in \eqref{motioncenterofmass}.  Then, it holds that
	\begin{align}
	\left\|\sqrt{v(t)}-V_*\right\|_{W^{1,\infty}\left(\supp v(t)\right)}&\lesssim e^{- 30t} \quad \text{ for all } t\geq 0
	\end{align}
	if $N=1$, and
	\begin{align}
	\left\|\sqrt{v(t)}-V_*\right\|_{W^{1,\infty}\left(\supp v(t)\right)}&\lesssim e^{-(16 +4N)t} \quad \text{ for all }t\geq 0
	\end{align}
	if $N\ge 2$.
\end{corollary}
This rate of convergence is again sharp for solutions that start, if $N\ge 2$, from affine transformations of the stationary solution, and if $N=1$, from dilated stationary solutions. Because we will discuss dilated stationary solutions later also in the multi-dimensional case, we will restrict ourselves here to the setting $N\ge 2$. Solutions starting from affine transformations of $v_*$ are then to leading order (modulo rescaling to fit the mass constraint) described by $v(t,x)\approx v_*(x - e^{-\mu_{2,0}t} Ax)$ for a symmetric and trace-free matrix $A$. The validity of this  asymptotics is best understood in terms of the perturbation equation, that we will introduce in the subsequent section. 

The occurrence of such affine transformations can be explained on the level of the eigenfunctions computed in Theorem \ref{spektrum}: The finite displacements $v_s$  generated by an eigenfunction $\psi$ are described by
\begin{equation}
\label{122}
v_s( x+s\grad\psi(x)) \det(I+s\grad^2 \psi(x))  = v_*(x),
\end{equation}
provided that $|s|\ll 1$. For $k=0$, the eigenfunctions are homogeneous harmonic polynomials of degree $l$, namely $\psi_{l,n,0}(x) = Y_{l,n}(x/|x|) |x|^l$. If $l=2$, the generating polynomials are quadratic, and thus of the form $\psi(x) = x\cdot Ax$  for a symmetric and trace-free matrix $A$. In this case, \eqref{122} defines  affine transformations.

For further improvements on the rate of convergence, we have to quote out affine transformations.

\subsection{Higher order corrections and the role of symmetries}

In order to improve on the convergence rates even further, we exploit symmetry invariances of the thin film equation in conjunction with symmetry properties of spherical harmonics, which determine  the angular   modulations of our eigenfunctions, see Theorem \ref{spektrum}. More precisely, we will obtain higher order convergence rates by assuming that the initial datum $v_0$ is invariant under certain orthogonal transformations. Because such transformations leave the thin film equation invariant and thanks to the uniqueness of solutions near self-similarity \cite{SeisTFE}, the invariance under those orthogonal transformations is inherited by the solution for all times. We will show that the orthogonality condition leads to a selection among the eigenfunctions forcing a large class of eigenmodes to remain \emph{inactive} during the evolution. The slowest active mode will then govern the large-time asymptotics.
\begin{figure}
\begin{picture}(95,90)
\put(0,0){\includegraphics[scale=0.25]{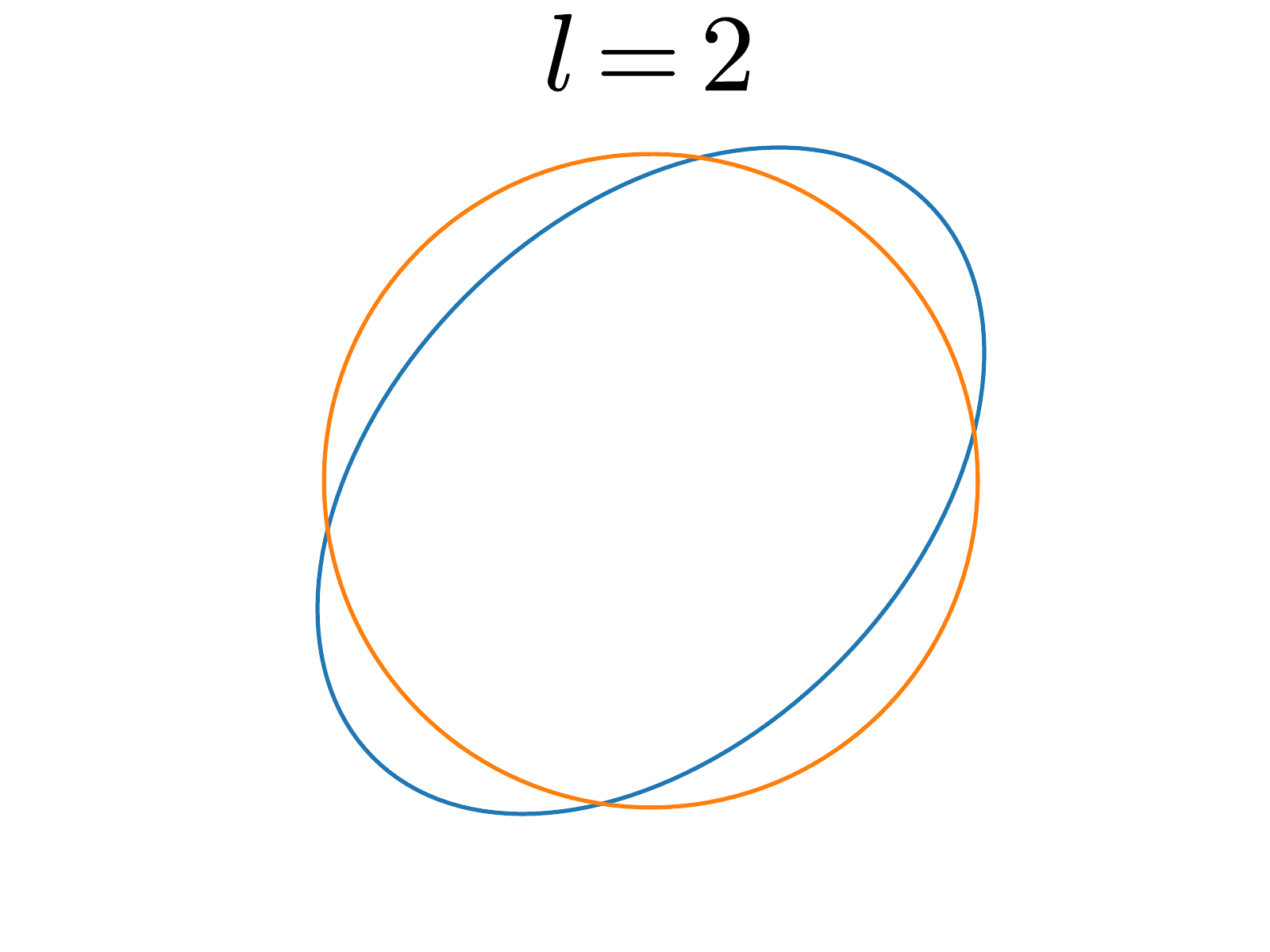}}
\end{picture}
\begin{picture}(95,90)
\put(0,0){\includegraphics[scale=0.25]{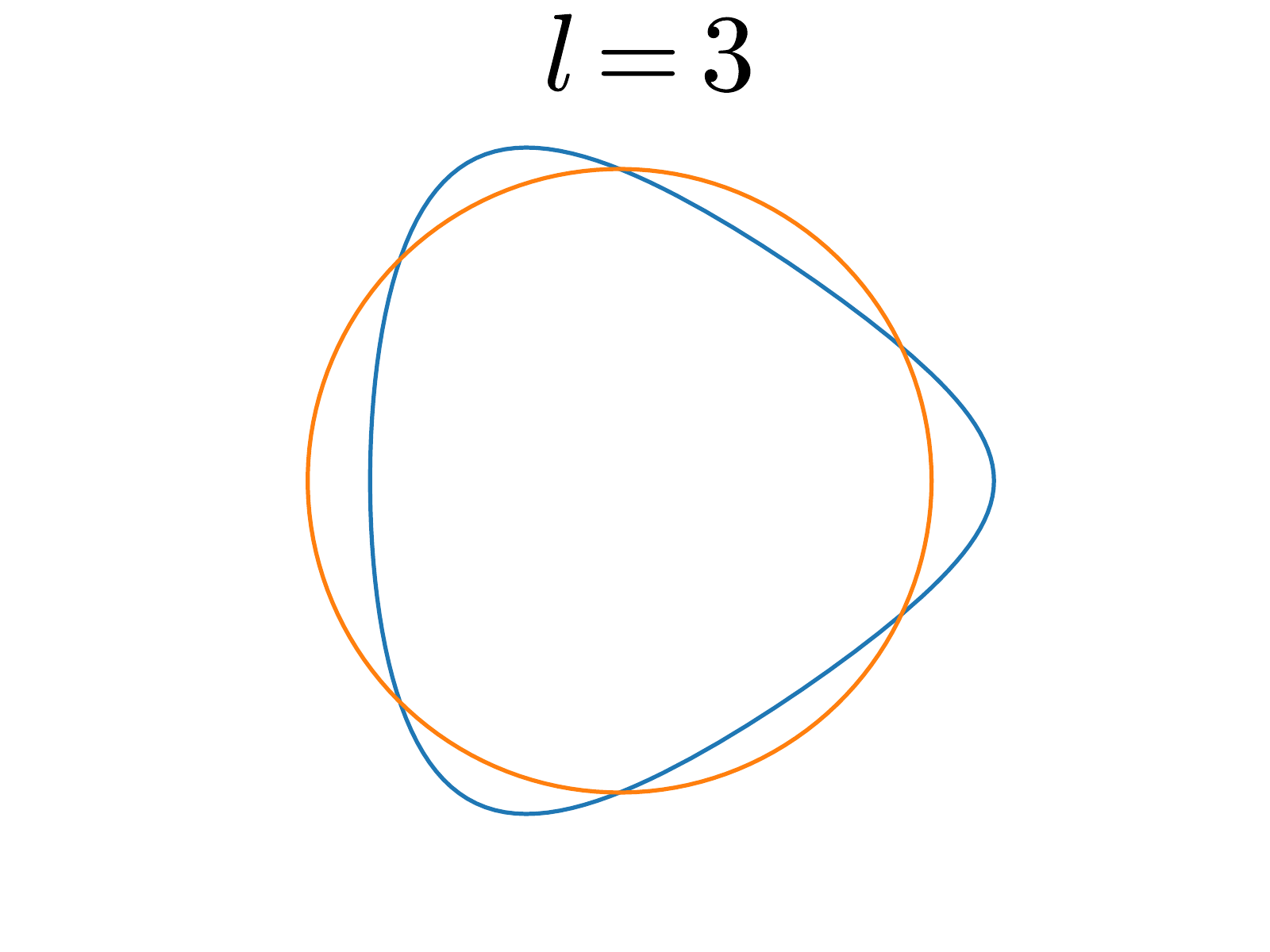}}
\end{picture}
\begin{picture}(95,90)
\put(0,0){\includegraphics[scale=0.25]{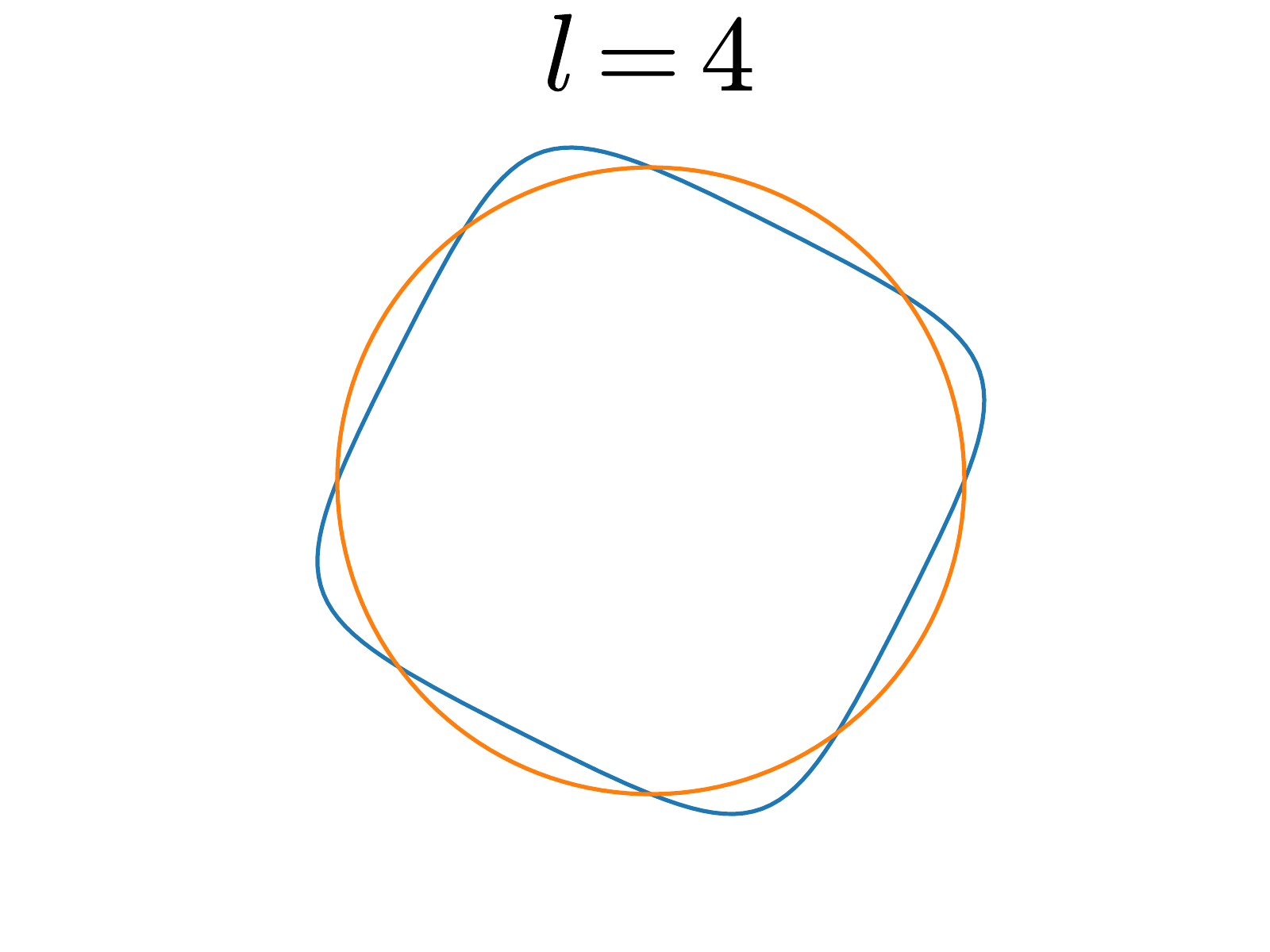}}
\end{picture}
\begin{picture}(95,90)
\put(0,0){\includegraphics[scale=0.25]{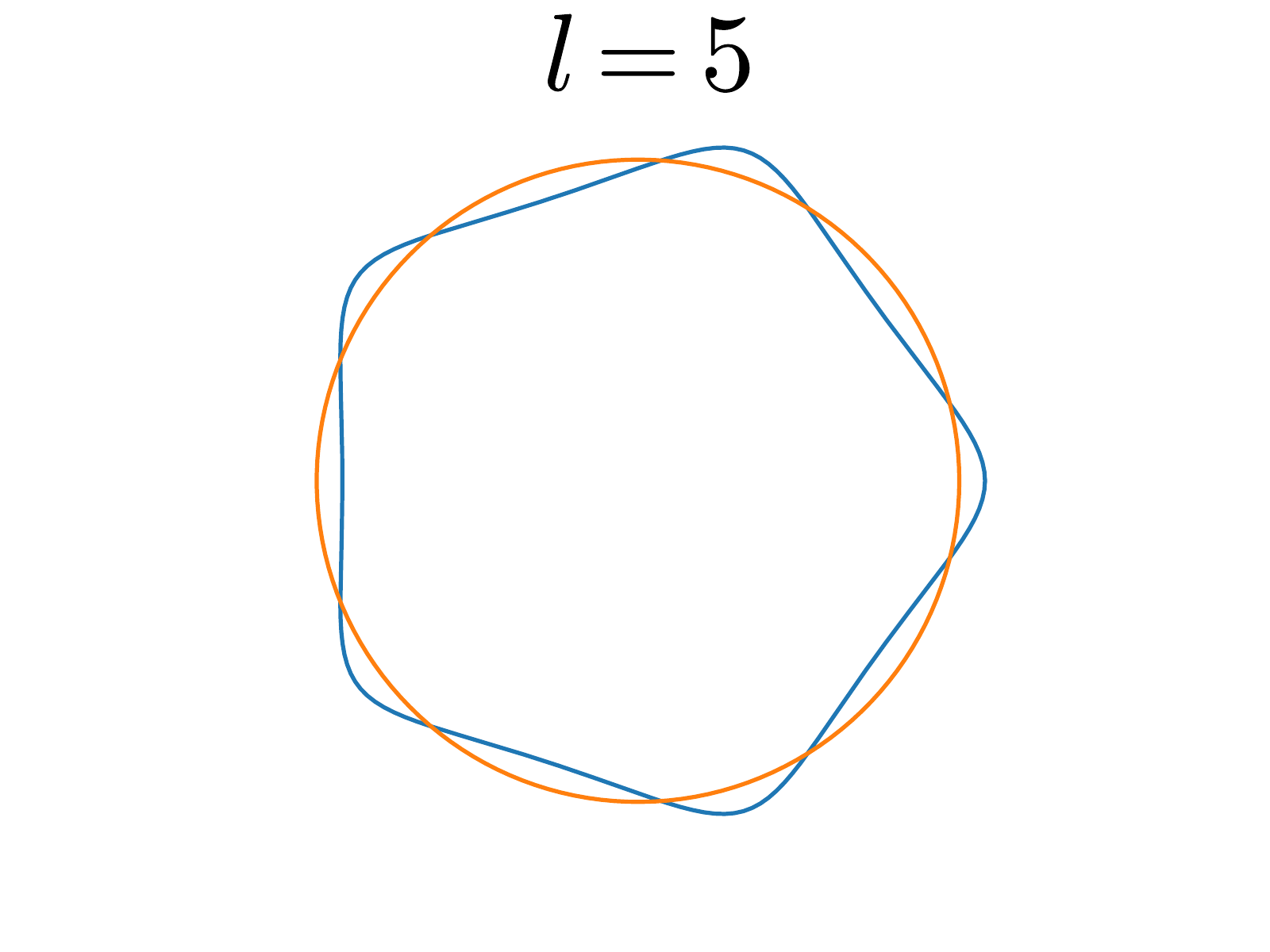}}
\end{picture}
\caption{The finite displacements of $v_*$ generated by the eigenfunctions $\psi_{l,n,0}$ in the physical case $N=2$.}\label{fig0}
\end{figure}

To motivate our approach for modding out certain modes, it is enlightening to study briefly the situation in two space dimensions, $N=2$. In Figure \ref{fig0}, we have plotted some finite displacements, cf.~\eqref{122}, generated by  eigenfunctions $\psi_{l,n,0}$ with $l\in\{1,\dots, N_l\}$. Apparently, displacements generated by $\psi_{l,n,0}$ (and then also by any polynomial of the form $p(|x|)Y_{l,n}(x/|x|)$ including $\psi_{l,n,k}$) share precisely the symmetry properties of a regular $l$-polygon. Under the assumption that the solution has the symmetry properties such a regular $l$-polygon, all eigenmodes generated by $\psi_{m,n,k}$ with $m<l$ are necessarily inactive. In Remark \ref{R2} below, we will discuss the short elementary argument that rigorously supports this observation.

In higher space dimensions, the situation gets more involved and the structure of the spherical harmonics is more complex. In order to mod out eigenmodes, taking a more abstract approach is strongly advised. We choose a group theoretical approach, noticing that  the symmetry group of a regular $l$-polygon is a finite subgroup of the group of orthogonal transformations $O(N)$. Our goal is to determine geometric conditions on an arbitrary function, more precisely,  invariances under the action of a given finite subgroup of $O(N)$,   which guarantee that the $L^2$-projections of that function onto all spherical harmonics of a given degree $l$ vanish. 
To achieve this goal, we will eventually apply tools originating from the field of representation theory of groups, see, e.g.,  \cite{Mukai2003} or \cite{tomDieck1995} for elementary considerations.

The space of square integrable functions on the unit sphere $L^2\left(\mathbb{S}^{N-1}\right)$  can be decomposed into a direct Hilbert sum over the eigenspaces of $\laplace_{\S^{N-1}}$,
\[
L^2\left(\mathbb{S}^{n-1}\right) = \bigoplus\limits_{l\in \N_0}H_l,
\]
where the eigenspace $H_l$ is spanned by the spherical harmonics of degree $l$ and its dimension is given by $N_l$, see Theorem \ref{spektrum}. 
We remark that every eigenspace $H_l$ is invariant under the action of orthogonal transformations. More precisely, given an orthogonal matrix $g\in O(N)$, for every $f\in H_l $ we have that $f\circ g^{-1} \in H_l$.

If  $E$ is a finite subgroup of $O(N)$, we denote by  $H_l^E$ the subspace of $H_l$ consisting of all functions that are invariant under the action of all elements of $E$, i.e.,  $f\circ g^{-1}=f$ for any $f\in H_l^E$ and $g\in E$. The eigenmodes corresponding to an eigenvalue $\mu_{l,k}$ are all modded out by the action of elements in $E$ if that subspace is trivial, $\dim(H_l^E)=0$. We present and discuss our final convergence result under such an abstract condition and will discuss thereafter some specific choices of $E$, for which we will need some deeper insights from the representation theory of finite groups.

\begin{corollary}\label{grouptheorycorrections}
	Let $N\geq2$ and $v$ be given as in Corollary \ref{firstordercorrection} satisfying 
	\begin{equation}\label{401}
	\left\|\sqrt{v(t)}-V_*\right\|_{L^\infty(\supp v(t))}\lesssim e^{-\mu_{l,k}t} \quad \text{ for all }t\geq 0
	\end{equation}
	for some $l\in\N$ and $k\in \mathbb{N}_0$, such that the multiplicity of $\mu_{l,k}$ is given by $N_l$. Assume in addition that $v_0$ is invariant under the action of a finite subgroup $E$ of $O(N)$ such that
	\begin{align} \label{400}
		\dim\left(H_l^E\right)=0.
	\end{align}
	Then it holds that
	\begin{align}\label{h7}
	\left\|\sqrt{v(t)}-V_*\right\|_{W^{1,\infty}(\supp v(t))}&\lesssim e^{-\mu_+t} \quad \text{ for all }t\geq 0,
	\end{align}
	where $\mu_+$ is the next largest eigenvalue following $\mu_{l,k}$.
\end{corollary}

We shall briefly comment on the assumptions on $v(t)$ in the latter corollary. 
\begin{remark}\label{R3}
	It may be surprising that it suffices to demand the decay of $v(t)-\sqrt{V_*}$ in $L^\infty$ instead of $W^{1,\infty}$, what would be the expected setting due to the previous results. Due to the regularizing properties of the equation and the Lipschitz bound \eqref{o1} for the initial time, we will eventually see, that both assumptions are in fact equivalent in the given situation. We will discuss this phenomenon shortly in the proof of Corollary \ref{grouptheorycorrections}.
\end{remark}

Not every  eigenfunction corresponds to an orthogonal transformation and thus, a symmetry condition like \eqref{400} is in general not sufficient to jump from one eigenvalue to another. Indeed, all eigenfunctions $\psi_{0,1,k}$ are radially symmetric polynomials, and the slowest of the corresponding modes is generated by delayed Smyth--Hill solutions $u_*(\tau + \tau_0,y)$   of \eqref{TFE}, which turn into the dilations $\lambda(t)^{-N}v_*(\lambda(t)^{-1}x)$ with $\lambda(t)\approx 1 + \frac{1}{N+4}\tau_0 e^{-\mu_{0,1}t} $ solving the confined equation \eqref{FPE}, and converging towards the stationary $v_*$ with exponential rate $\mu_{0,1}$. We do not know if these modes can be eliminated  by a reasonable assumption on the initial configuration nor do we see how they can be suitably controlled during the evolution.  Therefore, in order to raise the convergence rates beyond eigenvalues $\mu_{0,k}$, the decay hypothesis \eqref{401}  seems necessary  to ensure that the respective  radial modes are inactive. We have to demand that the  multiplicity of the eigenvalue $\mu_{l,k}$ in \eqref{401} is precisely $N_l$, in order to  exclude possible resonances with  \emph{any} spherical harmonics of different  order (such that $\mu_{l,k} = \mu_{\tilde l,\tilde k}$).

To conclude the discussion about higher order asymptotics on the level of the confined thin film equation, we remark that the number of eigenvalues we are able to remove from the spectrum before reaching  $\mu_{0,1}$ (provided we find a suitable subgroup of $O(N)$) depends on the space dimension:  If the dimension is odd, $ N=2m-1$, then $\mu_{0,1}$ is the $(m+2)$th eigenvalue and has multiplicity one. In even dimensions, $N=2m$,  it coincides with $\mu_{m+2,0} $. 

We finally recall from the introduction that further and, in fact, much stronger statements on the large time asymptotics can be derived after a customary change of variables. These will be presented and discussed in the following section.

It remains to identify finite subgroups $E$ of $O(N)$, which mod out spherical harmonics of a given order $l$ in the sense of \eqref{400}. We will do that by applying a surprisingly helpful tool, the Molien series, which originates from the field of representation theory of groups. It was suggested to us by our colleague Linus Kramer.

The subspace $H_l\subseteq L^2\left(\S^{N-1}\right)$ of spherical harmonics of degree $l$ can be identified with the space of symmetric, trace-free tensors of rank $l$ that we will further denote by $H_l$ as well.
The generating function $h_E(t)$ for the dimensions $
\dim\left(H_l^E\right)$ of the subspace of $L^2(\S^{N-1})$ that is  invariant under the action of $E$ can be formally expressed as the power series
\begin{align}\label{generatingfunction}
		h_E(s)=\sum \limits_{l=0}^{\infty}\dim\left(H_l^E\right)s^l,
\end{align}
which is called Molien series or Hilbert series in the literature, cf.\ \cite[p.~11]{Mukai2003} or \cite[p.~479]{Stanley1979}.
A beautiful and functional way that is often used to compute this series explicitly is given by Molien's formula
\begin{align}
	h_E(s)=\frac{1}{|E|}\sum\limits_{g\in E}\frac{1-s^2}{\det\left(I-sg\right)},
\end{align} 
see \cite{Mendes1975}, \cite{BurnettMeyer1954}.
In the physical case $N=2$, the Molien series is known for all  finite subgroups of $O(2)$, as will be discussed in the following.

\begin{itemize}
\item \emph{Cyclic groups.} The first class of subgroups, $\mathfrak{S}_n$ for $n\in \mathbb{N}$, is generated by rotations by an angle of $2\pi/n$. The corresponding Molien series is given by
\begin{align}
	h_{\mathfrak{S}_n}(s)=\frac{1+s^n}{1-s^n}=(1+s^n)\sum\limits_{l=0}^\infty s^{ln} = 1 +2s^{n}+2s^{2n}+2s^{3n}+\dots,
\end{align}
see \cite[p.~143]{BurnettMeyer1954}. In view of the Hilbert series representation \eqref{generatingfunction} of $h_{\mathfrak{S}_n}(s)$, this formula proves that the corresponding invariant subspaces   must be trivial \eqref{400} precisely if $l$ is not divisible by $n$.
In other words, the projection of a function that is invariant under rotations of an angle of $2\pi/n$ onto the subspaces spanned by spherical harmonics of degree $l$ has to vanish if $l$ is not divisible by $n$. Moreover, if non-trivial, $H_l^{\mathfrak{S}_n}$ has dimension $2$, and thus,  recalling that for $N=2$ each of the tensor spaces $H_l$ with $l\geq 1$ is two-dimensional  $N_l=2$, it is $H_l^{\mathfrak{S}_n} = H_l$.

\item \emph{Dihedral groups.}
The second class of finite subgroups, $\mathfrak{D}_n$ for $n\in \mathbb{N}$, is generated by two elements. Again a rotation of the angle $2\pi/n$ and additionally a reflection. In this case the Molien series reads as
\begin{align}
	H_{\mathfrak{D}_n}(s)=\frac{1}{1-s^n}=\sum_{l=0}^\infty s^{ln}=1+s^{n}+s^{2n}+s^{3n}+\dots,
\end{align}
see \cite[p.~59]{Humphreys1990}.
If a function is invariant under the action of $\mathfrak{D}_n$ instead of $\mathfrak{S}_n$, the projection onto $H_l$ vanishes for the same $l$ as before. This time, however, the nontrivial subspace are one-dimensional.
\end{itemize}

We remark that the zeroth order term $s^0=1$ in the Molien series does not affect the convergence rates since the mass of the initial datum $v_0$ is already fixed.

In higher dimensions, classifying the finite subgroups of $O(N)$ becomes more complicated.
For $N=3$, we discuss the subgroups of $O(3)$ that only consist of rotations in more detail. The following results, together with more far-reaching ones, can be found in \cite[p.~143]{BurnettMeyer1954}.
\begin{itemize}
	\item \emph{Cyclic groups.} The class $\mathfrak{S}_n$ for $n\in \N$ is generated by rotations by an angle of $2\pi/n$ around a fixed axis. The corresponding Molien series is given by
	\begin{align}
		h_{\mathfrak{S}_n}(s)= \frac{1}{1-s}\frac{1+s^n}{1-s^n}=\left(1+s+s^2+s^3+\dots\right)\left(1 +2s^{n}+2s^{2n}+2s^{3n}+\dots\right).
	\end{align}
	This formula shows that no invariant subspace is ensured to be trivial in this case.
	\item \emph{Dihedral groups.} In three dimensions, the dihedral group $\mathfrak{D}_n$ is generated by two rotations: A rotation by an angle of $2\pi/n$ around a fixed axis and a rotation by an angle of $\pi$ around an axis perpendicular to the first one. The corresponding Molien series is given by
	\begin{align}
		h_{\mathfrak{D}_n}(s) = \frac{1}{1-s^2}\frac{1+s^{n+1}}{1-s^n}= \left(1+s^{n+1}\right)\left(1+s^2+s^4+\dots\right)\left(1+s^n+s^{2n}+\dots\right).
	\end{align}
	In this case, the invariant subspace $H_l^E$ becomes trivial if and only if $l \neq (n+1)k+nm_1+2m_2$ for all $k\in \left\{0,1\right\}$ and $m_1,m_2\in\N_0$.
	\item \emph{Platonic solids.} The last group is given by the three rotation groups of the platonic solids. The tetrahedral group $\mathfrak{T}$ (the rotation group of the tetrahedron) has the Molien series
	\begin{align}
		h_{\mathfrak{T}}(s) = \frac{1}{1-s^4}\frac{1+s^{6}}{1-s^3}=\left(1+s^6\right)\left(1+s^4+s^8+\dots\right)\left(1+s^3+s^6+\dots\right).
	\end{align}
	In this case, the invariant subspace $H_l^E$ becomes trivial if and only if $l \neq 4m_1+3m_2$ for all $m_1,m_2\in \N_0$.
	
	The octahedral group $\mathfrak{O}$ (the rotation group of the cube or the octahedron) has the Molien series
	\begin{align}
	h_{\mathfrak{O}}(s) = \frac{1}{1-s^4}\frac{1+s^{9}}{1-s^6}=\left(1+s^{9}\right)\left(1+s^4+s^8+\dots\right)\left(1+s^6+s^{12}+\dots\right).
	\end{align}
	In this case, the invariant subspace $H_l^E$ becomes trivial if and only if $l \neq 9k+4m_1+6m_2$ for all $k\in \left\{0,1\right\}$ and $m_1,m_2\in\N_0$.
	
	The isocahedral group $\mathfrak{I}$ (the rotation group of the cube or the dodecahedron or the isocahedron) has the Molien series
	\begin{align}
	h_{\mathfrak{I}}(s) = \frac{1}{1-s^{10}}\frac{1+s^{15}}{1-s^6}=\left(1+t^{15}\right)\left(1+s^{10}+s^{20}+\dots\right)\left(1+s^6+s^{12}+\dots\right).
	\end{align}
	In this case, the invariant subspace $H_l^E$ becomes trivial if and only if $l \neq 15k+10m_1+6m_2$ for all $k\in \left\{0,1\right\}$ and $m_1,m_2\in\N_0$.
\end{itemize}
Regarding the four dimensional case $O(4)$, extensive results can be found in \cite{Mendes1975}.
Besides, various results regarding the Molien series in general dimensions are available, see for example \cite{Humphreys1990}.

\begin{remark}
	We remark that some of the references given above do not work in exactly the same setting that we consider here.
	In fact, it is not necessary to decompose the space $L^2\left(\S^{N-1}\right)$ into eigenspaces of $\Delta_{\S^{N-1}}$. Instead, one could also decompose it into spaces of homogeneous polynomials of fixed degree, 
	\begin{align}
		L^2\left(\S^{N-1}\right) = \bigoplus \limits_{l\in \N_0}P_l,
	\end{align}
	where $P_l$ is the space of homogeneous polynomials of degree $l$. Given a finite subgroup $E$ of $O(N)$, we similarly denote by $P_l^E$ the subspace of $P_l$ consisting of all functions that are invariant under the action of all elements of $E$. Let
	\begin{align}
		p_E(s) = \sum_{l=0}^\infty \dim \left(P_l^E\right)s^l
	\end{align}
	be the corresponding generating function. In this situation Molien's formula has to be adapted, namely
	\begin{align}
		p_E(s)= \frac{1}{|E|}\sum\limits_{g\in E}\frac{1}{\det\left(I-sg\right)},
	\end{align}
	see for example \cite[p.~13]{Mukai2003}.
	We obtain $h_E(s)=(1-s^2)p_E(s)$, what enables us to transfer results to the given setting.
\end{remark}

\section{New variables and main results}\label{newvariables}

As announced earlier, one of the main analytical challenges in deriving fine large time asymptotics for the (confined) thin film equation is the free moving boundary. Following Koch \cite{KochHabilitation}, we perform a von Mises-type change of dependent and independent variables, which brings the equation into a setting in which solutions depend differentiably on the initial datum \cite{SeisTFE}. The transformation applies when the solution is Lipschitz close to the stationary solution in the sense of \eqref{o1}, cf.~\cite{Seis15,SeisTFE}. The underlying geometric procedure is the following, which is also illustrated in Figure \ref{fig2}: 
\begin{figure}[t]
\begin{center}
\includegraphics[width=0.9\textwidth]{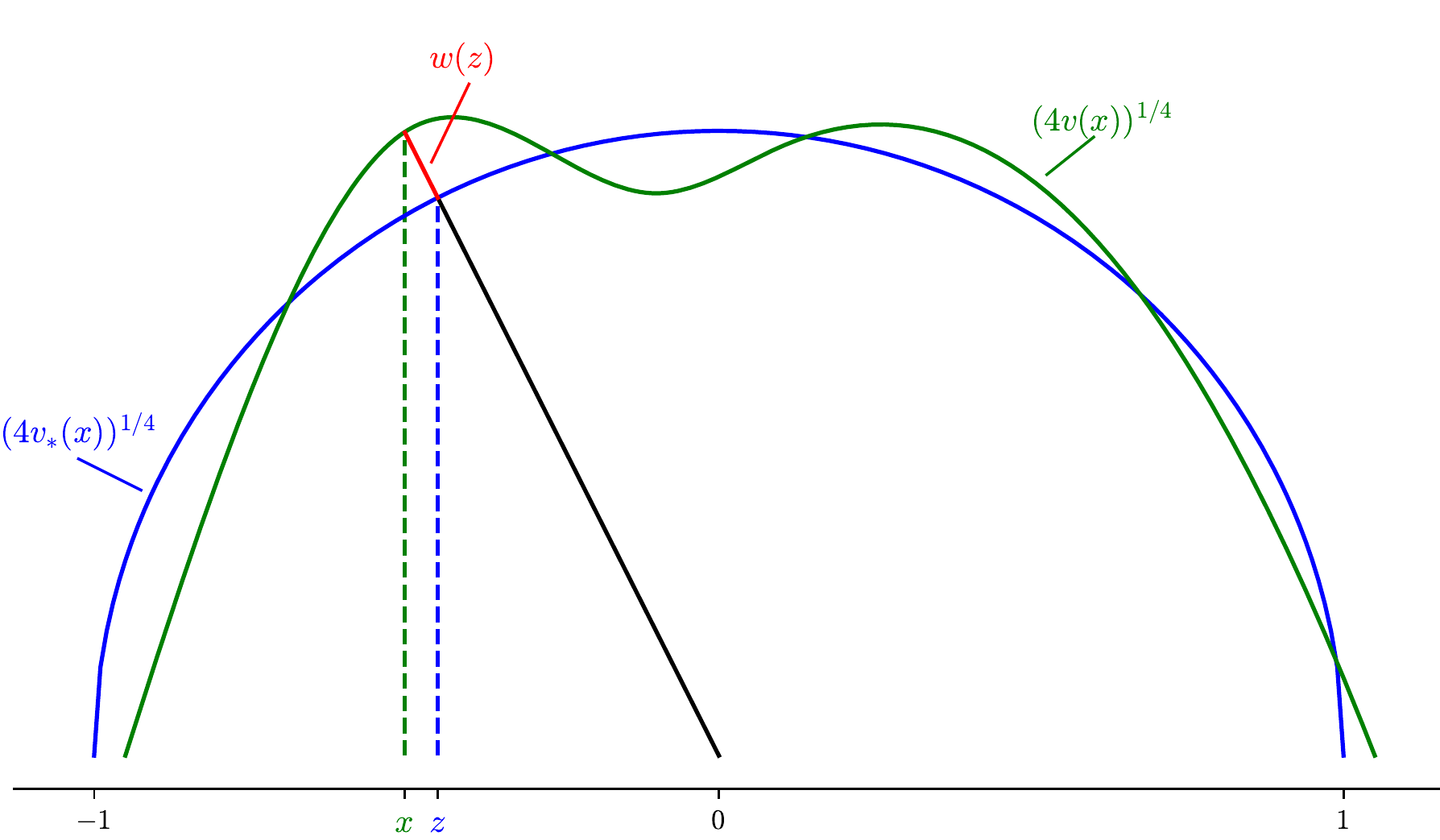}
\end{center}\caption{The change of variables from $(x,v(x))$ to $(z,w(z))$.}\label{fig2}
\end{figure}
The stationary $(4 v_*)^{1/4}$ describes a hemisphere over the  $N$-dimensional unit ball $B=B_1(0)$. We orthogonally project each point $(x,(4 v(x))^{1/4})$ of the graph of $(4 v)^{1/4}$ onto the closest point $(z,(4 v_*(z))^{1/4})$ on the hemisphere and denote by $w(z)$ the (minimal) distance. Analytically this amounts to the choice 
\begin{align}\label{transformationz}
	z=\frac{x}{\sqrt{2( v(t,x))^{1/2}+|x|^2}}
\end{align}
for the new independent variable, and we see that $x=z$ precisely if $v$ is the stationary solution \eqref{123}. The formula for the dependent variables reads
\begin{align}\label{transformationw}
	1+w(t,z)= \sqrt{2( v(t,x))^{1/2}+|x|^2},
\end{align}
and thus $w$ vanishes if $v$ is $v_*$. We will accordingly refer to $w$ as the \emph{perturbation}.

The transformation is applicable also in situations in which $v$ and $v_*$ have not the same mass. This observation is reflected by the fact that  $\mu_{0,0}=0$ occurs in the spectrum of the linear operator, see Theorem \ref{spektrum}. We will not eliminate this eigenvalue on the level of the perturbation, but only for the original variables through the mass constraint \eqref{120}. For the general theory that we perform in terms of the perturbation, any constant solution $w\equiv\const$ is admissible and corresponds to a Smyth--Hill solution \eqref{smythhillsolution} of arbitrary mass $M$. 

The derivation of an evolution equation for the new variable $w$ is lengthy and tedious. It has been described in detail already in \cite{SeisTFE}, using the sloppy $\star$ notation, see \eqref{nonlinearityperturbationequation} below. For our purposes it is necessary to rederive the transformed equation in a way that carries more structure than the formulation chosen in \cite{SeisTFE}. We postpone these computations to the appendix and state here our findings only.
The \emph{perturbation equation} for the $w$ variables is
\begin{align}\label{perturbationequation}
	\partial_tw+\mathcal{L}^2w+N\mathcal{L}w = \frac{1}{\rho} \nabla \cdot \left(\rho^2F[w]\right)+\rho F[w ] \quad \text{ on } \left(0,\infty \right)\times B_1(0),
\end{align}
where $\rho(z) = \frac12(1-|z|^2)$ is a weight function degenerating at the boundary,  $\mathcal{L}w=-\rho^{-1}\grad\cdot\left( \rho^2\nabla w\right)= -\rho \Delta w +2z\cdot \nabla w$ is the building block of the thin film linear  operator  and
\begin{align}\label{nonlinearityperturbationequation}
F[w] = p\star R[w]\star\left(\rho\nabla^3w\star\nabla w+ \rho(\nabla^2w)^{2\star} + \nabla^2w\star\nabla w+ (\nabla w)^{2\star}\right),
\end{align}
is the nonlinearity. The star product $a\star b$ denotes an \emph{arbitrary} linear combination of entries of the tensors $a$ and $b$, and thus, in particular, the above  $F[w]$ defines a class of nonlinearities and  both representatives in \eqref{nonlinearityperturbationequation} may be different from each other. We write $a^{k\star} = a \star \cdots \star a$, where the $\star$-product has $k$ factors.  Moreover, $p$ is a polynomial tensor in $z$, which might have zero entries. The rational factors $R[w]$ are tensors of the form
\[
R[w] = \frac{(\grad w)^{k\star}}{(1+w+z\cdot \grad w)^l},
\]
for some $k\in \N_0$ and $l\in\N$. Finally, the distributive property respects only the tensor class, e.g.~$p\star(a+b) = p\star a +\tilde p\star b$ with two possibly different polynomial tensors $p$ and $\tilde p$.
This shortened $\star$ notation is suitable in the present work because the exact form of the nonlinearity is not important for our analysis. We finally recall from our introduction that the linear operator $\L$ also occurs in the context of the porous medium equation \eqref{118} with $m=\frac{3}{2}$, and was analyzed, for instance, in \cite{Seis14,Seis15}. It is readily checked that $\L$ is symmetric (and, in fact, self-adjoint \cite{Seis14}) with respect to the inner product
\begin{equation}
\label{124}
\la w,\tilde w\ra  = \int_{B_1(0)} w\tilde w\, \rho dz,
\end{equation}
which induces a Hilbert space with norm $\|\cdot\|$ in the obvious way.

The perturbation equation \eqref{perturbationequation} is well-posed for small Lipschitz initial data $w_0$,
\begin{equation}\label{126}
\|w_0\|_{W^{1,\infty}} \ll1,
\end{equation}
as was proved in \cite{SeisTFE}. We will recall the precise statement in Theorem \ref{maintheoremseis} below. The above smallness condition is equivalent to \eqref{o1} under the change of variables.

It follows from the statement of Theorem \ref{spektrum} that the order of the eigenvalues $\mu_{l,k}$ depends on the space dimension $N$. For us, it only plays a role when we want to determine conditions on the initial datum $v_0$ that lead to  improvements in the convergence rates for the confined thin film equation, see Corollaries \ref{exactleadingorder}, \ref{firstordercorrection}, and \ref{grouptheorycorrections} presented above. On the level of the perturbation equation, it is more convenient to    rename the eigenvalues $\left\{\mu_k\right\}_{k\in \mathbb{N}_0}$  and order them in a strictly increasing way, that is $\mu_k<\mu_{k+1}$. Correspondingly, we denote by $\psi_{k,n}$ all eigenfunctions corresponding to $\mu_k$ for $n\in \left\{1,\dots,\tilde{N}_k\right\}$. We note that the multiplicity of $\mu_k$ may change due to intersections between the eigenvalues, see Figure \ref{fig1}. We mostly stick to this notation for the  remaining work. 

All announced asymptotic results for solutions $v$ to the confined thin film equation will be derived from the following theorem that fully describes the higher order asymptotics of the perturbation equation. It is one of the two main results of the present work. Its proof can be found in Section \ref{applicationinvariantmanifolds}.

\begin{theorem}\label{Whoeheremoden}
	For any fixed $K\in \mathbb{N}_0$, there exists an $\varepsilon_0>0$ with the following properties: Let $w$ be a solution to \eqref{perturbationequation} with initial datum $w_0$ satisfying $\|w_0\|_{W^{1,\infty}}\leq \varepsilon_0$. Then, under the assumption  
	\begin{align}\label{h2}
	\lim \limits_{t\rightarrow \infty}e^{\mu_kt}\langle \psi_{k,n},w(t)\rangle  = 0 \quad \text{ for all } k\in\left\{0,\dots,K\right\} \text{ and } n\in\left\{1,\dots,\tilde{N}_k\right\},
	\end{align}
it holds that
	\begin{align}\label{403}
	\left\|w(t)\right\|_{W^{1,\infty} }\lesssim e^{-\mu_{K+1}t} \text{ for all } t\geq 0.
	\end{align}
\end{theorem}

To clarify the meaning of this Theorem, we first consider the case $K=0$. The smallest eigenvalue $\mu_K= \mu_0= 0$, corresponds to the constant eigenfunction $1$, and thus, condition \eqref{h2} turns into the requirement
\begin{align}\label{h1}
\lim\limits_{t\rightarrow \infty} \int_{B_1(0)} w(t,z)\rho(z)dz=0.
\end{align} 
As we will see in the proof of Corollary \ref{exactleadingorder}, the latter is equivalent to the  mass constraint  \eqref{120} for the $v$ variable. By imposing a condition of the solution's mass, we rule out $\mu_0=0$ as a relevant  eigenvalue  for the evolution, or, in other words, the corresponding mode is inactive. It follows that the leading order asymptotics are dominated by the next eigenvalue in order, $\mu_1$, in the sense that it determines the rate of convergence and governs the evolution towards the stationary $v_*$.

The theorem states that this procedure can be iterated. Because the mappings $\la \psi_{k,n},\cdot\ra$ act as projections onto the respective eigenspaces, condition \eqref{h2} ensures that the first $K$ modes (with their multiplicities) are inactive during the evolution, that is,   the modes  do not affect the long-time behavior anymore.  We can thus improve the rate of convergence  and the theorem shows that the leading order asymptotics is then    governed by the smallest active mode. 
In the proofs of Corollaries \ref{firstordercorrection} and \ref{grouptheorycorrections} we identify symmetry conditions for solutions to the thin film equation which ensure the decay \eqref{h2} for the perturbation equation.

The proof for the higher-order asymptotics of the perturbation variable $w$ in Theorem \ref{Whoeheremoden} is based on the construction of invariant manifolds, which are localized around the stationary solution $w\equiv 0$. This is our second main result, which is of independent interest. To state it properly, we have to introduce some further notation. 

First, we denote by $S^t(g)$ the flow generated by the perturbation equation, that is $S^t(g)=w(t,\cdot)$ where $w(t,z)$ solves the perturbation equation with initial datum $g$. 
We consider the Hilbert space $H$ that is induced by the inner product
\begin{align}
\langle v,w\rangle_H = \langle v,w\rangle + \langle \mathcal{L}v,w\rangle =\langle v,w\rangle+  \langle v,\mathcal{L}w\rangle = \langle v,w\rangle +\langle \sqrt{\rho}\nabla v,\sqrt{\rho} \nabla w \rangle
\end{align}
and  the norm
\begin{align}
\left\|w\right\|^2_H = \left\|w\right\|^2+ \|\mathcal{L}^{1/2}w\|^2= \left\|w\right\|^2+\|\sqrt{\rho}\nabla w\|^2,
\end{align}
where $\|\cdot\|$ was defined via \eqref{124}.  It is equivalent to a scale invariant Hilbert space norm,
\begin{equation}
\label{103}
\|w\|_H^2 \sim \left\|w\right\|^2_{L^2}+\left\|\rho\nabla w\right\|^2_{L^2},
\end{equation}
as can be seen with the help of Hardy's inequality, cf.~Lemma \ref{hardytypeinequality} in the appendix. Furthermore, $E_c$ is the eigenspace  spanned by the eigenfunctions $\psi_{k,n}$ for $k\leq K$ and $n\in \left\{1,\dots,\tilde{N}_k\right\}$ with $K\in\N_0$ fixed and $E_s$ denotes its orthogonal complement in $H$, such that $H=E_c \oplus E_s$. In the following theorem, $E_c$ and $E_s$ are the center and stable eigenspaces, respectively. 
We finally have to refine the analysis from \cite{SeisTFE} by considering
\begin{equation}
\label{125}
\|w\|_{W}   = \|w\|_{L^{\infty}} + \|\grad w\|_{L^{\infty}}  + \|\rho \grad^2w \|_{L^{\infty}} + \|\rho^2\grad^3 w\|_{L^{\infty}},
\end{equation}
instead of the Lipschitz norm only. The necessity of considering (scale-invariant) higher-order norms is a crucial observation in our definition and analysis of the truncated equation \eqref{c3}. We will comment on it further in Section \ref{S5}.

\begin{theorem}\label{localmanifolds}
	For any fixed $K\in\N_0$ and $\mu \in \left( \mu_K,\mu_{K+1}\right)$, there exist  two constants $\varepsilon >\varepsilon_0>0$ (with $\varepsilon_0$ possibly smaller than in Theorem \ref{Whoeheremoden}),  and a  Lipschitz continuous mapping $\theta_\varepsilon:E_c\rightarrow E_s$ that is differentiable at zero with $\theta_\varepsilon(0)=0$ and $D\theta_\varepsilon(0)=0$ such that $W_{loc}^c $ given by
	\[
	W_{loc}^c = \left\{g\in H: g=g_c+\theta_\varepsilon\left(g_c\right), g_c\in E_c, \|g\|_{H}\leq \varepsilon \right\}\]
has the following properties:
	\begin{enumerate}
		\item For every $g\in W_{loc}^c$ with $\|g\|_{H}\leq \varepsilon_0$ it holds that $S^t(g)\in W_{loc}^c$ for all $t\geq 0$.
		\item For every $g\in H$ with $\|g\|_{W}\leq \varepsilon_0$ there exists a unique $\tilde{g} \in W_{loc}^c$  such that
		\begin{align}
		\left\|S^t\left(g\right)-S^t\left(\tilde{g}\right)\right\|_{W} \lesssim e^{-\mu t}
		\end{align}
		for every $t\geq 1$. 
	\end{enumerate}
\end{theorem}

The first property simply states that the \emph{local center manifold} $W^c_{loc}$ is locally invariant under the nonlinear evolution \eqref{perturbationequation}. From the properties of $\theta_{\eps}$ we infer that this manifold touches the center eigenspace $E_c$ tangentially at the origin. The second property provides a finite-dimensional approximation at a given rate by solutions in $W_{loc}^c$ for any given solution with sufficiently small initial datum. It is this feature that we exploit in order to derive fine large time asymptotics for the thin film equation.

The invariant manifold theorem is interesting on its own as it provides a nonlinear finite-dimensional object which solutions approximate at a given rate in the large time limit. In other words, once a rate of convergence is  determined,  any sufficiently small solution belonging to an infinite-dimensional function space can be approximated with the prescribed rate by a solution on  a finite-dimensional manifold. As outlined in the introduction, similar results have been derived earlier. What is particularly challenging here is the delicate degenerate parabolicity of the fourth-order equation \eqref{perturbationequation} modeling a free boundary problem whose mathematical understanding is still poor.

The construction of the invariant manifolds will be done in Section \ref{dynamicalsystemarguments}, and will be carried out for a truncated version of the perturbation equation first. In fact, our analysis provides even more information, that we omit here because they are not relevant for the large time asymptotics. For instance, we will show that the finite-dimensional approximation emerges from foliation of the Hilbert space $H$ over a global invariant manifold.

\section{From invariant manifolds to higher order asymptotics}\label{S4}

The goal in this section is the derivation of  the main results for the thin film equation stated in Corollaries \ref{exactleadingorder}, \ref{firstordercorrection}	and \ref{grouptheorycorrections} from Theorem \ref{Whoeheremoden} on the mode-by-mode asymptotics for the perturbation equation.

We start by noting that the transformations \eqref{transformationz} and \eqref{transformationw} yield that
\begin{equation}
\label{131}
v(x) = \rho(\Phi(x))^2\left(1+w(\Phi(x)\right)^4,
\end{equation}
where $\Phi(x)=z$ is the diffeomorphism introduced in \eqref{transformationz}.
%
%
%

In our proof of the leading order asymptotics, we apply Theorems \ref{Whoeheremoden} and \ref{localmanifolds} with $K=0$.

\begin{proof}[ of Corollary \ref{exactleadingorder}]
In a first step, we have to ensure that the mass constraint \eqref{120} implies the vanishing mean condition \eqref{h1}, which is the $K=0$ version of \eqref{h2}. We start by rewriting \eqref{120} with the help of the change of variables formula \eqref{131} and the expression for the Jacobian determinant \eqref{130} in the appendix, 
\begin{equation}\label{302}
	\begin{aligned}
	 \int_{\R^N} v_*(x)\, dx &=  \int_{\R^N} v(t,x)\, dx\\
	 & =    \int\limits_{B_1(0)} \rho(z)^2 (1+w(t,z))^{N+3}\left(1+w(t,z)+z\cdot\nabla w(t,z)\right)\, dz.
	\end{aligned}
	\end{equation}
The  term on the right-hand side can be simplified via an integration by parts,
	\begin{align}\mel
	\int\limits_{B_1(0)} \rho^2 (1+w)^{N+3}\left(1+w+z\cdot\nabla w\right)dz\\
	&= \int\limits_{B_1(0)} \rho^2 (1+w)^{N+4}dz + \frac{1}{N+4}\int\limits_{B_1(0)} \rho^2 z\cdot \nabla \left(1+w\right)^{N+4}dz &\\&
	=\frac{1}{N+4}\int\limits_{B_1(0)}\rho \left(1+w\right)^{N+4}\left(\left(N+4\right)\rho-N\rho+2|z|^2\right)dz\\
	& = \frac{2}{N+4}\int\limits_{B_1(0)} \left(1+w\right)^{N+4}\rho\, dz,
	\end{align}
	where we have used that $4\rho(z) +2|z|^2=2$ in the last identity. In particular, as $v_*$ is mapped onto $w_*=0$ under the change of variables, the latter identity entails that
	\begin{align}
	 \int_{\R^N} v_*\, dx  = \frac{2}{N+4}\int\limits_{B_1(0)} \rho\,  dz\quad  \left(= \frac{2|B_1(0)|}{(N+2)(N+4)}\right),
	\end{align}
	which can also be verified via an elementary computation.
Hence, we may cancel this term on both sides of \eqref{302} to obtain
	\begin{align}
	\frac{2}{N+4}\int\limits_{B_1(0)}\left(\left(1+w(t,z)\right)^{N+4}-1\right) \rho(z) dz=0 \quad \text{ for all } t\geq 0.
	\end{align}

Now we notice that any solution to the perturbation equation  $w(t)$ converges to leading order to a constant $a\in\R$. Indeed, if $K=0$, the local center manifold constructed in Theorem \ref{localmanifolds} is simply a ball $B_{\eps}(0)$ in $\R$. (We comment on this simple fact briefly in the proof of Theorem \ref{stabilityw}.) Hence passing to the large time limit, the previous identity translates into $ (1+a)^{N+4}-1=0$, where $|a|\leq \varepsilon$ and thus $a=0$. This proves \eqref{h1}.

Applying now Theorem \ref{Whoeheremoden} gives the decay estimate
	\begin{align}
	\left\|w(t)\right\|_{W^{1,\infty}}\lesssim e^{-\mu_{1,0}t}.
	\end{align}
	Using the transformation formulas \eqref{transformationz} and \eqref{transformationw}, we see that
	\[\label{407}
	w(t,\Phi(x)) + \frac12w(t,\Phi(x))^2 = \sqrt{v(t,x)} - V_*(x),
	\]
for any $x\in \supp v(t)$, and the quadratic term on the left-hand side is of higher order because $w(t)$ is small. Therefore, the decay estimate for $w(t)$ implies the first part of the statement
\[
\|\sqrt{v(t)}-V_*\|_{L^{\infty}(\supp v(t))} \lesssim e^{-\mu_{1,0}t}.
\]
Lastly, we turn to the decay of the first derivatives. With help of \eqref{407} we derive 
$
\partial_i \left(\sqrt{v(x)}-V_*\right) = \left(1+w\right) \nabla w \cdot \partial_i\Phi
$. Recalling the transformation formulas \eqref{transformationz} and \eqref{transformationw}, we compute \begin{align}\partial_i\Phi(x) = \frac{e_i}{1+w} + \Phi(x)\frac{\partial_i\left(\sqrt{v(t)}-V_*\right)}{(1+w)^2}\end{align} and thus obtain
\begin{align}\label{409}
\nabla \left(\sqrt{v(t,x)}-V_*(x)\right) = \frac{1+w(t,\Phi(x))}{1+w(t,\Phi(x))+\Phi(x)\cdot w(t,\Phi(x))}\nabla w(t,\Phi(x))
\end{align}
for all $x\in \supp v(t)$. Having this identity at hand, the decay estimate for $w(t)$ in $W^{1,\infty}$ directly yields the second part of the statement,
\begin{align}
\left\|\nabla \left(\sqrt{v(t)}-V_*\right)\right\|_{L^\infty(\supp v(t))}\lesssim e^{-\mu_{1,0}t}.
\end{align}
\end{proof}

The  proofs of Corollaries  \ref{firstordercorrection}	and \ref{grouptheorycorrections} will build up on the fact that the $L^2$-projections of solutions $v(t)$ of the confined thin film equation onto eigenspaces generated by  certain eigenfunctions $\psi_{l,n,k}$ vanish for all times if they vanish initially. The following lemma illustrates how exactly this condition can be translated to the perturbation equation.

\begin{lemma}\label{4}
	Let $\psi_{l,n,k}$ be an eigenfunction of the linear operator $\mathcal{L}^2+N\mathcal{L}$ as given in Theorem \ref{spektrum}. Then it holds that
	\begin{align}
	\int_{\R^N} \left(v(x)-v_*(x)\right)\psi_{l,n,k}(x)\, dx = 2\int_{B_1(0)} w(z)\psi_{l,n,k}(z)\rho(z)\, dz + \mathcal{O}\left(\|w\|^2_{L^{\infty}}\right), 
	\end{align}
	provided that $w$ small in the sense of \eqref{126}.
	\end{lemma}
	The lemma, 	in particular, entails  that
	\[
	\left|\la w, \psi_{l,n,k}\ra \right|  \lesssim \left\| w\right\|_{L^\infty}^2,
	\] 
	provided that $\int _{\R^N} v \psi_{l,n,k}dx =\int _{\R^N} v_* \psi_{l,n,k}dx$. We will exploit this observation in the sequel.

\begin{proof}
	Theorem \ref{spektrum} shows that every eigenfunction $\psi_{l,n,k}(x)$ is given as a product of a polynomial in $|x|^2$ and a homogeneous harmonic polynomial of degree $l$, that is
	\begin{align}
		\psi_{l,n,k} = \sum \limits_{j=1}^{k} c(l,k,j)|x|^{2j}\psi_{l}(x),
	\end{align}
	where $\psi_{l}$ denotes an arbitrary homogeneous harmonic polynomial of degree $l$ and $c(l,k,j)$ a real-valued coefficient.
	Due to this structure of the eigenfunctions, the problem boils down to proving
	\begin{align}\label{225}
	\int_{\R^N} \left(v(x)-v_*(x)\right)\psi_l(x)|x|^{2j}dx = 2\int_{B_1(0)} w(z)\psi_l(z)|z|^{2j}\rho(z)dz + \mathcal{O}\left(\|w\|^2_{L^{\infty}}\right),
	\end{align}
	for any integer $j\le k$.
	
	To address \eqref{225}, we  first notice that  by our choice of the perturbation variables \eqref{transformationz} and \eqref{transformationw}, it holds that $\psi_l(x) = (1+w(z))^l\psi_l(z)$ and $|x| = (1+w(z))|z|$. Therefore, we find with the help of the transformation identities \eqref{131} and \eqref{130} that
	\begin{align}
	 \int_{\R^N} v\psi_l |x|^{2j}dx &= \int_{B_1(0)} \rho^2(1+w)^{3+l+2j+N} \psi_l(z) |z|^{2j}\left(1+w+z\cdot \nabla w \right)\,dz\\
	&=\int_{B_1(0)} \rho^2(1+w)^{4+l+2j+N}\psi_l(z)|z|^{2j}dz\\
	&\quad  + \int_{B_1(0)} \rho^2(1+w)^{3+l+2j+N} \psi_l(z)|z|^{2j}z\cdot\nabla w\, dz.
	\end{align}
	In the last term on the right-hand side, we integrate by parts and find after a short computation that
	\begin{align*}
 \mel \int_{B_1(0)} \rho^2 (1+w)^{3+l+2j+N} \psi_l(z)|z|^{2j}z\cdot\nabla w\, dz\\
& = -\int_{B_1(0)} \rho^2 (1+w)^{4+l+2j+N} \psi_l(z)|z|^{2j}\, dz \\
&\quad + \frac2{4+l+2j+N}\int_{B_1(0)} \rho (1+w)^{4+l+2j+N} \psi_l(z)|z|^{2j}\, dz
	\end{align*}
		where we have used the identities $z\cdot \grad \psi_l = l\psi_l$, which holds true because $\psi_{l}$ is a homogeneous polynomial of degree $l$, and  $2\rho+|z|^2=1$. It follows that
		\[
		\int_{\R^N} v\psi_l|x|^{2j}\, dx = 	\frac2{4+l+2j+N}\int_{B_1(0)} \rho (1+w)^{4+l+2j+N} \psi_l(z)|z|^{2j}\, dz.
		\]
	Next, we take into account the identity $(1+w)^m=1+mw+\mathcal{O}\left(\|w\|^2_{L^\infty}\right)$, which holds for $m\in \mathbb{N}$ and $\|w\|_{L^{\infty}}$ small by Taylor expansion, and derive
	\begin{align}\mel
	\int_{\R^N} v\psi_l|x|^{2j}dx=2 \int_{B_1(0)}  \rho w \psi_l|z|^{2j}dz+\frac{2}{4+l+2j+N}\int_{B_1(0)} \rho\psi_l|z|^{2j} dz + \mathcal{O}\left(\|w\|_{L^{\infty}}^2\right).
	\end{align}
	
	It remains to show that 
	\[
	\frac{2}{4+l+2j+N}\int_{B_1(0)} \rho \psi_l|z|^{2j}dz = \int_{\R^N} v_*\psi_l|x|^{2j}dx= \frac{1}{4}\int_{B_1(0)} \psi_l\left(1-|x|^2\right)^2|x|^{2j}\rho dx.
	\]
	 In the case $l\geq1$ both terms vanish thanks to the orthogonality of the eigenfunctions with respect to the inner product introduced in \eqref{124}. Indeed, the harmonic polynomial $\psi_l$ can be written as a linear combination of the eigenfunctions $\psi_{l,n,0}$ with $n\in \{1,\dots,N_l\}$, while the radial weights $|z|^{2j}$ and $ \left(1-|x|^2\right)^2|x|^{2j}$ lie  in the spaces $\huelle\left\{\psi_{0,0,i}:\:i\leq j\right\}$ and $\huelle\left\{\psi_{0,0,i}:\:i\leq j+2\right\}$, respectively.
 For $l=0$, it holds that $\psi_0=1$ and the claim follows via an elementary computation. This establishes \eqref{225} and thus the proof is finished.
\end{proof}

With help of the previous lemma, the proof of Corollary \ref{firstordercorrection} reduces to an easy combination of the already established results.

\begin{proof}[ of Corollary \ref{firstordercorrection}]
	As the solution of the confined thin film equation remains centered at the origin provided its initial data is, cf.~\eqref{motioncenterofmass}, we can make use of Lemma \ref{4} with $\psi_{l,n,k} = \psi_{1,n,0}$ to obtain,
	\begin{align}
		0 = \int_{\R^N} x_iv(t,x)dx =2 \int_{B_1(0)} z_i w(t,z)\rho dz +\mathcal{O}\left(\|w\|^2_{L^{\infty}}\right),
	\end{align}
	for every $i \in \left\{1,\dots,N\right\}$.
	In the proof of Corollary \ref{exactleadingorder} we already established convergence rates for $w$, namely $\|w\|_{L^{\infty}}\lesssim e^{-\mu_{1,0}t}$. This directly yields
	\begin{align}
		\lim \limits_{t\rightarrow \infty}e^{\mu_{1,0}t} \intB zw(t,z)\rho dz =0,
	\end{align}
	which makes Theorem \ref{Whoeheremoden} applicable. We therefore obtain
	\begin{align}
	\|w(t)\|_{W^{1,\infty}}\lesssim e^{-\mu t},
	\end{align}
	where $\mu$ is the next eigenvalue in line, which is $\mu = \mu_{0,1} =30$ if $N=1$ and $\mu = \mu_{2,0} = 16+4N$ if $N\ge 2$. 
	
	It remains to translate the convergence result for the perturbation equation into a convergence result for the confined thin film equation. The argument proceeds in  exactly the same way as the proof of Corollary \ref{exactleadingorder}. We drop the details. \end{proof}

The last proof of this section is based on similar ideas and exploits Lemma \ref{4} in more generality. 

\begin{proof}[ of Corollary \ref{grouptheorycorrections}]
	In a first step we establish the uniform decay estimate 
 $\left\|w(t)\right\|_{L^\infty}\lesssim e^{-\mu_{l,k}t}$, which directly implies  $\lim \limits_{t\rightarrow \infty}e^{\mu t}\langle \psi,w(t)\rangle  = 0 $	for all $\mu < \mu_{l,k}$ and their corresponding eigenfunctions $\psi$. Towards this uniform estimate, we notice that on the one hand it holds  $|w(t,z)| \lesssim \left|w(t,z)+\frac{1}{2}w(t,z)^2\right|$, because $w(t)$ is small as a consequence of the leading order asymptotics in Corollary \ref{exactleadingorder}. On the other hand, we deduce from the transformation formulas \eqref{transformationz} and \eqref{transformationw}  that $\left|w(t,z)+\frac{1}{2}w(t,z)^2\right| = \left| \sqrt{v(t,x)}-V_*(x)\right|$. A combination of both and \eqref{401} gives the estimate on $w(t)$.
	
Before we continue with the proof, we insert a short discussion about the assumptions on the decay of $v(t)-\sqrt{V_*}$, c.\ f.\  Remark \ref{R3}. Since all eigenmodes corresponding to eigenvalues $\mu$ smaller than $\mu_{l,k}$ decay fast enough, Theorem \ref{Whoeheremoden} provides a decay estimate for $w(t)$ in $W^{1,\infty}$, namely $\left\|w(t)\right\|_{W^{1,\infty}}\lesssim e^{-\mu_{l,k}t}$. Proceeding in the same way as in the proof of Corollary \ref{exactleadingorder}, we obtain $\left\|v(t)-\sqrt{V_*}\right\|_{W^{1,\infty}\left(\supp v(t)\right)}\lesssim e^{-\mu_{l,k}t}$. This shows that extending norm in the decay assumption in Corollary \ref{grouptheorycorrections} from $L^\infty$ to $W^{1,\infty}$ eventually provides an equivalent condition.
	
	Let us now turn back to the actual proof.
	To deduce a better convergence rate for $w(t)$ from Theorem \ref{Whoeheremoden}, we also have to show that the eigenmodes corresponding to $\mu_{l,k}$ are inactive, that is
	\begin{align}\label{224}
	\lim \limits_{t\rightarrow \infty}e^{\mu_{l,k} t}\langle \psi_{l,n,k},w(t)\rangle  = 0 \quad \text{ for all } n\in \left\{1,\dots,N_l\right\}.
	\end{align}
	Once this is proved, we obtain with help of Theorem \ref{Whoeheremoden} that $\|w\|_W \lesssim e^{-\mu_+ t}$, where $\mu_+$ is the next largest eigenvalue following $\mu_{l,k}$.
	From this point on, the proof proceeds in the same way as before.
	
	Let us now turn to the proof of \eqref{224}. Recalling   that $\left\|w\right\|_{L^\infty}\lesssim e^{-\mu_{l,k}t}$ and Lemma \ref{4}, it suffices to prove
	\begin{align}\label{226}
		\int_{\R^N} \left(v(t,x)-v_*(x)\right)\psi_{l,n,k}(x)\, dx =0 \quad \text{ for all }n\in \left\{1,\dots,N_l\right\},
	\end{align}
	for all $t\geq 0$. The argument for this identity is based on the invariance of $v(t)$ under orthogonal transformations contained in $E$. Since the confined thin film equation is invariant under orthogonal transformations, uniqueness of solutions to this equation guarantees that the solution $v(t)$ inherits this property from its initial datum $v_0$ for every time $t$.
	
	By the right choice of $E$, this geometric invariance ensures that the projection of $v(t)$ onto every homogeneous, harmonic polynomial of degree $l$ vanishes. The same trivially holds true for $v_*$.
	In order to exploit this fact, we have a closer look at the structure of the eigenfunctions $\psi_{l,n,k}$ appearing in \eqref{226}. Due to the condition that $\mu_{l,k}$ has multiplicity $N_l$, we know from Theorem \ref{spektrum} that every $\psi_{l,n,k}$ has the form
	\begin{align}
		\psi_{l,n,k} = \sum \limits_{j=1}^{k} c(l,k,j)|x|^{2j}\psi_{l}(x),
	\end{align}
	where $\psi_l$ denotes an homogeneous harmonic polynomial of degree $l$.
	
	Note that the product $v(t) \sum c(l,k,j)|x|^{2j}$ satisfies the same geometrical properties as $v(t)$ and thus its projection onto every homogeneous harmonic polynomial vanishes as well, i.\ e.\,
	\begin{align}
	0 = \int_{\R^N} v(t,x)\sum \limits_{j=1}^{k} c(l,k,j)|x|^{2j}\psi_{l}(x)\,dx = \int_{\R^N}v(t,x)\psi_{l,n,k} \,dx.
	\end{align}
	Again, the same holds true for $v_*$ and thus the proof of \eqref{226} is completed.
\end{proof}

\begin{remark}\label{R2}
	In the two-dimensional case $N=2$, Corollary \ref{grouptheorycorrections} can also be easily proved in a more direct way thanks to the fact that both, the spherical harmonics and the orthogonal transformations have a handy, explicit form in two dimensions. The spherical harmonics of degree $l$ are given by (in polar coordinates) $cos(l\varphi)$ and $sin(l\varphi)$. Recalling the form of a rotation or reflection matrix, a straightforward computation yields the same results as Corollary \ref{grouptheorycorrections}.
	
	However, this strategy becomes impracticable in higher dimensions, particularly because there is no longer such a convenient representation for general orthogonal projections. 
\end{remark}

\section{Theory for the perturbation equation}\label{theoryforperturbationequation}

In this section, we will recall  main aspects of the theory for the perturbation equation \eqref{perturbationequation} derived earlier in \cite{SeisTFE}, and we will provide higher order regularity estimates. Such estimates will be an important tool in our invariant manifold theory, which we will develop in the subsequent sections.

%

We start by recalling that the operator $\mathcal{L}$ is symmetric in $L^2(\rho)$ and satisfies the maximal regularity estimate
\begin{align}\label{b2}
\|\nabla w\|+\|\rho\nabla^2w\| \lesssim \|\mathcal{L}w\|.
\end{align} 
Indeed, such an estimate holds true for the more general class of degenerate elliptic operators   
\begin{equation}\label{504}
\mathcal{L}_\sigma \coloneqq -\rho^{-\sigma}\nabla \cdot \left(\rho^{\sigma+1}\nabla w\right),
\end{equation}
that naturally occur in the context of the  porous medium equation, see \cite{KochHabilitation,Kienzler16,Seis14,Seis15}. In this case, the underlying Hilbert space is $L^2\left(\rho^\sigma\right)$. 
We state the corresponding maximal regularity  estimate for the  fourth order linear problem associated to the perturbation equation \eqref{perturbationequation}, that is,
\begin{align}\label{b1}
\begin{cases}
\partial_tw+\mathcal{L}^2w+ N \mathcal{L}w &= f \quad \text{ in } (0,\infty)\times B_1(0)\\
w(0,\cdot)&=w_0 \quad \text{ in } B_1(0).
\end{cases}
\end{align}
This problem  is well-posed for $L^2(\rho)$ initial data and $L^2((0,T);L^2(\rho))$ inhomogeneities, see Lemma 7 in \cite{SeisTFE}. In the case with zero initial data, $w_0=0$, there is the maximal regularity estimate
\begin{equation}\label{112}
\begin{aligned}
\mel	\|\partial_tw\|_{L^p\left(\left(0,T\right);L^p(\rho^\sigma)\right)}+\|\nabla^2w\|_{L^p\left((0,T);L^p(\rho^\sigma)\right)}+\|\rho\nabla^3w\|_{L^p\left((0,T);L^p(\rho^\sigma)\right)}+\|\rho^2\nabla^4w\|_{L^p\left((0,T);L^p(\rho^\sigma)\right)} \\
&\lesssim \|f\|_{L^p\left((0,T);L^p(\rho^\sigma)\right)},
\end{aligned} 
\end{equation}
which holds true for any $p\in \left(1,\infty\right)$, $\sigma >0$ and $T>0$, see Lemma 8 and Proposition 19 (and its proof) in \cite{SeisTFE}. 

In order to motivate the results that are collected and derived in the following, we  have a closer look at the nonlinearity occurring in \eqref{nonlinearityperturbationequation}. The natural framework to prove well-posedness of the nonlinear problem \eqref{perturbationequation} is the class $C^{0,1}(B_1(0))$, in which the  singular terms $R_l[w]$ can be suitably controlled, at least, if $w$ is small in that class. Moreover, in such a situation, the nonlinearity is of the same regularity order as the linear elliptic operator $\L^2 $, and the inhomogeneity can thus be treated as a quadratic perturbation term. We will carry this out in a simple Hilbert space setting later in Section \ref{S5} (after a necessary truncation). A complete theory for the nonlinear equation \eqref{perturbationequation} forces us to construct  higher order norms that match the scaling of the (homogeneous) Lipschitz norm. This naturally leads to considering Carleson or Whitney measures, more precisely
\begin{align}
\|w\|_{X(p)}&=\sum\limits_{(l,k,|\beta|)\in \mathcal{E}} \sup\limits_{\substack{z\in\overline{B_1(0)}\\0<r\leq1}} \frac{r^{4k+|\beta|-1}}{\theta(r,z)^{2l-|\beta|+1}}\left|Q_r^d(z)\right|^{-\frac{1}{p}}\|\rho^l\partial_t^k\partial_z^\beta w\|_{L^p\left({Q_r^d(z)}\right)}\\
&\quad +\sum\limits_{(l,k,|\beta|)\in \mathcal{E}} \sup\limits_{T\geq1}
\|\rho^l\partial_t^k\partial_z^\beta w\|_{L^p\left(Q(T)\right)},\\
\|f\|_{Y(p)}&= \sup\limits_{\substack{z\in\overline{B_1(0)}\\0<r\leq1}} \frac{r^3}{\theta(r,z)}\left|Q_r^d(z)\right|^{-\frac{1}{p}}\|f\|_{L^p\left(Q_r^d(z)\right)} +\sup\limits_{T\geq1}  \|f\|_{L^p\left(Q(T)\right)},
\end{align}
where $
\mathcal{E}=\left\{ (0,1,0),(0,0,2),(1,0,3),(2,0,4) \right\}$ and $\theta(r,z) = \max\{r,\sqrt{\rho(z)}\}$. Moreover, $Q_r^d(z)$ is the Whitney cube $(r^4/2,r^4)\times B_r^d(z)$ and $Q(T) = (T,T+1)\times B_1(0)$. We remark that the balls $B_r^d(z)=\left\{z'\in \overline{B_1(0)}: d(z,z')<r\right\}$ are not defined with respect to the Euclidean metric on $B_1(0)$ but the semi-distance 
\begin{align}\label{500}
	d(z,z')\coloneqq \frac{|z-z'|}{\sqrt{\rho(z)}+\sqrt{\rho(z')}+\sqrt{|z-z'|}}.
\end{align}

The occurrence of this semi-distance can be motivated by interpreting the parabolic problem \eqref{b1} as a (fourth order) heat flow on a weighted Riemannian manifold $(\mathcal{M},\textbf{g},\omega \textbf{vol})$, cf.~\cite{Grigoryan06}. Indeed, considering
$\textbf{g}=\rho^{-1}(dx)^2$ as the Riemannian metric on the disc $B$ and choosing a suitable weight $\omega$ on the volume form, the elliptic operator $\L$ turns out to be the Laplace--Beltrami operator on $(\mathcal{M},\textbf{g},\omega \textbf{vol})$. On this manifold, the  induced geodesic distance is equivalent to $d(z,z')$ in \eqref{500}.

 Considering this \emph{intrinsic metric} is helpful as the theories for heat flows are often also available on weighted manifolds \cite{Grigoryan06}.
For the subsequent computations, we recall some properties of the intrinsic distance from \cite{Seis15}: The intrinsic balls are equivalent to  Euclidean balls, more precisely there exists a positive constant $C$ such that 
\begin{align}\label{equivalenceballs}
	B_{C^{-1}r\theta(r,z)}(z)\subseteq B_r^d(z)\subseteq B_{Cr\theta(r,z)}(z)
\end{align}
for every $z$ in $\overline{B_1(0)}$ and any $r$. Furthermore, it holds for any $r$ that
\begin{align}
	\sqrt{\rho(z')}\lesssim r \quad \Rightarrow\quad \sqrt{\rho(z)}\lesssim r \quad \text{ for all } z\in B_r^d(z') 
\end{align}
and 
\begin{align}
	\sqrt{\rho(z')}\gg r \quad \Rightarrow \quad \rho(z)\sim \rho(z') \quad \text{ for all } z\in B_r^d(z')
\end{align}
which, in particular, implies that 
\begin{equation}\label{502}
\theta(r,\cdot)\sim \theta(r,z')\quad\mbox{in }B_r^d(z').
\end{equation}

 Variants of these norms were considered earlier in the treatment of the Navier--Stokes equations, a class of geometric flows, the porous medium equation and the thin film equation \cite{KOCH200122,Koch2012,Kienzler16,John15,Seis15}, see also the review in \cite{KochLamm15}. The choice of the large time contributions is rather arbitrary, see also Remark \ref{R1}.

Still on the level of the linear equation \eqref{b1}, it is proved in \cite{SeisTFE} that for any $p>N+4$, the solution to \eqref{b1} satisfies the estimate
\begin{align}\label{113}
\|w\|_{W^{1,\infty}} +\|w\|_{X(p)}\lesssim \|f\|_{Y(p)} + \|w_0\|_{W^{1\infty}},
\end{align}
provided that the right-hand side is finite. The well-posedness theory for the perturbation equation \eqref{perturbationequation} and our   higher-order regularity estimate below do heavily rely on that bound.

For further reference, we recall the main results for \eqref{perturbationequation} from the literature. 
\begin{theorem}[\cite{SeisTFE}]\label{maintheoremseis}
	Let $p>N+4$ be given. There exists $\varepsilon_0 >0$ such that for every $w_0\in W^{1,\infty}$ with $
	\|w_0\|_{W^{1,\infty}} \leq \varepsilon_0$ 
	there exists a solution $w$ to the nonlinear equation \eqref{perturbationequation} with initial datum $w_0$ and $w$ is unique among all solutions with  $\|w\|_{L^\infty(W^{1,\infty})} + \|w\|_{X(p)} \lesssim \varepsilon_0$. Moreover, this solution $w$ satisfies the estimate
	\begin{align}\label{estimateagainstinitialdata}
	\|w\|_{L^\infty(W^{1,\infty})} + \|w\|_{X(p)}\lesssim \|w_0\|_{W^{1,\infty}}
	\end{align}
	and	is smooth, and analytic in time and angular direction.
\end{theorem}

Strictly speaking, the result described here slightly differ from \cite{SeisTFE}.

\begin{remark}\label{R1}
	For accuracy, we remark that in \cite{SeisTFE}, the linear bound \eqref{113} and the nonlinear theory in  	  Theorem \ref{maintheoremseis} were derived  for  slightly different $X(p)$ and $Y(p)$ norms. Indeed, in this earlier work the large time contributions $
	\|\rho^l\partial_t^k\partial_z^\beta w\|_{L^p\left(Q(T)\right)}$ and $   \|f\|_{L^p\left(Q(T)\right)}$ came both with a factor $T$. With regard to the theory developed in the present paper, dropping this factor is more convenient.
\end{remark}

In the present paper, we have to extend the theory   from $C^{0,1}$ data to a higher regularity setting. Indeed, it turns out that the truncation that we introduce on the level of the nonlinearity in Section \ref{S5} needs to cut-off derivatives up to third order. In order to subsequently relate the truncated equation  to the original one \eqref{perturbationequation}, these derivatives need to be controlled by the initial data. 
We will chose the uniform  higher-order norms whose homogeneous parts have the same scaling as the homogeneous Lipschitz norm at the boundary, $\|\cdot\|_W$, which we introduced in \eqref{125}.

Our main contribution in the present section is the following higher order regularity result.

\begin{theorem}\label{dritteableitung}
	There exists $\varepsilon_0>0$, possibly smaller than in Theorem \ref{maintheoremseis}, such that for every $w_0\in W^{1,\infty}$ with
	$
	\|w_0\|_{W}\leq \varepsilon_0$, 	the unique solution $w$ from from Theorem \ref{maintheoremseis}  satisfies
	\begin{align}
	\|w\|_W \lesssim \|w_0\|_W.
	\end{align} 
\end{theorem}

\begin{proof} \emph{Step 1. Second order derivatives.} We will prove the slightly stronger bound
	\begin{align}\label{114}
	\|\rho \grad^2 w\|_{L^{\infty}} + \|\rho \nabla w\|_{X(p)} \lesssim \|w_0\|_{W^{1,\infty}}+\|\rho \nabla^2w_0\|_{L^\infty}.
	\end{align}
	For this purpose, for every $i=1,\dots,N$, we consider the dynamics of $\rho \partial_i w$ under the nonlinear equation \eqref{perturbationequation}, that is,
	\[
	\partial_t(\rho \partial_iw)+\mathcal{L}^2(\rho \partial_iw)+N\mathcal{L}(\rho \partial_iw)=\rho \partial_if[w] + NE[w]+\mathcal{L}E[w] + E[\mathcal{L}w], 
	\]
	where $E[v] = -\rho z_i \Delta v - 2\rho \partial_iv
	+ (N\rho -2|z|)\partial_iv+ 2\rho z \cdot \nabla\partial_iv$ is the commutator of the operators $\rho\partial_i $ and $ \L$, and this equation is equipped with the initial datum $\rho \partial_iw_0.$
	From the a priori bound in \eqref{113}, we know that
	\[
	\|\rho \partial_iw\|_{W^{1,\infty}} +\|\rho \partial_iw\|_{X(p)}\lesssim \|\rho \partial_if[w]\|_{Y(p)} + \|NE[w]+\mathcal{L}E[w] + E[\mathcal{L}w]\|_{Y(p)}+ \|\rho \partial_iw_0\|_{W^{1,\infty}}.
	\]
	In view of the bound from Theorem \ref{maintheoremseis}, in order to prove \eqref{114} it suffices thus  to prove that
	\begin{equation}\label{a3}
	\begin{aligned}
	\mel	\|\rho \partial_if[w]\|_{Y(p)} + \|NE[w]+\mathcal{L}E[w] + E[\mathcal{L}w]\|_{Y(p)}\\
	& \lesssim \|w\|_{W^{1,\infty}} + \|w\|_{X(p)} + \varepsilon_0 \left( \|\rho \nabla w\|_{W^{1,\infty}} + \|\rho \nabla  w\|_{X(p)} \right),
	\end{aligned}
	\end{equation}
	and to choose $\varepsilon_0$ sufficiently small.
	
	From \cite{SeisTFE} we are aware of another form of the nonlinearity $f[w]$ of the perturbation equation \eqref{perturbationequation}, namely $f[w]=f^1[w]+f^2[w]+f^3[w],$ where
	\begin{align}
	f^1[w] &= p \star R[w]\star\left(\left(\nabla w\right)^{2\star}+\nabla w \star \nabla^2w \right),\\
	f^2[w] & = p\star R[w]\star \rho \left(\left(\nabla^2w\right)^{2\star}+ \nabla^3w\star \nabla w\right),\\
	f^3[w]& =p \star R[w]\star \rho^2 \left(\left( \nabla^2w\right)^{3\star} +\nabla^2w\star \nabla^3w+ \nabla w\star \nabla^4w \right),
	\end{align}
	and 
	\[
	R[w]= \frac{\left(\nabla w\right)^{k\star}}{\left( 1+w +z\cdot \nabla w\right)^{l}}
	\]
	for some $k \in \mathbb{N}_0$, $l\in\N$, whose values may be different in any occurrence of $R[w]$. (Of course, the reader may derive this presentation also directly from \eqref{perturbationequation} and \eqref{nonlinearityperturbationequation}.) The computation of derivatives of these expressions is tedious but straightforward. As an auxiliary result we  notice that $\grad R [w] = p\star R  + p\star R\star \grad^2 w$.  Here are the final formulas:
	\begin{align*}
	\partial_i f^1[w]&  =p\star R[w]\star\left( (\grad w)^{2\star} +\grad w\star \grad^2 w + (\grad^2 w)^{2\star} +\grad w\star\grad^3 w\right),\\
	\partial_i f^2[w]  &=p\star R[w]\star\left( (\grad^2 w)^{2\star} +\grad w\star \grad^3 w + \rho (\grad^2 w)^{3\star} +\rho \grad^2 w\star\grad^3 w + \rho\grad w\star\grad^4w\right),\\
	\partial_i f^3[w] & = p\star R[w]\star\left(\rho(\grad^2 w)^{3\star} + \rho\grad^2 w\star\grad^3 w+\rho\grad w\star\rho^4w  + \rho^2 (\grad^2 w)^{4\star} \right.\\
	&\quad +\left. \rho^2 \grad w\star \grad^5 w+ \rho^2 (\grad^2 w)^{2\star}\star \grad^3w +\rho^2 (\grad^3 w)^{2\star} +\rho^2 \grad^2 w\star\grad^4 w\right).
	\end{align*}
	Combining them, and multiplying by $\rho$, we thus find that
	\[
	\rho\partial_i f[w] = p\star R[w]\star\left( I + J  \right),
	\]
	where 
	\begin{align*}
	I & = (\grad w)^{2\star} +\rho \grad w\star\grad^2 w +\rho (\grad^2 w)^{2\star} +\rho \grad w\star\grad^3 w +\rho^2 \grad^2 w \star\grad^3 w\\
	&\quad + \rho^2 \grad w\star \grad^4 w +\rho^2 \grad^2 w\star\grad^3 w + \rho^3 \grad^2 w\star\grad^4 w +  \rho^3 \grad w\star\grad^5 w,\\
	J & = \rho^3 (\grad^3 w)^{2\star} +\rho^2 (\grad^2 w)^{3\star}  +\rho^3 (\grad^2 w)^{2\star} \star\grad^3 w + \rho^3 (\grad^2 w)^{4\star} .
	\end{align*}
	Because $|p\star R[w]|\lesssim 1$ thanks to the control of $w$ and $\grad w$ during the evolution, in our estimate of $\rho \partial_i f[w]$ it is enough to control $I$ and $J$. Here, the first term is much easier to handle. Indeed, using the fact that $\|\grad w\|_{Y(p)} \lesssim \|\grad w\|_{L^{\infty}}$ and $\|\nabla^2 w\|_{Y(p)} + \|\rho \nabla^3w\|_{Y(p)}+\|\rho^2\nabla^4w\|_{Y(p)} \lesssim \|w\|_{X(p)} $, which comes directly out of the definition of the $Y(p)$ norm, and invoking the a priori estimate in Theorem \ref{maintheoremseis}, we readily find that
	\begin{align*}
	\|I\|_{Y(p)} &\lesssim \left(\|\grad w\|_{L^{\infty}} + \|w\|_{X(p)}\right)\left(\|\grad w\|_{L^{\infty}} + \|\rho \grad^2 w\|_{L^{\infty}} + \|\rho^3 \grad^5 w\|_{Y(p)}\right)\\
	& \lesssim  \|w_0\|_{W^{1,\infty}} + \eps_0 \left( \|\rho \grad^2 w\|_{L^{\infty}} + \|\rho^3 \grad^5 w\|_{Y(p)}\right).
	\end{align*}
	
	The estimates of the terms appearing in $J$ are more involved as we have to make use of suitable interpolations. Some  were already discussed in  \cite{SeisTFE}, but we present the ideas here for the convenience of the reader.  	Let $\eta$ be a smooth cut-off function satisfying $\eta = 1$ in $B_r^d(z_0)$ and $\eta =0$ outside $B_{2r}^d(z_0)$ for $r\leq 1$. Inside of the ball $B_r^d(z_0)$, we then have that
	\[
	\rho^3 |\grad^3 w|^2 \lesssim \rho |\grad \zeta|^2 + \rho |\grad^2 w|^2,
	\]
	if $\xi =\eta \rho \nabla^2w$. It follows that
	\begin{align*}
	\|\rho^3 |\grad^3 w|^2\|_{L^p\left(B_r^d(z_0)\right)}  & \lesssim \|\rho |\grad \zeta|^2\|_{L^p}  + \|\rho |\grad^2 w|^2\|_{L^p\left(B_r^d(z_0)\right)}.
	\end{align*}
	To estimate the first term on the right hand side, we make use of the interpolation inequality \eqref{408} with $m=2$ and $i=1$ in  Lemma \ref{interpolationintequality} of the appendix and find
	\begin{align*}
	\|\rho |\grad\zeta|^2\|_{L^p} & = \|\grad \zeta\|_{L^{2p}\left(\rho^p\right)}^2 \lesssim \|\zeta\|_{L^{\infty}} \|\grad^2\zeta\|_{L^p\left(\rho^p\right)}.
	\end{align*}
	We then deduce from the definition of $\zeta$, by using Leibniz' rule and  the fact that $|\grad^k \eta|\lesssim r^{-k}\theta(r,z_0)^{-k}$, which follows from the behavior of the intrinsic balls in \eqref{equivalenceballs}, that
	\begin{align}
	\|\rho |\grad\zeta|^2\|_{L^p} & \lesssim \|\rho \grad^2 w\|_{L^{\infty}} \left( \|\rho^2 \grad^4 w\|_{L^p\left(B_{2r}^d(z_0)\right)} + \|\rho \grad^3 w\|_{L^p(B_{2r}^d(z_0))} \right.\\
	&\quad  +\frac1{r\theta(r,z_0)}\|\rho \grad^2 w\|_{L^p\left(B_{2r}^d(z_0)\right)}+\frac1{r\theta(r,z_0)}\|\rho^2 \grad^3 w\|_{L^p\left(B_{2r}^d(z_0)\right)}\\
	&\quad \left.+\frac1{r^2\theta(r,z_0)^2}\|\rho^2 \grad^2 w\|_{L^p\left(B_{2r}^d(z_0)\right)}\right).
	\end{align}
	The $\rho$'s can we always pulled out of the norms by estimating against $\theta(r,z_0)^2$, because $\theta(r,z_0)\sim \theta(r,z)=\max\left\{r,\sqrt{\rho(z)}\right\}$ by \eqref{502}. In view of the definitions of the $Y(p)$ and $X(p)$ norms, we then deduce that
	\[
	\|\rho^3 |\grad^3 w|^2 \|_{Y(p)} \lesssim \|w\|_{X(p)} \|\rho \grad^2 w\|_{L^{\infty}} + \|\rho |\grad^2 w|^2\|_{Y(p)},
	\]
	and the second term can be estimated as in our bound for  $I$, so that we find
	\begin{equation}
	\label{117}	
	\|\rho^3 |\grad^3 w|^2 \|_{Y(p)} \lesssim \eps_0 \|\rho \grad^2 w\|_{L^{\infty}} 
	\end{equation}	 
	thanks to the estimates from Theorem \ref{maintheoremseis}
	
	The second term in $J$ can be estimated very similarly. This time we choose $\zeta = \eta \grad w$ and eventually arrive at
	\begin{equation}
	\label{115}
	\|\rho^2 |\grad^2 w|^3\|_{Y(p)} \lesssim \|\grad w\|_{L^{\infty}}^2  \left(\|\grad w\|_{L^{\infty}} + \|w\|_{X(p)}\right) \lesssim \|w_0\|_{W^{1,\infty}},
	\end{equation}
	thanks to the a priori estimates in Theorem \ref{maintheoremseis}. (Notice that details for this estimate can be found in \cite{SeisTFE}.)
	The latter bound also entails an estimate for the fourth term in $J$. Indeed, we have	
	\begin{equation}
	\label{116}
	\|\rho^3 |\grad^2 w|^4\|_{Y(p)} \le \|\rho \grad^2 w\|_{L^{\infty}} \|\rho^2 |\grad^2 w|^3\|_{Y(p)}\lesssim  \|g\|_{W^{1,\infty}}\|\rho \grad^2 w\|_{L^{\infty}} \le \eps_0 \|\rho \grad^2 w\|_{L^{\infty}}.
	\end{equation}
	Finally, in order to bound the third term in $J$, we interpolate between \eqref{117} and \eqref{116}. Altogether, we find the estimate
	\[
	\|J\|_{Y(p)} \lesssim \|w_0\|_{W^{1,\infty}} + \eps_0 \|\rho \grad^2 w\|_{L^{\infty}}.
	\]

	Our estimates on $I$ and $J$ yield the desired control on $\rho \partial_i f[w]$. 	To prove the full statement in \eqref{a3}, it remains only to choose $\eps_0$ small enough and to notice that
	\begin{align}
	|NE[w]+\mathcal{L}E[w] + E[\mathcal{L}w]| \lesssim |\rho^2\nabla^4w|+|\rho \nabla^3w|+|\nabla^2w|+|\nabla w|,
	\end{align}
	which provides 
	\[
	\|NE[w]+\mathcal{L}E[w] + E[\mathcal{L}w]\|_{Y(p)} \lesssim \|w\|_{W^{1,\infty}} + \|w\|_{X(p)} \lesssim \|w_0\|_{W^{1,\infty}}
	\]
	in a similar manner as before. This finishes the proof.

	\medskip
	
	\emph{Step 2. Third order derivatives.} The prove of the estimates proceeds analogously to the first step, only this time, much more terms have to be considered. 
	For every $i,j=1,\dots,N$ we consider the dynamics of $\rho \partial_j(\rho \partial_i w)$, that is,
	\begin{align}\MoveEqLeft
	\partial_t\left(\rho \partial_j(\rho \partial_iw)\right) + \mathcal{L}^2\left(\rho \partial_j(\rho \partial_iw)\right)+N\mathcal{L}\left(\rho \partial_j(\rho \partial_iw)\right) \\
	&=\rho \partial_j\left(\rho \partial_if[w]\right) +\rho \partial_j \left(NE[w]+\mathcal{L}E[w] + E[\mathcal{L}w]\right) \\
	&\quad + NE[\rho \partial_iw]+\mathcal{L}E[\rho \partial_iw] + E[\mathcal{L}(\rho \partial_iw)],
	\end{align}
	which is equipped with the initial datum $\rho \partial_j(\rho \partial_iw_0)$.
	Again, thanks to the a priori bound \eqref{113}, we know that
	\begin{align}\MoveEqLeft
	\|\rho\partial_j\left(\rho \partial_iw\right)\|_{W^{1,\infty}}+ \|\rho\partial_j\left(\rho \partial_iw\right)\|_{X(p)}\\
	&\lesssim \|\rho \partial_j\left(\rho \partial_if[w]\right)\|_{Y(p)}+\|\rho \partial_j \left(NE[w]+\mathcal{L}E[w] + E[\mathcal{L}w]\right)\|_{Y(p)}\\
	&\quad +\|NE[\rho \partial_iw]+\mathcal{L}E[\rho \partial_iw] + E[\mathcal{L}(\rho \partial_iw)]\|_{Y(p)}+\|\rho \partial_j\left(\rho \partial_iw_0\right)\|_{W^{1,\infty}},
	\end{align}
	which can be rewritten as
	\begin{align}\MoveEqLeft
	\|\rho^2  \partial_{ij}^2 w \|_{W^{1,\infty}}+ \|\rho^2  \partial_{ij}^2 w   \|_{X(p)}\\
	&\lesssim \|\rho^2  \partial_{ij}^2 f[w] \|_{Y(p)}+\|\rho \partial_j \left(NE[w]+\mathcal{L}E[w] + E[\mathcal{L}w]\right)\|_{Y(p)}\\
	&\quad +\|NE[\rho \partial_iw]+\mathcal{L}E[\rho \partial_iw] + E[\mathcal{L}(\rho \partial_iw)]\|_{Y(p)}+\| w_0\|_{W },
	\end{align}
	by the virtue of the second order derivative \eqref{114}. The linear terms are, again, relatively easy to bound, as we have 
	\begin{align*}
	\mel |NE[\rho\partial_i w] +\L E[\rho\partial_i w] + E[\L(\rho\partial_i w)]| + |\rho\partial_j \left(NE[w] + \L E[w] +E[\L w]\right)|\\
	& \lesssim \rho^3 |\grad^5 w| +\rho^2 |\grad^4 w| + \rho |\grad^3 w| + |\grad^2 w| + |\grad w|,
	\end{align*}
	and thus, the $Y(p)$ norm of the linear terms is controlled by the $X(p)$ and $L^{\infty}$ norms of $w$ and $\rho\grad w$, which are in turn bounded by $\|w_0\|_W$ by the virtue of Theorem \ref{maintheoremseis} and the second order estimates in \eqref{114}. 
	
	Let's thus focus on the nonlinear terms. They take the form
	\begin{align*}
	\rho^2 \partial_{ij}^2 f[w] & = p\star R[w]\star K ,
	\end{align*}
	where
	\begin{align*}
	K & = \rho (\grad w)^{2\star} + \rho \grad w\star \grad^2 w + \rho (\grad^2 w)^{2\star} + \rho \grad w \star \grad^3 w + \rho^2 (\grad^2 w)^{3\star} +\rho^2 \grad^2 w\star\grad^3 w \\
	&\quad+ \rho^2 \grad w\star\grad^4 w+ \rho^3 (\grad^2 w)^{2\star} \star\grad^3 w+\rho^3 (\grad^2 w)^{4\star} + \rho^3 \grad^2 w\star \grad^4 w +\rho^3 \grad w\star\grad^5 w \\
	&\quad + \rho^4 \grad w \star\grad^6 w+ \rho^4 \grad^2 w \star \grad^5 w + \rho^4 (\grad^2 w)^{2\star} \star \grad^4 w+ \rho^4 (\grad^2 w)^{3\star}\star\grad^3 w\\
	&\quad  +\rho^4 (\grad^2 w)^{3\star}\star \grad^3 w+ \rho^4 (\grad^2 w)^{5\star} +   \rho^3 (\grad^3 w)^{2\star}  + \rho^4 \grad^2 w \star (\grad^3 w)^{2\star}  + \rho^4 \grad^3 w\star\grad^4 w ,  
	\end{align*}
	as the reader may check in a lengthy but straightforward exercise. The bound of $K$ is surprisingly simple as, thanks to the second order estimates \eqref{114}, no interpolations have to be performed. We simply have
	\begin{align*}
	\mel	\|K\|_{Y(p)}\\
	& \lesssim \left(\|\grad w\|_{L^{\infty}} + \sum_{k=1}^4\|\rho \grad^2 w\|_{L^{\infty}}^k\right) \left(\|\grad w\|_{L^{\infty}} + \|w\|_{X(p)} + \|\rho \grad w\|_{X(p)}\right) \\
	&\quad + \|\grad w\|_{L^{\infty}} \|\rho^4 \grad^6 w\|_{Y(p)} + \|w\|_{X(p)} \|\rho^2 \grad^3 w\|_{L^{\infty}}  + \|\rho \grad^2 w\|_{L^{\infty}} \|w\|_{X(p)} \|\rho^2 \grad^3 w\|_{L^{\infty}}\\
	&\lesssim \|w_0\|_{W^{1,\infty}} + \|\rho\grad^2 w_0 \|_{L^{\infty}} + \eps_0\left( \|\rho^2 \grad^2 w\|_{L^{\infty}} + \|\rho^2 \grad^2 w\|_{X(p)}\right),
	\end{align*}
	where we invoked the second order estimates \eqref{114} and the a priori estimates from Theorem \ref{maintheoremseis} in the second inequality. We derive the statement of the theorem by choosing $\eps_0$ sufficiently small.
\end{proof}

\section{The truncated problem}\label{S5}
The particular form of the nonlinearity limitates the well-posedness theory for the Cauchy problem for \eqref{perturbationequation} to a small neighborhood of the trivial solution $w\equiv 0$. It follows that the resulting semi-flow is  necessarily {\em local}. In order to construct a \emph{global} semi-flow, whose existence simplifies the construction of invariant manifolds significantly, it is customary to consider a truncated version of the perturbation equation. We thus introduce a cut-off function that eliminates the nonlinear terms  (locally) near points where the solution $w$, or one of its (suitably weighted) derivatives, is too large. This way, the equation becomes linear at these points. The cut-off remains inactive as long as the solution is globally  small with respect to $\|\cdot \|_W$, which is the case for solutions of the perturbation equation for sufficiently small initial datum due to Theorem \ref{dritteableitung}.

To make this truncation more precise
we recall that the perturbation equation reads
\begin{align}\label{100}
\partial_tw+\mathcal{L}^2w+N\mathcal{L}w=\rho^{-1}\nabla \cdot\left( \rho^2 F[w]\right)+\rho F[w],
\end{align}
where the nonlinear terms are schematically given by
\begin{align}\label{101}
F[w] = p\star R_l[w]\star\left(\rho\nabla^3w\star\nabla w+ \rho(\nabla^2w)^{2\star} + \nabla^2w\star\nabla w+ (\nabla w)^{2\star}\right),
\end{align}
cf.~\eqref{perturbationequation} and \eqref{nonlinearityperturbationequation}. Let $\hat{\eta}:[0,\infty) \rightarrow [0,1]$ be a smooth cut-off function that is supported on $[0,2)$ with $\hat{\eta}(x) =1$ if $0 \leq x\leq1$. For $\varepsilon \in (0,1)$, we define
\begin{align}
\eta_\varepsilon = \eta_\varepsilon\left[w,\nabla w,\rho\nabla^2w,\rho^2\nabla^3w\right] \coloneqq \hat{\eta}\left(\frac{w^2}{\varepsilon^2}\right)\hat{\eta}\left(\frac{|\nabla w|^2}{\varepsilon^2}\right)\hat{\eta}\left(\frac{\left|\rho\nabla^2w\right|^2}{\varepsilon^2}\right)\hat{\eta}\left(\frac{\left|\rho^2\nabla^3w\right|^2}{\varepsilon^2}\right).
\end{align}
The truncated problem we consider now is the following:
\begin{align}\label{c3}
\partial_tw+\mathcal{L}^2w+N\mathcal{L}w=\rho^{-1}\nabla \cdot\left( \rho^2 F_{\eps}[w]\right)+\rho F_{\eps}[w],\quad F_{\eps} = \eta_{\eps}F.
\end{align}
It is clear that this equation coincides with \eqref{perturbationequation} as long as all terms $\left|w\right|$, $\left|\nabla w\right|$, $\left|\rho \nabla^2w\right|$ and $\left|\rho^2 \nabla^3w\right|$ are globally bounded from above by $\varepsilon$. As we already know for solutions $w(t)$ of the full perturbation equation \eqref{perturbationequation} that $\left\|w(t)\right\|_{W}$ is controlled by $ \left\|w_0\right\|_{W  },$ provided that the initial datum $w_0$ is sufficiently small, the solutions of both equations coincide  if $\left\|w_0\right\|_{W}\ll \varepsilon$. Thus, in this situation the truncation does not change the dynamics, even though it has the advantage that we end up with a globally well-posed equation, see Theorem \ref{wellposednessH1}.
We remark that the choice of a pointwise truncation is necessary in order to ensure the differentiability of the nonlinearity in $w$. It has, however, the drawback that the regularity estimates from \cite{SeisTFE} seem not to carry over to the truncated problem. The technical difficulties arise from the fact that derivatives are falling onto the cut-off functions and the resulting terms fail to be controlled in a way analogously to the nonlinear terms in the original problem.

Moreover, it is crucial that derivatives up to third order are suitable truncated. This looks at first glance surprising because the original theory \cite{SeisTFE} for the perturbation equation \eqref{100} requires only the control of Lipschitz norms. However, it turns out that the well-posedness theory for a truncated equation becomes unexpectedly subtle if the truncation is performed only up to first order.

We will prove well-posedness of \eqref{c3} in the Hilbert space $H$, which, as we  will see, appears very naturally in the treatment of the truncated equation.
Even though it is in general not necessary to work in a Hilbert space setting to construct invariant manifolds, see, e.g., \cite{ChenHaleTan97}, this choice will be extremely convenient. Moreover, we can take advantage of the spectral analysis developed in \cite{McCannSeis15} in a nearly identical setting.

In order to prove well-posedness of the truncated problem in $H$, we need to extend   the maximal regularity result \eqref{b2} for the operator $\mathcal{L}$  to the Hilbert space $H$.
\begin{lemma}\label{maxregH}
	The operator $\mathcal{L}$ satisfies the maximal regularity estimate
	\begin{align}
	\left\|\nabla w\right\|_H+\|\rho\nabla^2 w\|_H \lesssim \left\|\mathcal{L}w\right\|_H.
	\end{align}
\end{lemma}
For the proof we refer to the theory for the operator $\mathcal{L}_\sigma $ in \eqref{504} and its derivatives developed in \cite{SeisTFE}, more precisely Lemmas 1,2 and 4 and their proofs.
The proof of Lemma \ref{maxregH} can be done analogously. It mainly relies on the observation that the operator $\mathcal{L}_\sigma$ commutates with tangential derivatives and its radial derivative $\partial_r\mathcal{L}_\sigma w$ can be rewritten in terms of $\mathcal{L}_{\sigma+1}\partial_r w$ and lower order terms. This makes the maximal regularity estimate for $\mathcal{L}_\sigma$,  equation \eqref{b2}, applicable.

The proof of well-posedness of the truncated problem exploits a fixed point argument. For this it is necessary to control the Lipschitz constants of the nonlinear terms $F_{\eps}$ in a suitable way.
\begin{lemma}\label{DifferenzF}
	It holds that
	\begin{align}
	\mel	\left\|\sqrt{\rho} F_{\eps} \left[w_1\right]-\sqrt{\rho}F_{\eps} \left[w_2\right]\right\| \\  
	&\lesssim \varepsilon \left( \left\|\rho \nabla^2 w_1-\rho \nabla^2 w_2 \right\|_H+\left\|\nabla w_1-\nabla w_2\right\|_H + \left\|w_1-w_2\right\|_H \right).
	\end{align}
\end{lemma}
\begin{proof}
	This is a straightforward computation embarking from the pointwise estimate 
	\begin{align}
	\mel \left|\rho  F_{\eps} [w_1]-\rho  F_{\eps} [w_2]\right|\\
	& \lesssim \varepsilon \left(\left|\rho^2 \nabla^3w_1-\rho^2\nabla^3w_2\right|+\left|\rho\nabla^2w_1-\rho\nabla^2w_2\right|
	+\left|\nabla w_1-\nabla w_2\right|+\left|w_1-w_2\right|\right),
	\end{align}
	which in turn can be readily checked. 
	Indeed, the latter implies that
	\begin{align*}
	\mel 	\left\|\sqrt{\rho} F_{\eps} [w_1] - \sqrt{\rho} F_{\eps} [w_2]\right\| \\
	& =\left\|\rho  F_{\eps} [w_1] - \rho F_{\eps} [w_2]\right\|_{L^2}\\
	& \lesssim \varepsilon\left( \left\|\rho^2 \nabla^3w_1-\rho^2\nabla^3w_2\right\|_{L^2}+\left\|\rho\nabla^2w_1-\rho\nabla^2w_2\right\|_{L^2}+\left\|\nabla w_1-\nabla w_2\right\|_{L^2}+\left\|w_1-w_2\right\|_{L^2}\right)
	\\
	&\lesssim \varepsilon\left( \left\| \nabla^2w_1-\nabla^2w_2\right\|_{H}+\left\|\nabla w_1-\nabla w_2\right\|_{H}+ \left\|w_1-w_2\right\|_{H}\right),
	\end{align*}	 
	where we have used \eqref{103} in the last inequality.
\end{proof}

With this preparation, we are in the position to derive well-posedness.

\begin{theorem}[Global well-posedness in $H$]\label{wellposednessH1}
	There exists $\varepsilon^*>0$ such that for every $\varepsilon\leq \varepsilon^*$ and every initial datum $w_0\in H$ the truncated problem \eqref{c3} has a unique global solution $w$. Moreover, the solution $w$ satisfies
	\begin{align}
	\left\|w\right\|_{L^\infty\left((0,\infty);H\right)}+\left\|\nabla w\right\|_{L^2\left((0,\infty);H\right)}+\|\rho\nabla^2w\|_{L^2\left((0,\infty);H\right)}\lesssim \left\|w_0\right\|_{H}.
	\end{align}
\end{theorem}
\begin{proof}
	We commence by considering the linear initial value problem
	\begin{align}\label{a6}
	\begin{cases}
	\partial_t\tilde{w}+\mathcal{L}^2\tilde{w}+N\mathcal{L}\tilde w &= \rho^{-1} \nabla \cdot\left( \rho^2 \tilde{F}\right)+\rho^2 \tilde{F}\\
	\tilde{w}(0,\cdot)&=w_0
	\end{cases}
	\end{align}
	for fixed $\tilde{F}   \in L^2((0,\infty);L^2(\rho^2 ))$. The problem \eqref{a6} has a unique weak solution $\tilde{w}$ on the time interval $(0,T)$, see Lemma 7 in \cite{SeisTFE}. It satisfies the estimate
	\begin{equation}
	\label{a7}
	\begin{aligned}
	\mel \left\|\tilde{w}\right\|_{L^\infty\left((0,T);H\right)}+\|\nabla \tilde{w}\|_{L^2\left((0,T);H\right)}+\left\|\rho\nabla^2\tilde{w}
	\right\|_{L^2\left((0,T);H\right)}\\
	&\le C_T\left(\left\|\rho\tilde{F}\right\|_{L^2\left((0,T);L^2\right)} +\left\|w_0\right\|_{H} \right).
	\end{aligned}\end{equation}
	To derive \eqref{a7}, we test the equation with $w$ in the inner product $\langle \cdot,\cdot \rangle_H$ and obtain after multiple integration by parts
	\begin{align}\label{104}
	\frac{1}{2}\frac{d}{dt}\left\|\tilde w\right\|^2_H + \left\|\mathcal{L}\tilde w\right\|^2_H+N\left\|\mathcal{L}^{1/2}\tilde w\right\|^2_H= - \langle \rho \tilde{F},\nabla \tilde w\rangle_H+\langle \rho \tilde{F},\tilde w\rangle_H.
	\end{align}
	Using the Cauchy-Schwarz inequality in the energy space $L^2(\rho)$, we furthermore notice that
	\begin{align}
	\left|\langle \rho \tilde{F},\nabla \tilde w\rangle_H\right|&\leq\left|\la \rho\tilde F,\grad \tilde w\ra\right| +\left| \la \rho\tilde F,\grad\L \tilde w\ra\right| \\& \le \|\sqrt{\rho} \tilde F\|\left( \|\sqrt{\rho} \grad \tilde w\| + \|\sqrt{\rho} \grad \L \tilde w\|\right)\
	 \le \|\sqrt{\rho} \tilde F \|\left(\|\tilde w\|_H + \|\L\tilde  w\|_{H}\right)
	\end{align}
	and
	\begin{align}
	\left| \langle \rho \tilde{F},\tilde w \rangle_H \right| & \le \left| \langle \rho \tilde{F}, \tilde w\rangle \right| + \left| \langle \rho \tilde{F},\mathcal{L}\tilde  w\rangle\right|  \\
	& \leq \left\|\rho \tilde{F}\right\|\left(\left\|\tilde  w\right\|+\left\|\mathcal{L}\tilde  w\right\|\right)\leq \left\|\rho \tilde{F}\right\|\left(\left\| \tilde w\right\|_H+\left\|\mathcal{L}^{1/2}\tilde  w\right\|_H\right).
	\end{align}
	We now invoke Young's inequality and the fact that $\rho\le 1$ and we drop the non-negative lower-order term on the left-hand side to derive the differential inequality
	\[
	\frac{d}{dt}\left\|\tilde w\right\|^2_H + \left\|\mathcal{L}\tilde w\right\|^2_H \lesssim \left\|\sqrt{\rho}\tilde{F}\right\|^2 + \left\| \tilde w\right\|^2_H.
	\]
	We deduce \eqref{a7} with help of the maximal regularity result of Lemma \ref{maxregH}   and a Gronwall type argument.

	To show well-posedness for the nonlinear problem, we apply a fixpoint argument. The estimate in Lemma \ref{DifferenzF} shows that the nonlinearity $F_{\eps}[w]$ belongs to $ L^2((0,T);L^2(\rho^2))$ whenever  $w\in L^\infty((0,T);H)$ is given such that $\nabla w, \rho \nabla^2w \in L^2((0,T);H)$. By the linear theory, 
	there exists thus a solution $\tilde{w}=\tilde{w}(w,w_0)$ to the Cauchy problem \eqref{a6} with $\tilde{F} = F_{\eps}[w]$, and the estimate \eqref{a7} and 
	Lemma \ref{DifferenzF} (applied to  $w_1=\tilde{w}$ and $w_2=0$)   yield
	that
	\begin{align}\MoveEqLeft
	\left\|\tilde{w}\right\|_{L^\infty\left((0,T);H\right)}+\left\|\nabla \tilde{w}\right\|_{L^2\left((0,T);H\right)}+\left\|\rho \nabla^2\tilde{w}
	\right\|_{L^2\left((0,T);H\right)}\\&\le  C_T\varepsilon\left( \left\|\rho \nabla^2w\right\|_{L^2\left((0,T);H\right)}+\left\|\nabla w\right\|_{L^2\left((0,T);H\right)}  +\left\|w\right\|_{L^\infty\left(\left(0,T\right);H\right)}\right)+C_T\left\|g\right\|_{H}.	
	\end{align}
	Similarly, given $w_1$ and $w_2$ in the same class of functions, the difference of the corresponding solutions $\tilde{w}_1\left(w_1,g\right)$ and $\tilde{w}_2\left(w_2,g\right)$ to the associated linear problems is bounded by
	\begin{align}\MoveEqLeft
	\left\|\rho\nabla^2 \tilde{w}_1-\rho\nabla^2 \tilde{w}_2\right\|_{L^2\left(H\right)}+\left\|\nabla \tilde{w}_1-\nabla \tilde{w}_2\right\|_{L^2\left(H\right)}+\left\|\tilde{w}_1-\tilde{w}_2\right\|_{L^\infty(H)} \\ 
	&\le  C_T\varepsilon\left(\left\|\rho\nabla^2 w_1-\rho\nabla^2 w_2\right\|_{L^2\left(H\right)}+\left\|\nabla w_1-\nabla w_2\right\|_{L^2\left(H\right)}+\left\|w_1-w_2\right\|_{L^\infty(H)}\right).
	\end{align}
	We conclude that, for $\varepsilon$ sufficiently small, the mapping $w\mapsto \tilde{w}(w,w_0)$ is a contraction on the space $\left\{w\in L^\infty\left((0,T); H \right)\text{ with }\nabla w \in L^2\left((0,T); H\right) \text{ and } \rho \nabla^2w \in L^2\left((0,T); H\right)\right\}$. An application of Banach's fixed point theorem shows that there exists a unique solution $w$ to the truncated problem \eqref{perturbationequation} with initial datum $w_0\in H$. We stress that the constructed solution is defined locally in time and that the size of the admissible $\eps$ is dependent on $T$. In the following, we choose $\eps$ for $T=1$ and show that the constructed solution can be extended globally in time.

	Our starting point is the estimate for the linear problem \eqref{104}, in which we choose $\tilde w=w$ and $\tilde F=F_{\eps}[w]$. In order to avoid a time-dependency in the estimate for $w$, we should estimate the nonlinearities slightly differently as above. We notice that the nonlinearity obeys the pointwise estimate
	\begin{equation}\label{106}
	|F_{\eps}[w]|   \lesssim \rho |\grad w||\grad^3 w| + \rho |\grad^2 w|^2 + |\grad w| |\grad^2 w| + |\grad w|^2 ,
	\end{equation}
	which implies that
	\[
	\|\rho F_{\eps}[w]\|_{L^1} \lesssim \|\grad w\|_{L^2} \|\rho^2 \grad^3 w\|_{L^2} + \|\rho \grad^2 w\|_{L^2}^2 + \|\grad w\|_{L^2} \|\rho \grad^2 w\|_{L^2} + \|\grad w\|_{L^2}^2
	\]
	via the Cauchy--Schwarz inequality.	In view of the norm characterization in \eqref{103}, the latter can be rewritten as
	\[
	\|\rho F_{\eps}[w]\|_{L^1} \lesssim \|\grad w\|_H\left(\|\grad w\|_H + \|\rho\grad^2 w\|_H\right).
	\]
	We also notice that
	\[
	|w| + |\grad w| + |\L w|+  \rho |\grad \L w| \lesssim |w| + |\grad w| + \rho |\grad^2 w| + \rho^2 |\grad^3 w| \lesssim \eps
	\]
	in the support of the nonlinearity $F_{\eps}$ by our choice of the cut-off. Thanks to the previous two bounds, the nonlinear terms on the right-hand side of \eqref{104} are estimated as follows:
	\begin{align*}
	\mel \left| \langle \rho {F}_{\eps}[w],\nabla  w\rangle_H\right|+\left|\langle \rho {F}_{\eps}[w], w\rangle_H\right| \\
	& = \left| \langle \rho {F}_{\eps}[w],  w\rangle\right|+\left|  \langle \rho {F}_{\eps}[w],\nabla  w\rangle\right|+\left|  \langle \rho {F}_{\eps}[w],\L  w\rangle\right|+ \left| \langle \rho {F}_{\eps}[w],\nabla \L w\rangle\right|\\
	& \lesssim \eps \|\rho F_{\eps}[w]\|_{L^1}\\
	& \lesssim \eps \|\grad w\|_H\left(\|\grad w\|_H + \|\rho\grad^2 w\|_H\right).
	\end{align*}
	Substitution into \eqref{104} thus yields
	\[
	\frac{d}{dt} \|w\|_H^2 + \|\L w\|_{H}^2 \lesssim \eps \|\grad w\|_H\left(\|\grad w\|_H + \|\rho \grad^2 w\|_H\right),
	\]
	where we have again dropped  the lower order term on the left-hand side. In view of the maximal regularity estimate from Lemma \ref{maxregH}, the right-hand side can be absorbed into the left-hand side provided that $\eps$ is chosen sufficiently small. This gives
	\[
	\frac{d}{dt} \|w\|_H^2 + \frac1C\|\L w\|_{H}^2 \le 0,
	\]
	for some $C>1$, and the local solution can thus be extended globally for all times. The estimate in the assertion of the theorem follows. 
	%
	%
	%
	%
	%
\end{proof}

It will be crucial for our analysis to have some smoothing properties established for the truncated equation \eqref{c3}. This will be achieved in the following two lemmas. 

\begin{lemma}\label{L1}
	There exists $\eps^*$ possibly smaller than in Theorem \ref{wellposednessH1}, such that for any $0< \eps\le \eps^*$ the following holds: If $w$ is the solution to the truncated equation \eqref{c3} with initial datum $w_0\in H$	then it holds that
	\begin{align}\label{e8}\MoveEqLeft
	\left\|\partial_tw\right\|_{L^q\left((1/4,2);L^q(\rho)\right)} +\left\|w\right\|_{L^q\left((1/4,2);L^q(\rho)\right)}+\left\|\nabla w\right\|_{L^q\left((1/4,2);L^q(\rho)\right)}\\&+ \left\|\nabla^2w\right\|_{L^q\left((1/4,2);L^q(\rho)\right)} +\left\|\rho\nabla^3w\right\|_{L^q\left((1/4,2);L^q(\rho)\right)}+\left\|\rho^2\nabla^4w\right\|_{L^q\left((1/4,2);L^q(\rho)\right)} \lesssim \left\|w_0\right\|_{H}
	\end{align}
	for any $q\in (1,\infty)$.
\end{lemma}

\begin{proof}We will perform an iterative argument for which it is convenient to localize time on an arbitrary scale. For this purpose, we fix   $T\in(0,2)$ and introduce a smooth cut-off function $\phi_1:\mathbb{R}^+_0\rightarrow [0,1]$, satisfying $\phi_1(t)=0$ if $t\leq T$ and $\phi_1(t)=1$ if $t\ge 2T$. Of course, its growth rate is inversely proportional to the cut-off scale $T$, but having this quantity uniformly finite throughout the proof, we will simply write $|\phi'_1|\lesssim 1$ for convenience. Smuggling $\phi_1$ into the truncated equation \eqref{c3} gives
	\begin{align}
	\partial_t(w \phi_1) + \mathcal{L}^2(w\phi_1) + N\mathcal{L}(w\phi_1) = \rho^{-1}\nabla \cdot \left(\rho^2F_{\eps}[w]\right)\phi_1 +\rho F_{\eps}[w]\phi_1 + w\phi_1'.
	\end{align}
	We note that $w\phi_1$ has zero initial datum, which makes the maximal regularity theory for $\mathcal{L}^2+N\mathcal{L}$ applicable: From \eqref{112} and elementary computations we infer the maximal regularity estimate
	\begin{equation}
	\label{c4}
	\begin{aligned}
	\MoveEqLeft 
	\left\|\partial_t (w\,  \phi_1)\right\|_{L^2\left(L^2(\rho)\right)}+\left\|\nabla^2w\,\phi_1\right\|_{L^2\left(L^2(\rho)\right)}+\left\|\rho\nabla^3w\,\phi_1\right\|_{L^2\left(L^2(\rho)\right)} + \left\|\rho^2\nabla^4w\, \phi_1\right\|_{L^2\left(L^2(\rho)\right)} \\
	& \lesssim \|\eta_{\eps} F\, \phi_1\|_{L^2(L^2(\rho))} + \|\rho \grad \eta_{\eps}\, F  \phi_1\|_{L^2(L^2(\rho))} + \|\rho \eta_{\eps} \grad F\, \phi_1\|_{L^2(L^2(\rho))} + \left\|w\phi_1'\right\|_{L^2\left( L^2(\rho)\right)},
	\end{aligned}
	\end{equation}
	where we have set $F=F_{\eps}[w]$ for brevity. For brevity, we have dropped the time interval $(0,2)$ in the norms. The final term on the right-hand side is easily controlled via the a priori estimates from Theorem \ref{wellposednessH1} and the defining properties of the temporal cut-off $\phi_1$: It holds that	
	\[
	\left\|w\phi_1'\right\|_{L^2\left(L^2(\rho)\right)} \lesssim   \left\|w\right\|_{L^\infty\left(L^2(\rho)\right)} \le \|w\|_{L^{\infty}(H)} \lesssim  \left\|w_0\right\|_{H}.
	\] 	
	For the first and the second term, we use the pointwise bound on the nonlinearity on the support of $\eta_{\eps}$,
	\begin{equation}
	\label{107}
	|F[w]| \lesssim \eps \left(|\grad w| + |\grad^2 w|  + \rho|\grad^3 w|\right)\lesssim \rho^{-1}\eps^2,
	\end{equation}
	cf.~\eqref{106}. More precisely, plugging the first of the two estimates into the first term on the right-hand side of \eqref{c4}, we find that
	\begin{align*}
	\|\eta_{\eps} F\, \phi_1\|_{L^2(L^2(\rho))} &\lesssim \eps\left(\|\grad w \phi_1\|_{L^2(L^2(\rho))}+ \|\grad^2 w \phi_1\|_{L^2(L^2(\rho))} + \|\rho \grad^3 w \phi_1\|_{L^2(L^2(\rho))}\right).
	\end{align*}
	We interpolate the first term with the help of Lemma \ref{interpolationintequality} in the appendix, so that
	\begin{align*}
	\|\eta_{\eps} F\, \phi_1\|_{L^2(L^2(\rho))} &\lesssim \eps\left(\| w \phi_1\|_{L^2(L^2(\rho))}+ \|\grad^2 w \phi_1\|_{L^2(L^2(\rho))} + \|\rho \grad^3 w \phi_1\|_{L^2(L^2(\rho))}\right).
	\end{align*}
	The two last terms on the right-hand side can be absorbed into the left-hand side of \eqref{c4} if $\eps$ is chosen sufficiently small, while the first term  is controlled by the initial datum through the energy estimate of Theorem \ref{wellposednessH1}.

	To estimate the second term on the right-hand side of \eqref{c4}, we notice that
	\[
	|\grad\eta_{\eps}| \lesssim 1 + \frac1{\eps}\left(|\grad^2 w| + \rho |\grad^3 w| +\rho^2 |\grad^4 w|\right),
	\] 
	and thus, using that $\rho\le 1$ and the second estimate in \eqref{107}, we find 
	\begin{align*}
	\mel
	\|\rho \grad \eta_{\eps}\, F  \phi_1\|_{L^2(L^2(\rho))}\\
	& \lesssim \|\chi_{\supp \eta_{\eps} } \, F  \phi_1\|_{L^2(L^2(\rho))}\\
	&\quad + \eps\left(\| \grad w \phi_1\|_{L^2(L^2(\rho))} +\| \grad^2 w    \phi_1\|_{L^2(L^2(\rho))} +\|\rho \grad^3 w   \phi_1\|_{L^2(L^2(\rho))} \right).
	\end{align*}
	The first term can be estimated as before and the second one can be absorbed into the left-hand side of \eqref{c4} if $\eps$ is sufficiently small.
	
	It remains to study the third term on the right-hand side of \eqref{c4}. Here, we find after a small computation that
	\[
	\rho|\grad F| \lesssim |F| +\eps\left(|\grad w|+ |\grad^2 w| + \rho|\grad^3w | +\rho^2|\grad^4w|\right).
	\]
	Hence, in view of the bound in \eqref{107}, the only new term we have to deal with is the fourth-order term. This one, however, can be controlled as the second- and third-order term before by absorption into the left-hand side of \eqref{c4}.
	
	Combining all the estimates that we discussed, adding the lower order term from the energy inequality in Theorem \ref{wellposednessH1} to the left-hand side, making use of the interpolation inequality in Lemma \ref{interpolationintequality} in the appendix to include the first order spatial gradient and finally  dropping all higher order terms, we arrive at
	\begin{equation}
	\label{108}
	\left\| w\,  \phi_1 \right\|_{L^2\left(L^2(\rho)\right)}+ 	\left\|\partial_t (w\,  \phi_1)\right\|_{L^2\left(L^2(\rho)\right)}+\left\|\nabla( w\,\phi_1)\right\|_{L^2\left(L^2(\rho)\right)}
	\lesssim \|w_0\|_H.
	\end{equation}

	We are now in the position to invoke the Sobolev inequality Lemma \ref{sobolevinequality} in the appendix, namely
	\begin{align}
	\left\|w\right\|_{L^q\left(L^q(\rho)\right)}\lesssim\left\|\partial_tw\right\|_{L^p\left(L^{p}(\rho)\right)} + \left\|w\right\|_{L^p\left(L^{p}(\rho)\right)} + \left\|\nabla w\right\|_{L^p\left(L^{p}(\rho)\right)},
	\end{align}
	where the integrability exponents $1\leq p\leq q<\infty$ are such that
	\begin{align}
	1-\frac{N+2}{p} = -\frac{N+2}{q}.
	\end{align}
	In our situation, that is $p=p_1=2$,   we deduce from \eqref{108} the inequality 
	\begin{equation}
	\label{109}
	\left\|w\phi_1\right\|_{L^{q_1}\left(L^{q_1}(\rho)\right)} \lesssim   \left\|w_0\right\|_{H},
	\end{equation}
	where now $ q= q_1 = \frac{2(N+2)}{N}$.

	In order to further increase the order of integrability, we have to use the maximal regularity estimate in $L^q$, see \eqref{112}. 
	We introduce a new smooth cut-off function $\phi_2:\mathbb{R}^+_0\rightarrow [0,1]$, such that $\phi_2(t)=0$ if $t\leq 2T$ and $\phi_2(t)=1$ if $t\ge 3T$. Using the maximal regularity estimate for $w\phi_2$ and $q_1$, we get
	\begin{align*}
	\mel	\left\|\partial_t(w\phi_2)\right\|_{L^{q_1}\left(L^{q_1}(\rho)\right)} + \left\|\nabla^2(w\phi_2)\right\|_{L^{q_1}\left(L^{q_1}(\rho)\right)}+\left\|\rho\nabla^3(w\phi_2)\right\|_{L^{q_1}\left(L^{q_1}(\rho)\right)}+\left\|\rho^2\nabla^4(w\phi_2)\right\|_{L^{q_1}\left(L^{q_1}(\rho)\right)} \\ 
	&\lesssim  \|\eta_{\eps} F\, \phi_2\|_{L^{q_1}(L^{q_1}(\rho))} + \|\rho \grad \eta_{\eps}\, F  \phi_2\|_{L^{q_1}(L^{q_1}(\rho))} + \|\rho \eta_{\eps} \grad F\, \phi_2\|_{L^{q_1}(L^{q_1}(\rho))}+ \left\|w\phi_2'\right\|_{L^{q_1}\left(L^{q_1}(\rho)\right)}.
	\end{align*}	
	The treatment of the right-hand side is almost identical to the $p=2$ case, only that now equation \eqref{109} is invoked where before the energy equation was used. We eventually arrive at
	\[
	\left\| w\,  \phi_2 \right\|_{L^{q_1}\left(L^{q_1}(\rho)\right)}+ 	\left\|\partial_t (w\,  \phi_2)\right\|_{L^{q_1}\left(L^{q_1}(\rho)\right)}+\left\|\nabla( w\,\phi_2)\right\|_{L^{q_1}\left(L^{q_1}(\rho)\right)}
	\lesssim \|w_0\|_H,
	\]
	and we may use the Sobolev inequality once more with $p_2 \le \min\{q_1,N+2\}$. By iterating this procedure, the order of integrability can be further increased.  After finitely many steps, depending only on the space dimension, and by choosing $T$ carefully, the statement follows.
\end{proof}

Theorem \ref{wellposednessH1} shows that the truncated equation generates a global semiflow in the Hilbert space setting. 
We define $S^t_\varepsilon : H \rightarrow H$ as the corresponding flow map,
\[
S_{\eps}^t(w_0) = w(t,\cdot)
\]
where $w$ is the unique solution to the truncated nonlinear problem \eqref{c3} with initial datum $w_0$. 
Our invariant manifold construction is based on that flow. More accurately, we choose to consider a discrete time setting by working with the time-one map rather than with the continuous flow. Compared to constructing the manifolds for the semiflow directly, this has the advantage, that the differentiability of the time-one map is a weaker property than its counterpart for flows, the variation of constants formula. We write $S_\varepsilon \coloneqq S_\varepsilon^1$.

The main regularity results for the perturbation variable $w$ are stated  uniformly in time and space, while  our invariant manifold theory will rely on Hilbert spaces. The connection of both necessitates to establish  suitable  smoothing estimates. We will do so in the following lemma which we improve after one time step. As we are interested in the long-time behavior, such a delayed smoothing statement does not cause any problems.

\begin{lemma}\label{smoothingestimate}
	Let $\varepsilon^*$ be as in lemma \ref{L1} and $\varepsilon\leq \varepsilon^*$. For any $w_0\in H$ the following holds: If $w(t)= S_\varepsilon^t(w_0)$ is the solution to the truncated equation, then 
	\begin{align}
	\left\|w(t)\right\|_{L^\infty} +\left\|\nabla w(t)\right\|_{L^\infty} + \left\|\rho \nabla^2w\right\|_{L^\infty}+\left\|\rho^2\nabla^3w\right\|_{L^\infty} \lesssim \left\|w_0\right\|_{H}
	\end{align}
	for all $t\geq 1/2$. In particular, this yields $\|S_\varepsilon(w_0)\|_{W}\lesssim \left\|w_0\right\|_{H}$. Moreover, there exists $\varepsilon^0\leq \min \left\{\varepsilon,\varepsilon_0\right\}$ such that $S^t_\varepsilon\left(S_\varepsilon(w_0)\right)=S^t\left(S_\varepsilon(w_0)\right)$ for $t>0$, provided that $\left\|w_0\right\|_H\leq \varepsilon^0$.
\end{lemma}

\begin{proof}
	Due to the Morrey-type embedding inequality \ref{morreyinequality} in the appendix, we have that $\left\|w\right\|_{L^\infty}\lesssim \left\|w\right\|_{L^q(\rho)}+\left\|\nabla w\right\|_{L^q(\rho)}$, provided that $q$ is sufficiently large. We can extend this estimate to higher order derivatives and find
	\begin{equation}\label{505}
	\|w\|_{W} \lesssim  \left\|w\right\|_{L^q(\rho)}+\left\|\nabla w\right\|_{L^q(\rho)} + \|\grad^2 w\|_{L^{q}(\rho)} +  \|\rho\grad^3 w\|_{L^{q}(\rho)}+ \|\rho^2\grad^4 w\|_{L^{q}(\rho)}.
	\end{equation}
	Thus, in order to establish the asserted estimate, we have to improve the estimate in  Lemma \ref{L1} to a pointwise-in-time statement. For this, we invoke a simple construction.
	
	For an arbitrarily given function $f\in L^q(1/4,1/2)$, we consider the set
	\[
	J_f=\left\{t\in (1/4,1/2) : \left|f(t)\right|>8\left\|f\right\|_{L^q(1/4,1/2)}\right\}.
	\]
By Chebyshev's inequality,  it holds that 
\[
\left\|f\right\|_{L^q(1/4,1/2)}\geq \left\|f\right\|_{L^q(J_f)}  \geq 8\left\|f\right\|_{L^q(1/4,1/2)} |J_f|^{1/q},
\]
 where $\left|\cdot\right|$ denotes the Lebesgue measure, and thus,  $|J_f|\leq \left(1/8\right)^q$. Moreover, since $q\ge 1$, we have also an estimate on the complementary set in $(1/4,1/2)$, namely $|J_f^c|\geq1/4-\left(1/8\right)^q \geq 1/8$.
Applying this estimate to the function $f(t)=\left\|w(t)\right\|_{L^q(\rho)}+\left\|\nabla w(t)\right\|_{L^q(\rho)}+\left\|\nabla^2 w(t)\right\|_{L^q(\rho)}+\left\|\rho\nabla^3w(t)\right\|_{L^q(\rho)}+\left\|\rho^2\nabla^4w(t)\right\|_{L^q(\rho)}$ and using the above estimate \eqref{505}, we find that
	\begin{align}
	\left\|w\right\|_{L^\infty(J_f^c;W)}&\lesssim \|f \|_{L^{\infty}(J_f^c)} \lesssim \|f\|_{L^{q}(1/4,1/2)}\\
&	\lesssim \left\|w\right\|_{L^q\left((1/4,1/2);L^q(\rho)\right)}+\left\|\nabla w\right\|_{L^q\left((1/4,1/2);L^q(\rho)\right)}+ \left\|\nabla^2w\right\|_{L^q\left((1/4,1/2);L^q(\rho)\right)}\\
	&\quad +\left\|\rho\nabla^3w\right\|_{L^q\left((1/4,1/2);L^q(\rho)\right)}+\left\|\rho^2\nabla^4w\right\|_{L^q\left((1/4,1/2);L^q(\rho)\right)} .
	\end{align}
By the virtue of Lemma \ref{L1}, the right-hand side is bounded by $\|w_0\|_H$.	This shows that there exists a time $\hat{t}\in (1/4,1/2)$ such that 
\begin{align}\label{p4}
\|w(\hat{t})\|_{W} = \|S^{\hat{t}}_\varepsilon(w_0)\|_{W} \leq C \|w_0\|_{H}.
\end{align}

	Now suppose that $\left\|w_0\right\|_H\leq \varepsilon^0$
	From Theorems \ref{maintheoremseis} and  \ref{dritteableitung} we know, that the nonlinear flow $S^t(w_0)$ can be controlled in $W$ by its initial data $g$ in the $W$-norm, i.e., $\|S^t(w_0)\|_{W} \leq \tilde{C}\|w_0\|_{W}$ for every $t\geq0$, provided that $\|w_0\|_{W}$ is sufficiently small. 
	If we now choose $\eps^0$ in a way such that $\tilde{C}C\eps^0\leq \eps$, we obtain that $S^t_\varepsilon\left(S^{\hat{t}}_\varepsilon(w_0)\right) = S^t\left(S_\varepsilon^{\hat{t}}(w_0)\right)$ for every $t\geq 0 $ and thus
	\begin{align}
	\|S^{\hat{t}+t}_\varepsilon(w_0)\|_{W} = \|S^t\left(S_\varepsilon^{\hat{t}}(w_0)\right)\|_{W}\leq \tilde{C}\|S^{\hat{t}}_\varepsilon(w_0)\|_{W}\leq \tilde{C}C\|w_0\|_{H}
	\end{align}
	for every $t\geq0$. Since $\hat{t}\in (1/4,1/2)$, this gives the result.
\end{proof}

By construction of the solution in Theorem \ref{wellposednessH1}, we know that $S_\varepsilon$ is Lipschitz-continuous. We decompose the global flow $S_\varepsilon$ into a linear and nonlinear part
\begin{align}
S_\varepsilon = L + R_\varepsilon, \quad \text{ where } L\coloneqq e^{-\left(\mathcal{L}^2+N\mathcal{L}\right)}.
\end{align}
As a difference of Lipschitz continuous functions, $R_\varepsilon$ is Lipschitz continuous as well.
Actually, its Lipschitz constant can be estimated in terms of $\varepsilon$ and becomes thus a contraction if $\varepsilon$ is sufficiently small.
\begin{lemma}\label{lipschitzR}
	Let $\varepsilon^*>0$ as in Lemma \ref{L1} and $0<\varepsilon\leq \varepsilon^*$. Then, for any $g$, $\tilde{g} \in H$ it holds that
	\begin{align}
	\left\|R_\varepsilon(g)-R_\varepsilon(\tilde{g})\right\|_{H}\lesssim \varepsilon\left\|g-\tilde{g}\right\|_{H}.
	\end{align}
\end{lemma}

\begin{proof}
	Let $g$, $\tilde{g}\in H$ be given. Then $w(t,x)=S_\varepsilon^t(g)$ and $\tilde{w}(t,x)=S_\varepsilon(\tilde{g})$ solve the truncated problem \eqref{c3} with initial data $g$ or  $\tilde{g}$, respectively. We set $v(t)=w(t)-L^tg$, where $L^tg$ is the solution to the linear problem with initial datum $g$, so that, in particular $v(1,x)=R_\varepsilon(g)$. Analogously we define $\tilde{v}$. Then $v-\tilde{v}$ solves the equation 
	\begin{align*}
	\mel	\partial_t(v-\tilde{v}) + \mathcal{L}^2(v-\tilde{v})+N\mathcal{L}(v-\tilde{v}) \\ 
	&= \frac{1}{\rho}\nabla \cdot \left( \rho^2\left(F_{\eps}[w]- 
	F_{\eps}[\tilde{w}] \right)\right) + \rho \left(F_{\eps}[w]-F_{\eps}[\tilde{w}]\right),
	\end{align*}
	with zero initial datum. 
	With the help of estimate \eqref{a7} from the proof of Theorem \ref{wellposednessH1} we deduce that 
	\begin{align*}\MoveEqLeft
	\left\|v(1)-\tilde{v}(1)\right\|_{H} + \left\|\nabla v-\nabla \tilde{v}\right\|_{L^2\left((0,1);H\right)} + \left\|\rho\nabla^2v-\rho\nabla^2\tilde{v}\right\|_{L^2\left((0,1);H\right)}\\
	&\lesssim \left\|\rho F_{\eps}[w] -\rho F_{\eps}[\tilde{w}]\right\|_{L^2\left((0,1);L^2\right)} \\
	&\lesssim \varepsilon\left( 	\left\|w-\tilde{w}\right\|_{L^\infty\left((0,1);H\right)} + \left\|\nabla w-\nabla \tilde{w}\right\|_{L^2\left((0,1);H\right)} + \left\|\rho\nabla^2w-\rho\nabla^2\tilde{w}\right\|_{L^2\left((0,1);H\right)} \right),
	\end{align*}
	where we used Lemma \ref{DifferenzF} in the last step. Since $S_\varepsilon$ is Lipschitz continuous, the right-hand side is controlled by $\varepsilon \|g-\tilde{g}\|_{H}$. This finishes the proof.
\end{proof}

Additionally we would like to know that $R_\varepsilon$ is quadratic near the origin. The superlinear behavior entails the differentiability of $R_\varepsilon$ in the origin, with derivative zero. Neither this information nor the regularity will be necessary for our construction of the invariant manifolds. However, as we will see, it provides the additional geometric insight that the center manifold $W_\varepsilon^c$ touches the stable Eigenspace $E_c$ tangentially, see Theorem \ref{centermanifold}. The proof of the quadratic estimate is rather technical and exploits smoothing properties of the nonlinear flow. We are able to show the quadratic behavior after a regularizing time step, in a  similar way as in  Lemma \ref{smoothingestimate}, what still is sufficient for our purpose.
\begin{lemma}\label{quadratischeabschaetzung}
	Let $\varepsilon^*$ be as in Lemma \ref{L1}. For all $0\leq \varepsilon\leq \varepsilon_*$ and every $g\in H$ it holds that 
	\begin{align}
	\left\|R_\varepsilon\left(S_\varepsilon(g)\right)\right\|_{H} \lesssim \left\|g\right\|^2_{H}.
	\end{align}
\end{lemma}

\begin{proof}
	Let $w(t,x)=S^t_\varepsilon(g)$ and set $W(t,x) = w(t+1,x)$, which yields $W(0,\cdot)=S_\varepsilon(g)$. Let $v$ solve the  initial value problem
	\begin{align}
	\begin{cases}
	\partial_tv+\mathcal{L}^2v+N\mathcal{L}v &= \frac{1}{\rho}  \nabla\cdot \left(\rho^2F_{\eps}[W]\right)  +\rho F_{\eps}[W]\quad \text{ in } (0,\infty)\times B_1(0)\\
	v(0,\cdot)&=0 \quad \text{ in } B_1(0),
	\end{cases}
	\end{align} 
	so that $ v(1,\cdot) = R_\varepsilon \left(W(0,\cdot)\right)=  R_\varepsilon\left(S_\varepsilon(g)\right)$.
	Thanks to the proof of Theorem \ref{wellposednessH1}, more precisely estimate \eqref{a7}, we know that
	\begin{align*}
	\left\|v(1)\right\|^2_{H}  &\lesssim \int \limits_0^1 \left\|\sqrt{\rho}F_{\eps}[W]\right\|^2\, dt  = \int \limits_1^2 \left\|\sqrt{\rho}F_{\eps}[w]\right\|^2\, dt,
	\end{align*}
	and by the virtue of the pointwise estimate \eqref{106} and Young's inequality, we deduce
	\begin{align*}\left\|v(1)\right\|_{H} &\lesssim \left\|\rho\nabla^3w\right\|_{L^4\left((1,2);L^4(\rho^2)\right)}^2 +\left\|\nabla^2w\right\|_{L^4\left((1,2);L^4(\rho^2)\right)}^2+\left\|\nabla w\right\|_{L^4\left((1,2);L^4(\rho^2)\right)}^2.
	\end{align*}
	It thus remains to invoke the smoothing property from Lemma \ref{L1} with $q=4$ and the bound $\rho\le 1$
	in order to prove the lemma.
\end{proof}

\begin{lemma}\label{Lipschitzglaettung}
	Let $\varepsilon^*$ be as in Lemma \ref{L1} and $\varepsilon\leq \varepsilon^*$. Let $\varepsilon^0 \leq \min \left\{\varepsilon, \varepsilon^*\right\}$ be as in Lemma \ref{smoothingestimate}. Then, for any $g, \tilde{g}\in H^1_{1,2}$ with $\left\|g\right\|_{H},\left\|\tilde{g}\right\|_{H}\leq \eps^0$ it holds that
	\begin{align}
	\left\| S^{\hat{t}}_\varepsilon(g)-S_\varepsilon^{\hat{t}}(\tilde{g}) \right\|_{W} \lesssim \left\|g-\tilde{g}\right\|_{H}
	\end{align}
	for some $\hat t\in(\frac45,1)$.
\end{lemma}
\begin{proof}
	Similar to the previous proof, we will make use of a maximal regularity estimate for the linear equation. However, this proof will be less technical, because the previous lemma, combined with a result of \cite{SeisTFE}, will allow us to consider the flow without the cut-off function $\eta_\varepsilon$.
	
	Let $w(t)$ denote $S^t_\varepsilon(g)$ and $\tilde w(t)=S^t_\varepsilon(\tilde{g})$ respectively. Then, by Lemma \ref{smoothingestimate} we know that $\left\|w(t)\right\|_{W} \lesssim \left\|g\right\|_{H}$ for every $t\geq 1/2$. At this point we invoke Theorem 2 of \cite{SeisTFE} to also achieve even better control (in terms of $\rho$) on the higher derivatives: It guarantees  that the unique solution $w$ of the full nonlinear perturbation equation \eqref{perturbationequation} with (of course small) initial data $g$ satisfies $\left|\nabla^2w(x,t)\right|+\left|\rho\nabla^3w(x,t)\right|+\left|\rho^2\nabla^4w(t,x)\right| \lesssim t^{-\kappa} \left\|g\right\|_{W^{1,\infty}}$ for some positive $\kappa>0$. If we apply this result with $w(1/2)$ as the initial data, we obtain the estimate
	\begin{align}\MoveEqLeft\label{330}
	\left\|w(t)\right\|_{L^\infty}+\left\|\nabla w(t)\right\|_{L^\infty}+\left\|\nabla^2 w(t)\right\|_{L^\infty}+\left\|\rho\nabla^3w(t)\right\|_{L^\infty}+\left\|\rho^2\nabla^4w(t)\right\|_{L^\infty}&\\&\lesssim \left\|g\right\|_{H}\leq \eps^0
	\end{align}
	uniformly in time for every $t\geq 3/4$. The same holds true for $\tilde w(t)$ and $\tilde{g}$.
	That is, for $t\geq3/4$ both $w(t)$ and $\tilde w(t)$ solve the full nonlinear equation. 
	
	We  now introduce $v=w-\tilde w$, which solves the initial value problem
	\begin{align}
	\begin{cases}
	\partial_t v+\mathcal{L}^2v+N\mathcal{L}v &= \rho^{-1} \nabla \cdot\left( \rho^2 \left(F_1[w]-F_1[\tilde w]\right)\right) +\rho \left(F_2[w]-F_2[\tilde w]\right)  ,\\
	v(0,\cdot)&=0.
	\end{cases}
	\end{align}
	Arguing very similarly as in the proof of Lemma \ref{L1}, but using \eqref{330} instead of the truncation, we arrive at
	\begin{align}\MoveEqLeft
	\left\|\partial_t v\right\|_{L^q\left((4/5,1);L^q(\rho)\right)} +\left\|v\right\|_{L^q\left((4/5,1);L^q(\rho)\right)}+\left\|\nabla v\right\|_{L^q\left((4/5,1);L^q(\rho)\right)}\\&+ \left\|\nabla^2v\right\|_{L^q\left((4/5,1);L^q(\rho)\right)} +\left\|\rho\nabla^3 v \right\|_{L^q\left((4/5,1);L^q(\rho)\right)}+\left\|\rho^2\nabla^4v\right\|_{L^q\left((4/5,1);L^q(\rho)\right)} \lesssim \left\|g-\tilde{g}\right\|_{H}
	\end{align}
	for any $q\in (1,\infty)$. Lastly, we proceed as in the proof of Lemma \ref{smoothingestimate} to prove the existence of a $\hat{t}\in (4/5,1)$, such that
	\[
	\|w(\hat t) - \tilde w(\hat t)\|_{W} = \|v(\hat t)\|_{W}  \lesssim \left\|g-\tilde{g}\right\|_{H}.
	\]
\end{proof}

\section{Dynamical System Arguments}\label{dynamicalsystemarguments}
In this chapter we will construct invariant manifolds and prove Theorem \ref{localmanifolds}. We want to draw a heuristic picture of the concept, see als Figure \ref{fig3}
\begin{figure}
	\begin{center}
		\includegraphics[width=\textwidth]{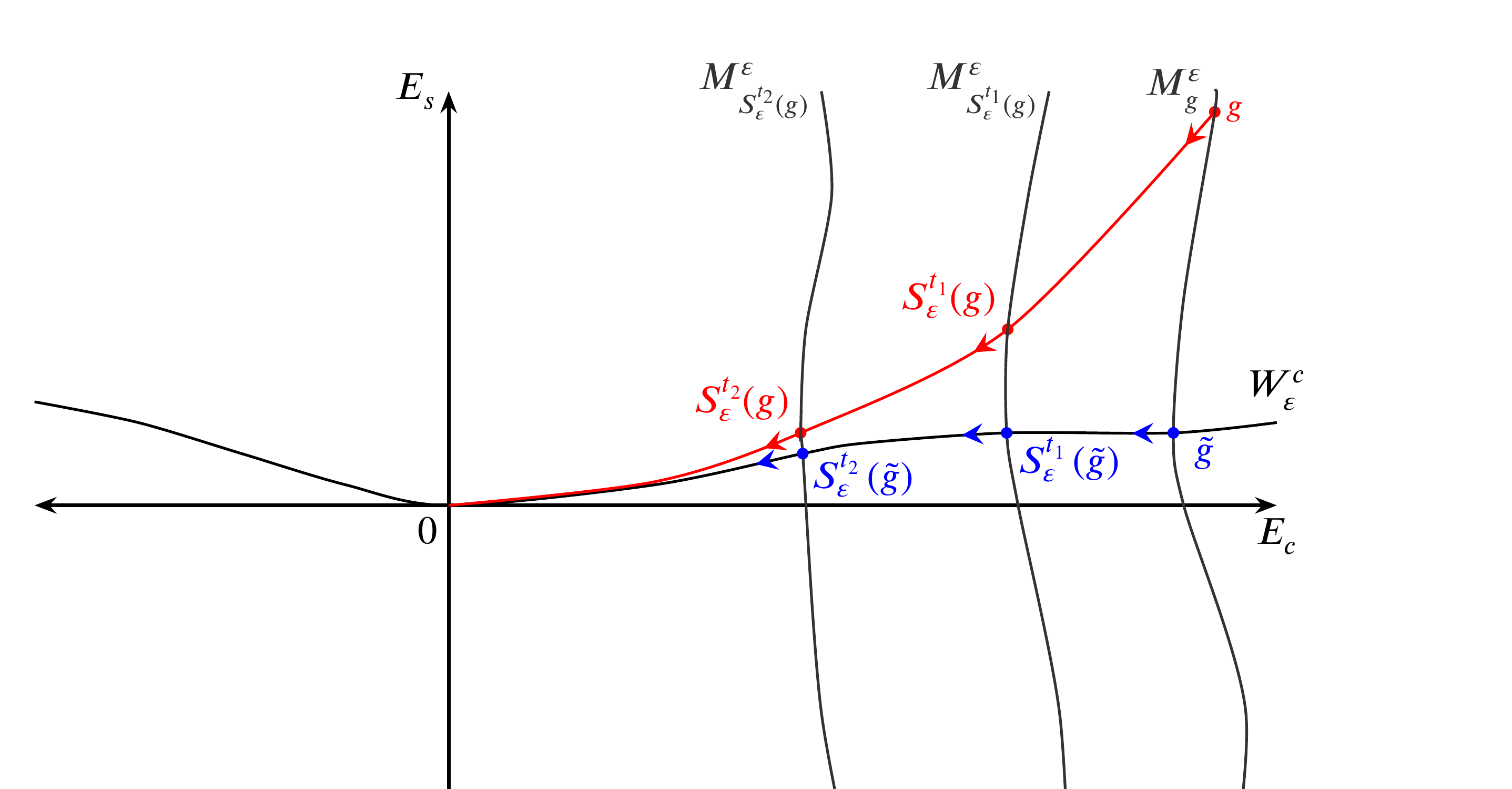}
	\end{center}
\caption{The flow $S^t_\varepsilon(g)$ starting from an arbitrary $g\in H$ approaches the flow starting from the unique intersection point $\tilde{g}$ of $W_\varepsilon^c$ and $M_g^\varepsilon$, $S^t\varepsilon(\tilde{g})$, that stays on the center manifold and whose longtime behavior dominates the asymptotics of $S^t_\varepsilon(g)$.}\label{fig3}
\end{figure}
for a geometric illustration. The center manifold, see Theorem \ref{centermanifold}, can be represented as the graph of a Lipschitz continuous function over the finite-dimensional   center eigenspace, and it touches the center eigenspace tangentially at the origin.  Here, the center eigenspace  is the subspace of $H$ spanned by the eigenfunctions of the first $K+1$ eigenvalues of $\mathcal{L}^2+N\mathcal{L}$, where $K$ is an arbitrarily fixed nonnegative integer.   Solutions to the truncated flow that lie on the center manifold remain on it for all subsequent times. 
The stable manifolds, see Theorem \ref{stablemanifold}, intersect with the center manifold in exactly one point,  and they form thus a foliation of the underlying Hilbert space $H$ over the center manifold. This foliation is invariant under the flow.  The stable manifolds can be described as  (displaced) graphs over the stable eigenspace, that is, the orthogonal complement of the center eigenspace. Given an arbitrary solution to the truncated perturbation equation, our construction provides a solution that approximates the given one with an exponential rate of at least $\mu_{K}$.





Throughout this section, we fix $\varepsilon^*$ as in Lemma \ref{L1} and choose some $
\varepsilon^0 \leq \min \left\{\varepsilon, \varepsilon_0\right\}$ as in Lemma \ref{smoothingestimate}. With these choices, all results from the previous two sections are admissible.


The linear operator $\mathcal{L}^2+N\mathcal{L}$ and the associated semi-flow operator  $L = e^{-\mathcal{L}^2-N\mathcal{L}}$  share the same eigenfunctions and an eigenvalue $\mu$ of $\mathcal{L}^2+N\mathcal{L}$  turns into the eigenvalue $e^{-\mu}$ of  $L$.  We recall that all spectrum information is contained in Theorem \ref{spektrum}. The  fact that the spectrum is discrete will facilitate our analysis substantially.

In our construction of the invariant manifolds, we follow an approach by Koch, see \cite{Kochoncentermanifolds}, and mainly stick to his notation. From now on we keep $K\in \mathbb{N}_0$ fixed, and we denote by $E_c$ the finite-dimensional subspace of $H$ spanned by the eigenfunctions corresponding to the eigenvalues $\{\mu_0,\dots,\mu_K\}$, that we call the \emph{center eigenspace}. The projection of $H$ onto the space $E_c$ is given by $P_c$. 
The \emph{stable eigenspace} $E_s$ is defined as the orthogonal complement of the center eigenspace, that is $E_s \coloneqq E_c^\perp$, such that $H=E_c \oplus E_s$, and $P_s=1-P_c$.
We denote the restriction of $L$ to $E_s$ by $L_s$; it can be estimated via $\left\|L_s\right\|_H\leq e^{-\mu_{K+1}}$. Indeed, for $w \in H$ it holds
\begin{align}
\left\|L_sw\right\|^2_{H} = \sum \limits_{k>K}\sum_{l} \langle  L w, \psi_{k,l}\rangle^2_H = \sum \limits_{k>K}\sum_{l}  e^{-2\mu_k} \langle w, \psi_{k,l} \rangle^2_H \leq e^{-2\mu_{K+1}} \left\|w\right\|^2_{H},
\end{align}
if the $\psi_{k,l}$'s are the eigenfunctions corresponding to $\mu_k$. 
For $L_c$, the restriction of $L$ onto $E_c$, we similarly obtain $\left\|L_c^{-1}\right\| \leq e^{\mu_K}$. Indeed,  we have
\begin{align}
\left\|L_c^{-1}w\right\|^2_{H}  = \sum \limits_{k\leq K} \sum_{l}\langle L^{-1} w,\psi_k \rangle^2_H = \sum \limits_{k \leq K}\sum_l e^{2\mu_k}\langle  w, \psi_k\rangle^2_H \leq e^{2\mu_k}\left\|w\right\|^2_{H}.
\end{align}
We define
\begin{align}
\Lambda_c = e^{-\mu_K}, \quad \Lambda_s = e^{-\mu_{K+1}} \quad \text{ and } \Lambda_{max}=1
\end{align}
and conclude
\begin{align}\label{e2}
\begin{split}
\left\|L_c^{-1}\right\|&\leq \Lambda_c^{-1}\quad \text{ or }\quad  \Lambda_c\left\|w\right\|_{H} \leq \left\|Lw\right\|_{H} \text{ for all } w\in E_c, \\
\left\|L_s\right\| &\leq \Lambda_s \quad\ \  \text{ or }\ \ \quad \left\|Lw\right\|_{H} \leq \Lambda_s\left\|w\right\|_{H} \text{ for all } w \in E_s,\\
\text{ and }\quad  \left\|L\right\| &\leq \Lambda_{max} \ \ \text{ or }\quad \ \ \left\|Lw\right\|_{H} \leq \left\|w\right\|_{H} \text{ for all }w\in H.
\end{split}
\end{align}
We arbitrarily choose $\Lambda_s<\Lambda_- = e^{-\mu_-} <\Lambda_c$ with 
$\mu_- < \mu_{K+1}<2\mu_-$  and $\Lambda_{max} < \Lambda_+$ 
and introduce the following norms, that will be used for the construction of the manifolds:
\begin{itemize}
	\item For $w \in H$ we define $\vertiii{w} \coloneqq \max \left\{ \left\|P_cw\right\|_{H}, \left\|P_sw\right\|_{H} \right\}$.\\
	\item For $\left\{w_k\right\}_{k\in \mathbb{Z}} \subseteq H$ we set $\left\|\left\{w_k\right\}_{k\in \mathbb{Z}}\right\|_{\Lambda_-,\Lambda_+}\coloneqq\sup \limits_{k\in \mathbb{N}_0} \max \left\{ \Lambda_+^{-k}\vertiii{w_k}, \Lambda_-^k\vertiii{w_{-k}}\right\}.$\\
	\item For $\left\{w_k\right\}_{k\in \mathbb{N}_0} \subseteq H$ we set $\left\|\left\{w_k\right\}_{k\in \mathbb{N}_0}\right\|_{\Lambda_-,+}\coloneqq\sup \limits_{k\in \mathbb{N}_0} \Lambda_-^{-k}\vertiii{w_k}.$
\end{itemize}
The corresponding Banach spaces of sequences are denoted by $\ell_{\Lambda_-,\Lambda_+}$ and $\ell_{\Lambda_-,+}$ respectively.

Our first result it the construction of the center manifold.

\begin{proposition}[Center Manifold]\label{centermanifold}
	Fix $\Lambda_- = e^{-\mu_-}$ in $\left( \Lambda_s,\Lambda_c \right)$. Let   $\varepsilon_{gap}>0$ such that
	\begin{align}\label{e4}
	\Lambda_s+\varepsilon_{gap} <\Lambda_-<\Lambda_c-\varepsilon_{gap}\quad \text{ and } \quad\Lambda_{max} +\varepsilon_{gap} < \Lambda_+.
	\end{align}
	Choose $\varepsilon\leq \varepsilon^*$ sufficiently small, such that
	\begin{align}\label{e1}
	\Lip\left(R_\varepsilon\right) \leq \varepsilon_{gap}.
	\end{align}
	(If necessary, choose $\varepsilon^0 \leq \min \left\{\varepsilon, \varepsilon_0\right\}$ even smaller according to Lemma \ref{smoothingestimate}.)
	Then there exists  a function $\theta_\varepsilon : E_c \rightarrow E_s$ with   $\theta_\varepsilon(0)=0$, that is differentiable at zero with $D\theta_\varepsilon(0)=0$, and the submanifold 
	\begin{align}
	W_\varepsilon^c \coloneqq  \left\{ w_c + \theta_\varepsilon\left(w_c\right) : w_c \in E_c \right\}
	\end{align}	
	satisfies the following conditions.
	\begin{enumerate}
		\item The function $\theta_\varepsilon$ is a contraction with $\Lip\left(\theta_\varepsilon\right)\lesssim \varepsilon_{gap}$ and $\left\|\theta_\varepsilon \left(g_c\right)\right\|_{H} \lesssim \left\|g_c\right\|_{H}^{\alpha}$  for all $g_c \in E_c$ for some $1<\alpha <\frac{\mu_{K+1}}{\mu_-}$. Moreover, it holds that $\left\|\theta_{\varepsilon}(g_c)\right\|_{W}\lesssim \vertiii{g_c}$.
		\item If the semiflow $\left\{ S_\varepsilon^t\right\}_{t\geq0}$ gets restricted to $W_\varepsilon^c$, it can be extended to an eternal Lipschitz flow on $W_\varepsilon^c$. More precisely, it holds that  $S_\varepsilon^t\left(W_\varepsilon^c\right)= W_\varepsilon^c$ for all $t\geq 0$ and for any $g \in W_\varepsilon^c$ there exists a semiflow $\left\{w(t)\right\}_{t\leq0}$ in $W_\varepsilon^c$ with $w(0)=g$.
		\item The manifold  $W_\varepsilon^c$ is characterized as follows: The point $g$ belongs to $W_\varepsilon^c$ if and only if there exists a flow $\left\{w(t)\right\}_{t\in \mathbb{R}}$ with $w(0)=g$ and 
		\begin{align}
		\left\|w(t)\right\|_{H} \leq 
		\begin{cases}
		\Lambda_+^t\vertiii{g}\quad  \text{ for all } t\geq 0 \\
		\Lambda_-^t\vertiii{g}\quad  \text{ for all }t \leq 0.
		\end{cases}
		\end{align}
	\end{enumerate}
\end{proposition}

The Lipschitz constants here and in the following are to be understood for a mappings from $H$ to $H$, if both are equipped with the $\vertiii{\cdot}$ norm.

\begin{proof}
	Our proof relies on the construction in \cite{Kochoncentermanifolds} in many parts. However, with regard to the subtle regularity issues we  have to modify the argument and need to establish additional properties. For this reason, we give here a self-contained presentation.
	
	First,  we note that thanks  to Lemma \ref{lipschitzR} by choosing $\eps$ sufficiently small, the Lipschitz condition \eqref{e1} on $R_{\eps}$ is realizable.
	We define $J:E_c\times \ell_{\Lambda_-,\Lambda_+} \rightarrow \ell_{\Lambda_-,\Lambda_+} $ by
	\begin{align}
	J_k\left(g_c, \left\{w_l\right\}_{l\in \mathbb{Z}}\right) = 
	\begin{cases}
	S_\varepsilon\left(w_{k-1}\right)&\text{ if } k\geq 1\\
	P_sS_\varepsilon\left(w_{-1}\right)+g_c& \text{ if }k=0\\
	P_sS_\varepsilon\left(w_{k-1}\right)+L_c^{-1}P_c \left(w_{k+1}-R_\varepsilon\left(w_k\right) \right)&\text{ if } k\leq -1.
	\end{cases}
	\end{align}
	This mapping is well defined, as we will show
	\begin{align}\label{e3}
	\left\|J\left(g_c, \left\{w_l\right\}_{l\in \mathbb{Z}}\right)\right\|_{\Lambda_-,\Lambda_+} \leq \max\left\{ \vertiii{g_c}, \kappa \left\|\left\{w_l\right\}_{l\in\mathbb{Z}}\right\|_{\Lambda_-,\Lambda_+}\right\},
	\end{align}
	with $\kappa \coloneqq \max \left\{ \frac{\Lambda_-+\varepsilon_{gap}}{\Lambda_c}, \frac{\Lambda_{max}+\varepsilon_{gap}}{\Lambda_+}, \frac{\Lambda_s+\varepsilon_{gap}}{\Lambda_-} \right\}.$ This quantity $\kappa$ is strictly smaller than one due to \eqref{e4}.
	To prove \eqref{e3} for positive times steps, $k\ge 1$, we compute with help of the triangle inequality and properties \eqref{e2} and \eqref{e1} of $L$ and $R_\varepsilon$ 
	\begin{align*}
	\Lambda_+^{-k}\vertiii{P_sS_\varepsilon\left(w_{k-1}\right)}& \leq \left(\Lambda_+^{-k} \vertiii{L_sP_sw_{k-1}}+\vertiii{P_sR_\varepsilon(w_{k-1})}\right)\\
	&\leq \Lambda_+^{-k}\left( \Lambda_s\vertiii{w_{k-1}}+\varepsilon_{gap} \vertiii{w_{k-1}} \right)
	\leq \frac{\Lambda_s+\varepsilon_{gap}}{\Lambda_+}\left\|\left\{w_l\right\}_{l\in \mathbb{Z}}\right\|_{\Lambda_-,\Lambda_+}.
	\end{align*}
We have a similar bound on the projection onto the center manifold,
	\begin{align}
	\Lambda_+^{-k}\vertiii{P_cS_\varepsilon(w_{k-1})}\leq \frac{\Lambda_{max}+\varepsilon_{gap}}{\Lambda_+}\left\|\left\{w_l\right\}_{l\in\mathbb{Z}}\right\|_{\Lambda_-,\Lambda_+}.
	\end{align}
	The bound for negative time steps, $k\le -1$, is verified in the same manner, namely
	\begin{align*}
	\mel
	\Lambda_-^k\vertiii{P_sS_\varepsilon\left(w_{k-1}\right)+L_c^{-1}P_c \left(w_{k+1}-R_\varepsilon\left(w_k\right) \right)}\\
	&\leq \max\left\{ \frac{\Lambda_s+\varepsilon_{gap}}{\Lambda_-},\frac{\Lambda-_+\varepsilon_{gap}}{\Lambda_c}\right\}\left\|\left\{{w_l}\right\}_{l\in \mathbb{Z}}\right\|_{\Lambda_-,\Lambda_+}.
	\end{align*}
	And finally, for $k=0$ the same strategy yields 
	\[
	\vertiii{P_sS_\varepsilon(w_{-1})+g_c} \leq \max\left\{ \frac{\Lambda_s+\varepsilon_{gap}}{\Lambda_-}\left\|\left\{w_l\right\}_{l\in\mathbb{Z}}\right\|_{\Lambda_-,\Lambda_+},  \vertiii{g_c}\right\},
	\]
	which completes the proof of \eqref{e3}.

	Making use of the inequalities \eqref{e2} and \eqref{e1} again, we derive similarly that $J(g_c,\cdot)$, for fixed $g_c \in E_c$, is a contraction on $\ell_{\Lambda_-,\Lambda_+}$, that is
	\begin{align}
	\left\|J_k\left(g_c, \left\{w_l\right\}_{l\in \mathbb{Z}}\right) - J_k\left(g_c, \left\{\tilde{w}_l\right\}_{l\in \mathbb{Z}}\right)\right\|_{\Lambda_-,\Lambda_+}\leq \kappa \left\|\left\{w_l\right\}_{l\in \mathbb{Z}} - \left\{\tilde{w}_l\right\}_{l\in \mathbb{Z}}\right\|_{\Lambda_-,\Lambda_+},
	\end{align}
	for every $\left\{w_l\right\}_{l\in \mathbb{Z}}$, $\left\{\tilde{w}_l\right\}_{l\in \mathbb{Z}}$ in $\ell_{\Lambda_-,\Lambda_+}$.
	Hence, by Banach's fixed point theorem, for every element $g_c \in E_c$ there exists a unique sequence $\left\{w_k\right\}_{k\in \mathbb{Z}} \in \ell_{\Lambda_-,\Lambda_+}$ with $J\left(g_c,\left\{w_k\right\}_{k\in \mathbb{Z}}\right) = \left\{w_k\right\}_{k\in \mathbb{Z}}$. By construction this fixed point sequence is a solution to the discrete semiflow with $P_cw_0 = g_c$. By the virtue of  \eqref{e3}, we also know that $\left\|\left\{w_k\right\}_{k\in \mathbb{Z}}\right\|_{\Lambda_-,\Lambda_+}\leq \vertiii{g_c}$.
	
	Now, we define the solution mapping $\hat{\theta}_\varepsilon : E_c \rightarrow \ell_{\Lambda_-,\Lambda_+}$ by $\hat{\theta}_\varepsilon\left(g_c\right) = \left\{w_k\right\}_{k\in \mathbb{Z}}$ and consider $\theta_\varepsilon:E_c\rightarrow E_s$ given by $\theta_\varepsilon\left(g_c\right)=  P_sw_0$. In other words, the initial datum of the solution sequence decomposes into $w_0=g_c+\theta_{\varepsilon}\left(g_c\right)$. 	Since $J(0,0)=0$, we obtain, by the uniqueness of the fixed point, that $\hat{\theta}_\varepsilon(0)=0$ and thus $\theta_\varepsilon(0)=0$.

	The contraction property, in particular, entails that the solution mapping  $\hat{\theta}_\varepsilon$ is Lipschitz continuous with bound  $\Lip\left(\hat{\theta}_\varepsilon\right)\leq \frac{1}{1-\kappa}$. Thus, also its ``coordinate'' $\theta_\varepsilon$ is Lipschitz continuous with the same bound. We will need to a stronger bound, in fact, a contraction estimate.  For any $g_c$ and $\tilde{g}_c \in E_c$ we have
	\begin{align}
	\vertiii{\theta_\varepsilon\left(g_c\right)-\theta_\varepsilon\left(\tilde{g}_c\right)} =\vertiii{P_s\left( S_\varepsilon\left(w_{-1}\right)-S_\varepsilon\left(\tilde{w}_{-1}\right) \right)},
	\end{align}
	where $\left\{w_k\right\}_{k\in \mathbb{Z}} = \hat{\theta}_\varepsilon\left(g_c\right)$ and $\left\{\tilde{w}_k\right\}_{k\in \mathbb{Z}} = \hat{\theta}_\varepsilon\left(\tilde{g}_c\right)$.
	Using the triangle inequality and the properties of $L$ and $R_\varepsilon$, we get
	for any $k\ge 0$ that 
	\begin{align*}
	\mel	 \Lambda_-^k\vertiii{P_s\left(w_{-k}-\tilde{w}_{-k}\right)}\\
	& \leq \frac{\Lambda_s}{\Lambda_-}\Lambda_-^{k+1}\vertiii{P_s\left(w_{-(k+1)}-\tilde{w}_{-(k+1)}\right)} + \frac{\varepsilon_{gap}}{\Lambda_-}\Lambda_-^{k+1}\vertiii{w_{-(k+1)}-\tilde{w}_{-(k+1)}}.
	\end{align*}
	Applying this inequality iteratively, we obtain
	\begin{align*}
	\vertiii{\theta_\varepsilon\left(g_c\right)-\theta_\varepsilon\left(\tilde{g}_c\right)} 
	& = \vertiii{P_s(w_0-\tilde w_0)}\\
	&\leq \left(\frac{\Lambda_s}{\Lambda_-}\right)^m\left\|\left\{w_k\right\}_{k\in \mathbb{Z}}- \left\{\tilde{w}_k\right\}_{k\in \mathbb{Z}}\right\|_{\Lambda_-,\Lambda_+}\\
	&\quad + \frac{\varepsilon_{gap}}{\Lambda_-}\sum \limits_{l=0}^{m-1}\left(\frac{\Lambda_s}{\Lambda_-}\right)^l\left\|\left\{w_k\right\}_{k\in \mathbb{Z}}- \left\{\tilde{w}_k\right\}_{k\in \mathbb{Z}}\right\|_{\Lambda_-,\Lambda_+}
	\end{align*} 
	for every $m\in \mathbb{N}$. Sending $m$ to infinity and using the Lipschitz bound for $\hat \theta_{\eps}$ yields
	\begin{align}
	\vertiii{\theta_\varepsilon\left(g_c\right)-\theta_\varepsilon\left(\tilde{g}_c\right)} \leq  \frac{\varepsilon_{gap}}{\Lambda_--\Lambda_s} \left\|\hat\theta_{\eps}(g_c)-\hat \theta_{\eps}(\tilde g_c)\right\|_{\Lambda_-,\Lambda_+}\le \frac{\varepsilon_{gap}}{\Lambda_--\Lambda_s}\frac{1}{\kappa-1} \vertiii{g_c-\tilde{g}_c}.
	\end{align}
	This proves that $\theta_{\eps}$ is Lipschitz with constant $\Lip (\theta_{\eps})\lesssim \varepsilon_{gap}$.
	
	We continue by deriving the superlinear behavior of $\theta_\varepsilon$ near zero, which eventually implies the differentiability properties stated in the the proposition. We compute, using the quadratic bound on $R_{\eps}$ in Lemma \ref{quadratischeabschaetzung}
\[
	\vertiii{\theta_\varepsilon\left(g_c\right)} = \vertiii{P_sw_0} \leq \vertiii{P_s R_\varepsilon\left(S_\varepsilon\left(w_{-2}\right)\right)} + \vertiii{P_sLw_{-1}} \leq C \vertiii{w_{-2}}^2 + \Lambda_s\vertiii{P_sw_{-1}}.
\]
	Similarly, we get $\vertiii{P_sw_{-k}} \leq C \vertiii{w_{-(k+2)}}^2 + \Lambda_s \vertiii{P_sw_{-(k+1)}}$
	for any $k\in \mathbb{N}_0$ and thus, for any $m \in \mathbb{N}$,
	\begin{align}
	\vertiii{\theta_\varepsilon\left(g_c\right)} \le  \Lambda_s^m \vertiii{w_{-m}} + C\sum \limits_{l=1}^{m}\Lambda_s^{l-1}\vertiii{w_{-(l+1)}}^2.
	\end{align}
	Recalling the definition of $\left\|\cdot\right\|_{\Lambda_-,\Lambda_+}$ and the fact that the solution sequence is bounded via \eqref{e3}, $\left\|\left\{w_k\right\}_{k\in\mathbb{Z}}\right\|_{\Lambda_-,\Lambda_+}\leq \vertiii{g_c}$, we obtain
	\begin{align}\mel 
	\Lambda_s^m \vertiii{w_{-m}} + \sum \limits_{l=1}^{m}\Lambda_s^{l-1}\vertiii{w_{-(l+1)}}^2 \\
	& \leq \left( \frac{\Lambda_s}{\Lambda_-} \right)^m \vertiii{g_c} + \frac{C}{\Lambda_s\Lambda_-^2}\sum \limits_{l=1}^m \frac{\Lambda_s^l}{\Lambda_-^{2l}} \vertiii{g_c}^2 \\
	&= \left( \frac{\Lambda_s}{\Lambda_-} \right)^m \vertiii{g_c} + \frac{C}{\Lambda_s\Lambda_-^2}\sum \limits_{l=1}^m \left(\frac{\Lambda_-}{\Lambda_s}\right)^{lk}\left(\frac{\Lambda_s^{k+1}}{\Lambda_-^{k+2}}\right)^l\vertiii{g_c}^2
	\end{align}
	for any $k\in \mathbb{N}$. We recall that $\Lambda_->\Lambda_s$. Hence, if  there exists a $k \in \mathbb{N}$, such that $\Lambda_-^{k+2}> \Lambda_s^{k+1}$, it holds that
	\begin{align}
	\vertiii{\theta_\varepsilon\left(g_c\right)} \le  \left( \frac{\Lambda_s}{\Lambda_-} \right)^m \vertiii{g_c} + \frac{C}{\Lambda_s\Lambda_-^2}\left(\frac{\Lambda_-}{\Lambda_s}\right)^{km}\frac{\Lambda_s^{k+1}}{\Lambda_-^{k+2} - \Lambda_s^{k+1}}\vertiii{g_c}^2,
	\end{align}
	and after optimizing in $m$, this becomes
	\begin{align}
	\vertiii{\theta_\varepsilon\left(g_c\right)} \lesssim \vertiii{g_c}^{1+\frac{1}{k+1}},
	\end{align}
	provided that the right-hand side is sufficiently small. (For larger $g_c$, this bound follows trivially from the linear estimate.) 
	It remains to verify the existence of a suitable $k$. This, however, follows easily from our choice of $\Lambda_-$, more precisely, from the assumption  $\mu_- < \mu_{K+1}<2\mu_-$. Indeed, the latter enables us to pick $k > \frac{2\mu_--\mu_{K+1}}{\mu_{K+1}-\mu_-}$, which implies $\Lambda_-^{k+2}> \Lambda_s^{k+1}$ as desired. This proves the first statement with $\alpha = 1+ \frac{1}{k+1}< \frac{\mu_{K+1}}{\mu_-}$.
	
	We turn to the last inequality of the first statement. By the definition of $\theta_{\eps}$, the construction of the fixed point and the smoothing estimate from Lemma \ref{smoothingestimate},  we have that
	\begin{align}
	\left\|\theta_{\varepsilon}\left(g_c\right)\right\|_{W} \leq \left\|S_\varepsilon\left(w_{-1}\right)\right\|_{W} \lesssim \left\|w_{-1}\right\|_{H}.
	\end{align}
	It remains to notice that $\left\|w_{-1}\right\|_{H}\lesssim \vertiii{w_{-1}}\leq \Lambda_-^{-1}\vertiii{g_c}\lesssim \vertiii{g_c}$ by the equivalence of the norms and the bound \eqref{e3} applied to the solution sequence.
	
	The second part of the proof covers the properties of the center manifold 	$W_\varepsilon^c$ which is defined as the graph of $\theta_\varepsilon$.
	We commence with the invariance of $W_\varepsilon^c$. For this we consider an arbitrary point on that manifold $g=g_c+\theta_\varepsilon\left(g_c\right)$ and consider the evolution $\left\{w_k\right\}_{k\in \mathbb{Z}} = S^k_\varepsilon(g) = \hat{\theta}_{\varepsilon}(g)$ starting at that point. We have to show that  for every time step $k\in \mathbb{Z}$, the solution $w_k$ lies in $W_\varepsilon^c$, or, equivalently, that $P_sw_k=\theta_{\varepsilon}\left(P_cw_k\right)$. By iteration, it suffices to show this only for $k=1$ and $k=-1$. We set $\tilde{g}_c=P_cw_1$. Then  $S^k_\varepsilon\left(\tilde{w}_0\right) = \hat{\theta}_\varepsilon\left(\tilde{g}_c\right)$ is the unique flow in $\ell_{\Lambda_-,\Lambda_+}$ that satisfies $P_c\tilde{w}_0=\tilde{g}_c$. Since $P_cw_1 = \tilde{g}_c$, we have by uniqueness that $w_{k+1}=\tilde{w}_k$ for every $k\in \mathbb{Z}$. This yields $P_sw_1 = P_s\tilde{w}_0= \theta_{\varepsilon}\left(\tilde{g}_c\right)=\theta_{\varepsilon}\left(P_cw_1\right)$. 
	The same procedure backwards in time yields the statement for $k=-1$.
	
	It remains to prove the characterization of the center manifold. First, for a point $w_0$ on that manifold,   i.e., $w_0=g_c+\theta_{\varepsilon}\left(g_c\right)$ for some $g_c\in E_c$, we already know that $\left\|\left\{S^k_\varepsilon\left(w_0\right)\right\}\right\|_{\Lambda_-,\Lambda_+} = \left\|\hat{\theta}_\varepsilon\left(g_c\right)\right\|_{\Lambda_-,\Lambda_+} \leq \vertiii{g_c}\leq\vertiii{g}$ by the virtue of \eqref{e3}. Otherwise, if a flow $\left\{w_k\right\}_{k\in \mathbb{Z}} = \left\{S^k_\varepsilon\left(w_0\right)\right\}_{k\in\mathbb{Z}}$ satisfies this bound, it must be a fixed point of $J\left(P_cw_0,\cdot\right)$. Since this fixed point is unique, we have $\hat{\theta}_\varepsilon\left(P_cw_0\right)=\left\{w_k\right\}_{k\in \mathbb{Z}}$ and thus $\theta_{\varepsilon}\left(P_cw_0\right) = P_sw_0$. This yields $w_0\in W_\varepsilon^c$.
	
\end{proof}

The regularity of $\theta_{\varepsilon}$ allows us to deduce the equivalence of the Hilbert space norm $\vertiii{\cdot}$  and the higher-order norm $\|\cdot \|_W$ on the finite-dimensional manifold $W_{\eps}^c$.

\begin{corollary}\label{equivalenceofnorms}
	The norms	 $\vertiii{g}$ and $\|g\|_{W}$  are equivalent\ for any $g\in W_\varepsilon^c$.
\end{corollary}
\begin{proof}
	Trivially, the embedding $W\hookrightarrow H$ is continuous on a bounded domain, that is, $\vertiii{g}\lesssim \|g\|_{W}$ for every $g \in W$.
	To show the reverse inequality, we take an element $g = g_c +\theta_{\eps}(g_c)$ in $W_\varepsilon^c$. Now, we notice that on the one hand, thanks to the regularity of $\theta_{\varepsilon}$ established in Proposition \ref{centermanifold}, we have $\left\|\theta_{\varepsilon}\left(g_c\right)\right\|_{W} \lesssim \vertiii{g_c}$. On the other hand, because $E_c$ is a finite-dimensional space, all norms on $E_c$ are equivalent, so that $\|g_c\|_W\lesssim \vertiii{g_c}$. We combine both insights and find
	\[
	\|g\|_W \le \|g_c\|_W + \|\theta_{\eps}(g_c)\|_W \lesssim \vertiii{g_c} \le \vertiii{g},
	\]
as desired.
\end{proof}

We will now construct the stable manifolds.

\begin{proposition}[Stable Manifold]\label{stablemanifold}
	Let   $\varepsilon_{gap}>0$ and $\eps$ be as in Proposition \ref{centermanifold} such that \eqref{e4} and \eqref{e1} hold.
	Then for every $g\in H$, there exists a map $\nu^\varepsilon_g:E_s\rightarrow E_c$ such that the submanifold
	\begin{align}
	M_g^\varepsilon \coloneqq g + \left\{ \nu_g^\varepsilon\left(g_s\right)+g_s : g_s\in E_s\right\}
	\end{align}
	satisfies the following conditions.
	\begin{enumerate}
		\item For every $g\in H$, the map $\nu_g^\varepsilon :E_s\rightarrow E_c$ is Lipschitz continuous with $\Lip\left(\nu_g^\varepsilon\right)\lesssim \varepsilon_{gap}$.
		\item For every $t\geq 0$ it holds that $S^t_\varepsilon\left(M_g^\varepsilon\right) \subseteq M_{S^t_\varepsilon(g)}^\varepsilon$ and $M_g^\varepsilon$ can be characterized as follows
		\begin{align}
		M_g^\varepsilon = \left\{ \tilde{g}\in H : \sup \limits_{k\in \mathbb{N}_0} \Lambda_-^{-k}\vertiii{S^k_\varepsilon\left(g\right)- S^k_\varepsilon\left(\tilde{g}\right)} \leq \vertiii{P_s\left(g-\tilde{g}\right)} \right\}
		\end{align}
		\item If $\eps_{gap}$ is sufficiently small (and $\varepsilon^0$ chosen accordingly), the following holds true: For every $g\in H$ the intersection $M_g^\varepsilon\cap W_\varepsilon^c$ consists of a single point $\tilde{g}$. This particularly yields that $\left\{M_g^\varepsilon\right\}_{g\in H}$ is a foliation of $H$ over $W_\varepsilon^c$. Moreover, it holds that
		\begin{align}
		\left\|\tilde{g}\right\|_{W} \lesssim \vertiii{g}.
		\end{align}
	\end{enumerate}
\end{proposition}
\begin{proof} The existence follows again by a fixed point argument, which is similar to the one of Proposition \ref{centermanifold}. We will thus only sketch it.
	
	We fix a function $g\in H$ and a positive constant $r$ and define 
	\[
	\ell_{\Lambda_-,+}^{g,r}\coloneqq \left\{ \left\{w_l\right\}_{l\in \mathbb{N}_0}\in \left(L^2(\rho)\right)^{\mathbb{N}_0}: \left\|\left\{w_l\right\}_{l\in\mathbb{N}_0}-\left\{S_\varepsilon^l(g)\right\}_{l\in \mathbb{N}_0}\right\|_{\Lambda_{-},+}\le r\right\}.
	\]
	 Note that $\ell_{\Lambda_-,+}^{g,r}$ equipped with the metric 
	 \[
	 d_g\left(\left\{w_l\right\}_{l\in \mathbb{N}_0},\left\{\tilde{w}_l\right\}_{l\in \mathbb{N}_0}\right) = \left\|\left\{w_l\right\}_{l\in \mathbb{N}_0}-\left\{\tilde{w}_l\right\}_{l\in \mathbb{N}_0}\right\|_{\Lambda_{-},+}
	 \]
	  is a closed subset of $\left(L^2(\rho)\right)^{\mathbb{N}_0}$.
	We consider the map $I^g:E_s \times \ell_{\Lambda_-,+}^{g,r} \rightarrow \ell_{\Lambda_-,+}^{g,r}$ defined by
	\begin{align}
	I^g_k\left( g_s,\left\{w_l\right\}_{l\in \mathbb{N}_0} \right) =
	\begin{cases}
	g_s+P_sg+L_c^{-1}P_c\left( w_1-R_\varepsilon\left(w_0\right) \right)& \text{ if } k=0\\
	P_s\left( S_\varepsilon\left(w_{k-1}\right)+L_c^{-1}P_c\left(w_{k+1}-R_\varepsilon\left(w_k\right) \right) \right)& \text{ if } k\geq 1,
	\end{cases}		
	\end{align}
which has the following useful property
	\begin{align}\label{e6}
	I^g_k\left( g_s, \left\{S_\varepsilon^l(g)\right\}_{l\in \mathbb{N}_0}\right) - S_\varepsilon^k(g) = g_s\delta_{0k}.
	\end{align}
	Moreover, by similar arguments as for the operator $J$ in the proof of Proposition \ref{centermanifold}, relying on \eqref{e2} and \eqref{e1} we compute for a fixed element $g_s\in E_s$ that
\[
	\left\|I^g\left(g_s, \left\{ w_l \right\}_{l\in\mathbb{N}_0}\right)-I^g\left(g_s, \left\{\tilde{w}_l\right\}_{l\in \mathbb{N}_0} \right)\right\|_{\Lambda_-,+} \leq \kappa \left\|\left\{ w_l \right\}_{l\in\mathbb{N}_0}-\left\{\tilde{w}_l\right\}_{l\in \mathbb{N}_0}\right\|_{\Lambda_-,+}
	\]
and
\begin{align*}
\mel
\left\|I^g\left(g_s,\left\{w_l\right\}_{l\in\mathbb{N}_0}\right)-\left\{S^l_\varepsilon(g)\right\}\right\|_{\Lambda_-,+} \\
&
	\leq \max\left\{\vertiii{g_s},\kappa\left\|\left\{w_l\right\}_{l\in\mathbb{N}_0}-\left\{S^l_\varepsilon(g)\right\}\right\|_{\Lambda_-,+}\right\},
\end{align*}
	where $\kappa = \max \left\{ \frac{\Lambda_- + \varepsilon_{gap}}{\Lambda_c}, \frac{\Lambda_s+\varepsilon_{gap}}{\Lambda_-}\right\}<1$. Notice that in the latter estimate, we made use of the formula \eqref{e6}.
	 Both estimates imply $I^g(g_s,\cdot)$ is a contraction and a self-mapping on the set $ \ell_{\Lambda_-,+}^{g,r}$, if we choose  $r=\vertiii{g_s}$.
	
	Hence, by Banach's fixed point theorem there exists a unique sequence $\left\{w_k\right\}_{k\in \mathbb{N}_0}$ satisfying
	\begin{align}
	I^g\left(g_s, \left\{ w_k \right\}_{k\in\mathbb{N}_0}\right)= \left\{w_k\right\}_{k\in \mathbb{N}_0} \quad \text{ and } \quad \left\|\left\{w_k\right\}_{k\in \mathbb{N}_0} - \left\{S_\varepsilon^k(g)\right\}_{k\in \mathbb{N}_0}\right\|_{\Lambda_-,+}\leq r.
	\end{align}
	By construction, this sequence $\left\{w_k\right\}_{k\in \mathbb{N}_0}$ is a semiflow to the truncated equation with $P_sw_0= g_s+P_sg$. We may now introduce  a solution mapping $\hat{\nu}^\varepsilon_g:E_s \rightarrow \ell_{\Lambda_-,+}^g$ by $\hat{\nu}_g^\varepsilon\left(g_s\right) \coloneqq \left\{w_k\right\}_{k\in \mathbb{N}_0}$, and we define
	$\nu_g^\varepsilon\left(g_s\right)=P_c\left(w_0-g\right)$.
	Due to the construction via a fixpoint argument, we deduce that $\hat{\nu}_g^\varepsilon$ is Lipschitz continuous with $\Lip\left(\hat{\nu}^\varepsilon_g\right)\leq \frac{1}{1-\kappa}$. 
	
	We will improve the Lipschitz constant in a similar way as in the previous proof.
	For this, let $\hat{\nu}_g^\varepsilon\left(g_s\right) =  \left\{w_k\right\}_{k\in\mathbb{N}_0}$ and $\hat{\nu}_g^\varepsilon\left(\tilde{g_s}\right) = \left\{ \tilde{w}_k\right\}_{k\in\mathbb{N}_0}$ be two fixed point solution sequences. It holds that $ \nu_g^\varepsilon\left(g_s\right)-\nu_g^\varepsilon\left(\tilde{g_s}\right)  =  P_c\left(w_0- \tilde{w}_0\right) $, and we compute
	\[
	\vertiii{P_c\left(w_k-\tilde{w}_k\right)} \leq\frac{1}{\Lambda_c}\vertiii{P_c\left(w_{k+1}-\tilde{w}_{k+1}\right)} + \frac{\varepsilon_{gap}}{\Lambda_c} \vertiii{w_k-\tilde{w}_k}
	\]
	with the help of the definition of the map $I$. Therefore, for every $m\in \mathbb{N}$ it holds
	\begin{align}
	\vertiii{\nu_g^\varepsilon\left(g_s\right)-\nu_g^\varepsilon\left(\tilde{g_s}\right)}
	&\leq \left(\frac{\Lambda_-}{\Lambda_c}\right)^m \left\|\left\{w_k\right\}_{k\in\mathbb{N}_0}-\left\{\tilde{w}_k\right\}_{k\in\mathbb{N}_0}\right\|_{\Lambda_-,+} \\
	&+ \frac{\varepsilon_{gap}}{\Lambda_c}\sum \limits_{l=0}^{m-1} \left(\frac{\Lambda_-}{\Lambda_c}\right)^l \left\|\left\{w_k\right\}_{k\in\mathbb{N}_0}-\left\{\tilde{w}_k\right\}_{k\in\mathbb{N}_0}\right\|_{\Lambda_-,+}.
	\end{align}
	Since $\frac{\Lambda_-}{\Lambda_c}<1$, for $k\rightarrow \infty$, this yields 
	\begin{align}
	\vertiii{\nu_g^\varepsilon\left(g_s\right)-\nu_g^\varepsilon\left(\tilde{g_s}\right)} \leq  \frac{\varepsilon_{gap}}{\left(\Lambda_c-\Lambda_-\right)\left(1-\kappa\right)}\vertiii{g_s-\tilde{g_s}}.
	\end{align}
	
	The stable manifold $M_g^\varepsilon$ is defined as the graph of $\nu_g^\varepsilon$ shifted by $g$. We first prove its characterization as stated in the second part of the proposition.
	Let $\tilde{g}$ be in $M_g^\varepsilon$, i.e., $\tilde{g}= g+\nu_g^\varepsilon\left(g_s\right)+g_s$ for some $g_s \in E_s$. We define $\left\{w_k\right\}_{k\in \mathbb{N}_0}=\hat{\nu}_g^\varepsilon\left(\tilde{g}\right)$ as the unique semi flow with $\Lambda_-^{-k}\vertiii{w_k-S^k_\varepsilon(g)}\leq \vertiii{g_s}=\vertiii{P_s\left(g-\tilde{g}\right)}$ and $P_sw_0=g_s+P_sg$. By definition of $\nu_g^\varepsilon$, we have
	\begin{align}
	\tilde{g}=g+\nu_g^\varepsilon\left(g_s\right)+g_s = g+P_c\left(w_0-g\right)+P_sw_0-P_sg =w_0,
	\end{align}
	and thus, $w_k= S^k_\varepsilon\left(\tilde{g}\right) $ satisfies the desired bound. Let us now assume that $S^k_\varepsilon\left(\tilde{g}\right)$ satisfies this bound. We define $g_s=P_s\left( \tilde{g}-g\right)$. Then $S^k_\varepsilon\left(\tilde{g}\right)$ is the unique fixpoint of $I^g\left(g_s,\cdot\right)$ with $\left\|S^k_\varepsilon\left(\tilde{g}\right)-S^k_\varepsilon(g)\right\|_{\Lambda_-,+}\leq \vertiii{g_s}$. By definition, this yields $S^k_\varepsilon\left(\tilde{g}\right)= \hat{\nu}_g^\varepsilon\left(g_s\right)$ and $\nu_g^\varepsilon\left(g_s\right)=P_c\left(\tilde{g}-g\right)$ and thus
	\begin{align}
	g+\nu_g^\varepsilon\left(g_s\right)+g_s= g + P_c\left(\tilde{g}-g\right) + P_s\left(\tilde{g}-g\right)=\tilde{g}.
	\end{align}
	
	Next, we have to verify  that $M_g^\varepsilon$ is positive invariant. For this, we take an arbitrary point $w_0$ in $M_g^\varepsilon$ and define $\tilde{w}_0=S_\varepsilon(w_0)$. We straightforwardly compute that $S^k_\varepsilon\left(\tilde{w}_0\right)$ is a fixpoint of $I^{S_\varepsilon(w_0)}\left(0, \cdot\right)$, which implies the desired property.

	To prove that there exists a single intersection point with the center manifold $W_{\eps}^c$, we consider the mapping $\chi(g_s) = \theta_{\eps}(\nu^{\eps}_g(g_s-P_sg)+P_cg)$ on $E_s$. Since $\theta_{\eps}$ and $\nu_g^{\eps}$ are both Lipschitz continuous with constant of order $\eps_{gap}$, the mapping $\chi$ itself is Lipschitz with a constant of the order $\eps_{gap}^2$, and thus, it is  a contraction if $\eps_{gap}$ is sufficiently small. We denote by $\tilde g_s$ the unique fixed point and set $\tilde g_c = \nu_g^{\eps}(\tilde g_s - P_sg)+P_cg$. By definition, $\tilde g = \tilde g_c+\tilde g_s$ lies in the intersection of $W^c_{\eps} $ and $M_g^{\eps}$. As every point in this intersection is itself a fixed point, the uniqueness follows.
	
	To estimate the intersection point $\tilde g$ against $g$, we argue similarly. Indeed, by construction, the Lipschitz property for $\nu_g^{\eps}$, and the fact that both $\theta_{\eps}(0)=0$ and $\nu_g^{\eps}(0)=0$,  it holds that
	\begin{align*}
	\vertiii{\tilde g} & = \vertiii{\nu_g^{\eps} (\tilde g_s -P_s g) +P_c g + i_s \theta_{\eps}(\nu_g^{\eps} (\tilde g_s -r_s g) +P_c g)}\\
	&\lesssim (1+\eps_{gap}) \vertiii{\nu_g^{\eps} (\tilde g_s -r_s g) +P_c g}\\
	&\lesssim \eps_{gap}  \vertiii{\tilde g_s -P_sg}  +  \vertiii{g}\\
	&\lesssim \eps_{gap} \vertiii{\tilde g} + \vertiii{g},
	\end{align*}
	where we have used that $\eps_{gap}\le 1$ in the third inequality. We arrive at
		\[
	\vertiii{\tilde g} \lesssim \vertiii{g},
	\]
	provided that $\eps_{gap}$ is sufficiently small.
	Because $\tilde{g}$ lies on the manifold $W_\varepsilon^c$, we can make use of Corollary \ref{equivalenceofnorms} to obtain $\left\|\tilde{g}\right\|_{W} \lesssim\vertiii{g}$.  
\end{proof}

Finally we are able to show the existence of a localized invariant manifold as claimed in Theorem \ref{localmanifolds} by combining the two preceding constructions with earlier proved regularity properties of the flow map $S$.

\begin{proof}[ of Theorem \ref{localmanifolds}]
	We choose $0<\varepsilon_{gap}<\min \left\{e^{-\mu_{K+1}}-e^{-\mu},  e^{-\mu}-e^{-\mu_K}\right\}$, such that the third statement in Proposition \ref{stablemanifold} applies, and we define $\Lambda_-=e^{-\mu}$. We furthermore pick $\varepsilon\leq \varepsilon^*$ and $\varepsilon^0 \leq \min \left\{\varepsilon,\varepsilon_0\right\}$ as in the hypotheses of Propositions \ref{centermanifold} and \ref{stablemanifold}.
	The construction of $W_{loc}^c$  then  follows directly from Proposition \ref{centermanifold}.
	
	To prove the first property in the theorem, we consider $g\in W_{loc}^c$ with $\|g\|_H\le\eps_0$ and we notice that by the semi-flow property from Theorem \ref{wellposednessH1}, it holds  $\left\|S^t_\varepsilon(g)\right\|_{L^\infty\left(\left(0,\infty\right);H\right)} \leq \tilde{C} \eps_0$ for some $\tilde C\geq 1$. Moreover, $S^t_\varepsilon(g)\in W_{\eps}^c$ by construction, and thus, by the equivalence of norms in Corollary \ref{equivalenceofnorms}, it holds that $\|S^t_\varepsilon(g)\|_{W}\leq C \|S^t_\varepsilon(g)\|_{H}   \le C\tilde C \eps_0$ for some $C\geq 1$.
 Thus, for $\varepsilon_0\leq \frac{1}{C\tilde{C}}\varepsilon$, we find $S^t(g) = S^t_{\eps}(g)$ by the definition of the truncation and, in particular, $\|S^t(g)\|_H\le \eps$, for any $t\ge0$.

We turn to the proof of the second property. We know that there exists a unique point $\tilde{g}$ in $W_\varepsilon^c\cap M_g^\varepsilon$ that satisfies $\left\|\tilde{g}\right\|_{H}\lesssim \left\|\tilde{g}\right\|_{W} \lesssim \left\|g\right\|_{H}\lesssim  \|g\|_{W}\leq \varepsilon_0$, see Proposition \ref{stablemanifold}. In particular, choosing $\varepsilon_0 \leq \varepsilon$ even smaller, if necessary, it holds that $S^k_{\eps}(g) = S^k(g)$ and $S^k_{\eps}(\tilde g) = S^k(\tilde g)$. Moreover, the estimate shows that $\tilde{g}$ actually lies in $W_c^{loc}$. Now, the characterization of the stable manifold yields
	\begin{align}
	\left\|S^k(g)-S^k\left(\tilde{g}\right)\right\|_{H} \lesssim \Lambda_-^k.
	\end{align}
	Since we are allowed to drop the $\varepsilon$ at $S^t_\varepsilon(g)$ and $S^t_\varepsilon\left(\tilde{g}\right)$, and 
	since the solution to the (truncated) equation depends continuously on the initial datum with respect to the Hilbert space topology,  $\left\|S^t_\varepsilon(g)-S^t_\varepsilon\left(\tilde{g}\right)\right\|_H \lesssim \left\|g-\tilde{g}\right\|_H$
	 holds for all $t\in [0,1]$ (see  the fixed point construction of solutions in Theorem \ref{wellposednessH1}), we obtain 
	\begin{align}
		\left\|S^t(g)-S^t\left(\tilde{g}\right)\right\|_H\lesssim e^{-\mu t}
	\end{align}
	for any $t\geq 0$.
	Next, we make use of Lemma \ref{Lipschitzglaettung} and obtain
	\[
	\left\|S^{t}\left(g\right)-S^{t}\left(\tilde{g}\right)\right\|_{W} 
	\lesssim  e^{-\mu t},	
	\]
	for any $t\geq \hat{t}$ and some $\hat t\in(4/5,1)$. The statement follows.
\end{proof}

\section{Mode-by-mode asymptotics for the perturbation equation}\label{applicationinvariantmanifolds}

In this final section, we exploit our invariant manifold theorem, Theorem \ref{localmanifolds} to prove the mode-by-mode asymptotics in Theorem \ref{Whoeheremoden}.  We start with a brief comment  on the projection of a function $w \in H$ onto the subspaces spanned by the eigenfunctions of $\mathcal{L}^2+N\mathcal{L}$. Let $\psi$ be such an eigenfunction for the eigenvalue $\lambda^2 + N\lambda$, or, equivalently, $\L\psi =\lambda\psi$. We consider the $H$-projection of $w$, and find via an integration by parts
\begin{align}
\langle \psi,w\rangle_{H}&= \int \psi w \rho dz + \int \nabla\psi \cdot \nabla w \rho^2dz \\
&= \int \psi w \rho dz + \int w\mathcal{L}\psi  \rho dz = \left(1+\lambda\right)\int \psi w \rho dz = (1+\lambda)\langle \psi,w\rangle.
\end{align}
This shows that the $H$-projection coincides, up to a constant, with the $L^2(\rho)$-projection, due to the right choice of the weights. Thus, it is enough to  consider the projection with respect to $\langle \cdot,\cdot\rangle$ in the following.

We notice that  the projection of $w$ onto the space spanned by the constant eigenfunction corresponding to the eigenvalue $\mu_0=0$ is given by
\begin{align}
	P_0w= c_{0,N}\int_{B_1(0)} w\rho dz
\end{align}
and the projection $w$ onto the eigenspaces spanned by the eigenfunctions corresponding to the next eigenvalue $\mu_1$ is given by 
\begin{align}
	P_1w = c_{1,N}\int_{B_1(0)} zw\rho dz,
\end{align}
where $c_{0,N}$ and $c_{1,N}$ are two positive constants. 

Eventually we will prove Theorem \ref{Whoeheremoden} by induction and thus commence by proving the case $K=0$ in the following theorem. We remark that thanks to smoothing effects, see Equation (54) in \cite{SeisTFE}, it holds that
\[
\|w(t)\|_{W}\le \|w_0\|_{W^{1,\infty}},
\]
for some $t\gtrsim 1$, and thus, instead of considering Lipschitz initial data, we may impose slightly stronger assumptions.

\begin{theorem}\label{stabilityw}
	There exists $\varepsilon_0>0$ such that the following holds. Let $w$ be a solution to \eqref{perturbationequation} with initial datum $w_0$. We further assume that  $\|w_0\|_{W} \leq \varepsilon_0$ and
	\begin{align}\label{h11}
	\lim\limits_{t\rightarrow \infty} \int w(z)\rho(z)dz=0.
	\end{align} Then we have
	\begin{align}
	\left\|w(t)\right\|_{W} \lesssim e^{-\mu_{1}t} \quad \text{ for all } t\geq 0.
	\end{align}
\end{theorem}
\begin{proof}
	We will make use of the invariant manifolds we just constructed in the case $K=0$. In this case, $E_c$ is one-dimensional and spanned by the constant eigenfunction $\psi_{1,0}$ corresponding to the eigenvalue $\mu_0=0$. Thus, we obtain $E_c \cong \mathbb{R}$.
	We fix $\mu \in (0,\mu_{1})$ and accordingly $\varepsilon$ and $\varepsilon_0$ as in Theorem \ref{localmanifolds} and claim the equality 
	\begin{align}\label{claim}
	W_{\varepsilon}^c= E_c.
	\end{align}
	To see this, we first pick a function $g \in E_c$, i.e., $g(x)= \alpha \in \mathbb{R}$. The constant function $w(t,x)\equiv \alpha$ solves equation \eqref{c3} with initial datum $g$ and satisfies the bounds
	\begin{align}
	\|w(t)\|\sim |\alpha| \lesssim 
	\begin{cases}
	\Lambda_+^t |\alpha|, \quad \text{ for } t\geq 0,\\
	\Lambda_-^t |\alpha|, \quad \text{ for } t\leq 0.		
	\end{cases}
	\end{align}
	By the characterization of the center manifold, we deduce $g \in W_{\varepsilon}^c$. Now let $g=g_c+\theta_{\varepsilon}(g_c)$ be a function in $ W_{\varepsilon}^c$. From above we know $E_c\subset W_{\eps}^c$, and thus   $g_c \in W_{\varepsilon}^c$. This forces $\theta_{\varepsilon}(g_c)=0$, which proves the claim \eqref{claim}.
	
	Let us know consider an initial datum $w_0$ with $\|w_0\|_{W} \le \eps_2$ and let $w(t)=S^t(w_0)$ be the corresponding solution to the perturbation equation. The Invariant Manifold Theorem \ref{localmanifolds} combined with the characterization \eqref{claim} yields the existence of constant $a$  with
	\begin{align}\label{g1}
	\|w(t)-a\|_H\lesssim \left\|w(t)-a\right\|_{W}\lesssim e^{-\mu t},\quad \text{ for } t\geq 1.
	\end{align}
	In particular, if $a(t)$ denotes the average of $w(t)$ or, in other words, the projection onto $E_c$, $a(t)=P_c w(t)= \fint w(t)\rho dx$, it holds that $|a(t)-a| \lesssim e^{-\mu t}$. Invoking the hypothesis \eqref{h11}, this estimate entails that $a=0$. 
	
	We want to improve on the decay rate of $a(t)$.     We  note that $a(t)$ solves the equation 
	\begin{align}
	\frac{d}{dt}a(t)= \fint \nabla \cdot \left(\rho^2F[w]\right)+\rho^2F[w]dz = \fint \rho^2F[w]dz.
	\end{align}
	The nonlinear term $\rho F[w]$ consists of a linear combination of respective two factors of $\nabla w$, $\rho \nabla^2w$ or $\rho^2 \nabla^3w$, cf.~\eqref{106}. Thus, we obtain the estimate $|\rho F[w]| \lesssim \|w\|^2_{\dot{W}}$, where we consider only the homogeneous part of the norm.  From \eqref{g1} we already know that $\|w(t)\|_{\dot{W}}\lesssim e^{-\mu t}$ for $t\geq 1.$ We conclude
	\begin{align}
	\left|\frac{d}{dt}a(t)\right|\lesssim e^{-2\mu t}
	\end{align}
	for $t\geq 1$. We integrate over the time interval $(t,\infty)$ and recall the assumption \eqref{h11} to obtain
	\begin{align}\label{110}
	|a(t)| \lesssim e^{-2\mu t} \quad \text{ for } t\geq 1.
	\end{align}
	
	As we may choose $\mu$ larger than $\frac{1}{2}\mu_{1}$, it remains to gain suitable control over the projection of $w(t)$ onto $E_s\cong H/\mathbb{R}$, namely $P_sw(t) = w(t)-a(t)$. We note that $P_sw$ solves the equation 
	\begin{align}
	\partial_tP_sw +\left( \mathcal{L}^2+N\mathcal{L} \right)P_sw = P_s\left( \frac{1}{\rho}\nabla \cdot\left(\rho^2F[w]\right)+\rho F[w] \right).
	\end{align}
	Since the eigenfunctions $\left\{\psi_i\right\}_{i\in \mathbb{N}_0}$ form  an orthogonal basis of $H$, it holds that $\langle P_sw,\left(\mathcal{L}^2+N\mathcal{L}\right)P_sw \rangle_H \geq \mu_{1} \|P_sw\|^2_{H}$ and thus, arguing similarly as in the proof of Theorem \ref{wellposednessH1}, we find that
	\begin{align}
	\mel	\frac{1}{2}\frac{d}{dt}\left\|P_sw\right\|^2_{H} +\mu_{1}\left\|P_sw\right\|^2_{H}\\
	&\leq -\la \grad P_sw,\rho F[w]\ra  - \la \grad\L P_sw,\rho F[w]\ra + \la P_sw,\rho F[w]\ra +\la \L P_s w,\rho F[w]\ra\\
	&  \le \|P_sw \|_{W} \left(\|\rho F[w]\|_{L^\infty} + \|\rho F[w]\|_{L^\infty}\right)\\
	&\leq\left( \|w\|_W + |a(t)| \right) \left(\|\rho F[w]\|_{L^\infty} + \|\rho F[w]\|_{L^\infty}\right).
	\end{align}
	Thanks to the uniform estimates on the nonlinearities that we quoted above and the bound in \eqref{g1}, we observe that the right-hand side decays with at least 
	$  e^{-3\mu t}$. Therefore, the latter estimate translates into 
	\begin{align}
	\frac{d}{dt}\left(e^{2\mu_{1} t} \left\|P_sw\right\|^2_{H}\right)\lesssim e^{\left(2\mu_{1}- 3\mu\right)t},
	\end{align}
	for any $t\ge 1$.	The right hand side is integrable, provided that we choose $\mu$ sufficiently close to $\mu_{1}$, so that $3\mu >2\mu_1$. Integration in time yields
	\begin{align}
	\left\|P_sw(t)\right\|^2_{H} \lesssim e^{-2\mu_1 t} \quad \text{ for } t\geq 1.
	\end{align}
	In combination with our estimate on the average, \eqref{100},	this bound gives
	\begin{align}
	\left\|w(t)\right\|_{H} \leq \left\|P_sw(t)\right\|_{H}+\left|a(t)\right| \lesssim e^{-\mu_1 t}\quad \text{ for } t\geq 1.	
	\end{align}
	We take into account Lemma \ref{smoothingestimate} to finally obtain the statement of the theorem, noting that the result is trivial for $t\lesssim 1$.
\end{proof}
\begin{remark}\label{g2}
	Using the final result of Theorem \ref{stabilityw} we are able to improve the convergence rate of $a(t)$ to $|a(t)|\lesssim e^{-2\mu_1t}$ for all $t\geq 0$.
\end{remark}

Having already proved the part of Theorem \ref{Whoeheremoden} concerning the smallest eigenvalue, we are now able to deduce the full statement with an analogue approach.

\begin{proof}[ of Theorem \ref{Whoeheremoden}]
	We prove this theorem by induction. The base case $K=0$ is proved in the latter theorem.
	
	Now, may assume that \eqref{h2} holds true and additionally 
	\begin{align}\label{h3}
	\left\|w(t)\right\|_{W}\lesssim e^{-\mu_{K}t} \text{ for all } t\geq 0.
	\end{align}
	This directly implies $|\rho F[w]|\lesssim e^{-2\mu_Kt}$. We will again exploit the invariant manifolds in a similar way as in the base case. The center eigenspace takes the form $E_c = \Span\big \{\psi_{k,n}\, |\, k\in\left\{0,\dots,K\right\} \text{ and } n\in \left\{1,\dots,N_k\right\} \big\}.$
	We fix $\mu \in \left(\mu_{1},\mu_{2}\right)$ and accordingly $\varepsilon$ and $\varepsilon_0$ as in Theorem \ref{localmanifolds}. We deduce the existence of $\tilde w_0 \in W_{loc}^c$ such that $\tilde{w}(t)=S^t(\tilde w_0 )\in W_{loc}^c$ satisfies 
	\begin{align}\label{g3}
	\left\|w(t)-\tilde{w}(t)\right\|_{W}\lesssim e^{-\mu t}\quad \text{ for all } t\geq 1,
	\end{align}	
	where $\tilde{w}(t) = P_c\tilde{w}(t)+ \theta_\varepsilon\left(P_c\tilde{w}(t)\right)$ with $P_c\tilde{w}(t) = \sum \limits_{n,k} \la \tilde{w}(t),\psi_{k,n}\ra \psi_{k,n}$.
	
	Now, we fix an arbitrary $k\in \left\{0,\dots,K\right\}$ and consider the projection of $w$ onto one of the eigenfunctions $\psi_{k,n}$. We obtain the ordinary differential equation
	\begin{align}
	\frac{d}{dt}\langle \psi_{k,n},w(t)\rangle +\mu_k\langle \psi_{k,n},w(t)\rangle =-\langle \nabla\psi_{k,n},\rho F[w(t)]\rangle  +\langle \psi_{k,n},F[w(t)]\rangle  \quad \text{ for all }t\geq 0,
	\end{align}
	which implies $\left|\frac{d}{dt}e^{\mu_kt}\langle \psi_{k,n},w(t)\rangle \right| \lesssim e^{-\left( 2\mu_K-\mu_k \right)t}$ due to the bound on $|\rho F[w]|$. We notice that $\lim \limits_{t\rightarrow \infty} e^{\mu_k}\la \psi_{k,n},w(t)\ra $ exists and vanishes by the virtue of assumption \eqref{h2}. We conclude that 
	\begin{align}\label{h4}
	|\langle\psi_{k,n},w\rangle |\lesssim e^{-2\mu_Kt}\quad \text{ for all }t\geq 0.
	\end{align}
	This yields $\|P_c w(t)\|_W\lesssim e^{-2\mu_Kt} $ and enables us to estimate the center part of $\tilde{w}(t)$ with help of \eqref{g3} and the triangle inequality, namely
	\begin{align}
		\left\|P_c\tilde{w}(t)\right\|_W \leq \left\|P_c \left(w(t)-\tilde{w}(t)\right)\right\|_W + \left\|P_cw(t)\right\|_W\lesssim e^{-\min\left\{2\mu_K, \mu \right\}t}
	\end{align}
	for all $t\geq 1$.
	Thanks to the regularity property of $\theta_\varepsilon$ derived in the first part of Proposition \ref{centermanifold} we deduce
	\begin{align}
		\left\|\theta_\varepsilon\left(P_c\tilde{w}(t)\right)\right\|_W\lesssim e^{-\min\left\{2\mu_K, \mu \right\}t} \quad \text{ for all }t\geq 1.
	\end{align}
	Combining the previous estimates, we have
	\begin{align}\label{p1}
		\left\|w(t)\right\|_{W}& \leq \left\|w(t)-\tilde{w}(t)\right\|_{W} + \left\|P_c\tilde{w}(t)\right\|_W +  \left\|\theta_\varepsilon\left(P_c\tilde{w}(t)\right)\right\|_W\\
		&\lesssim e^{-\min\left\{2\mu_K, \mu \right\}t} \quad \text{ for all }t\geq 1.
	\end{align}
	
	We note that \eqref{p1} gives a better rate than \eqref{h3}. Due to the structure of the eigenvalues it may happen, depending on $K$ and the space dimension $N$, that $2\mu_K <\mu$. In this case, inequality \eqref{p1} downgrades to $\left\|w(t)\right\|_W\lesssim e^{-2\mu_K t}$. Similarly, in this case the estimate for center part of $w(t)$, that is $P_cw(t)$, is also not good enough, as we want to prove $|\la \psi_{k,n},w(t)\ra | \lesssim e^{-\mu_{K+1}t}$.
	We overcome this problem by repeating the first step of this proof, now from the starting point \eqref{p1} instead of \eqref{h2}, which directly yields $|\rho F[w]|\lesssim e^{-4\mu_Kt}$. If $2\mu_K \le \mu_{K+1}$, we deduce via iteration that
		\begin{align}\label{p2}
		\left\|P_cw(t)\right\|_W  \lesssim e^{-2^m\mu_K} \quad \text{ for all } t\geq 1
	\end{align}
	and
	\begin{align}\label{p3}
		\|w(t)\|_W\lesssim e^{-\mu t} \quad \text{ for all }t\geq 1,
	\end{align}
	where $m$ is smallest natural number that satisfies $\mu_K \leq 2^{m-1}\mu_K<\mu <\mu_{K+1} \leq2^m\mu_K$.
	We remark that we are allowed to choose $\mu$ sufficiently close to $\mu_{K+1}$.
	In the case $2\mu_K \geq \mu_{K+1}$, we may directly continue from estimate \eqref{p1}, which corresponds to  $m=1$.
	
	To achieve the rate $\mu_{K+1}$, we investigate the projection of $w(t)$ onto $E_s$. Similar to the previous proof, testing the equation solved by $P_sw$ with $\rho P_sw$ yields
	\begin{align}\MoveEqLeft
	\frac{1}{2}\frac{d}{dt}\left\|P_sw(t)\right\|^2_{H}+\mu_{K+1}\left\|P_sw(t)\right\|^2_{H}\\
	&\le-  \la \grad P_sw,\rho F[w]\ra  - \la \grad\L P_sw,\rho F[w]\ra + \la P_sw,\rho F[w]\ra +\la \L P_s w,\rho F[w]\ra\\
	&  \le \|P_sw \|_{W} \left(\|\rho F[w]\|_{L^\infty} + \|\rho F[w]\|_{L^\infty}\right)\lesssim e^{-3\mu t} \quad \text{ for all }t\geq 1,
	\end{align} 
	where we used \eqref{p3} and the quadratic behavior of $\rho F[w]$. Just like in the previous proofs, choosing $\mu$ large enough such that $3\mu >2\mu_{K+1}$ we obtain
	\begin{align}
	\left\|P_sw(t)\right\|^2_{H}\lesssim e^{-2\mu_{K+1}} \quad \text{ for all }t\geq 1
	\end{align}
	and in total
	\begin{align}
	\left\|w(t)\right\|_{H}\leq \left\|P_cw(t)\right\|_{H}+\left\|P_sw(t)\right\|_{H}\lesssim e^{-2\mu_Kt}+e^{-\mu_{K+1}t}\lesssim e^{-\mu_{K+1}t}\quad \text{ for all }t\geq 1.
	\end{align}
	To carry this result over to the $W$-norm it remains to make use of the smoothing estimate in Lemma \ref{smoothingestimate}, noting again that the result is trivial for $t\lesssim 1$.
\end{proof}

\appendix
\section*{Appendices}
\addcontentsline{toc}{section}{Appendices}
\renewcommand{\thesubsection}{\Alph{subsection}}
\numberwithin{theorem}{subsection}
\subsection{Derivation of the perturbation equation}\label{appendixA}
As announced, we will re-derive the perturbation equation \eqref{perturbationequation} from the confined thin film equation \eqref{FPE} with the intention to improve on the representation of the nonlinearity \eqref{nonlinearityperturbationequation} compared to the former derivation in \cite{SeisTFE}. 

We start by recalling from \cite{SeisTFE} that the transformation mapping	
$\Phi_t(x)=z$ is a diffeomorphism as long as the solution $v$ to the thin film equation \eqref{FPE} is close to the stationary solution in the sense of \eqref{o1}, or equivalently, as long as the perturbation $w$ is small in the sense of \eqref{126}.  This can be seen by inspecting the Jacobian determinant
\begin{equation}
\label{130}
\det \nabla_x \Phi= \frac{1}{(1+w+z\cdot\nabla w)(1+w)^{N-1}}.
\end{equation}
Moreover, it was proved the following relation between $x$- and $z$-derivatives:
For an arbitrary function $f=f(t,z)$, it holds that
\begin{align}\label{8}
\partial_{x_i}\left(f(t,\Phi_t)\right)= \frac{\partial_if}{\tilde{w}}- \frac{\partial_iw}{\tilde{w}h}z\cdot \nabla f,\quad \partial_t \left(f(t,\Phi_t)\right) = \partial_t f - \frac{ \partial_t\tilde w}{h}z\cdot \grad f,
\end{align}
where $\tilde{w}=1+w$ and $h=\tilde{w}+z\cdot \nabla \tilde{w}$. The spatial derivatives appearing on the right-hand side of the two equations are taken with respect to the $z$ variable. From the two transformations \eqref{transformationz} and \eqref{transformationw}, it follows that 
\begin{align}\label{7}
\rho^2\tilde{w}^4 =   v.
\end{align}
By differentiating this identity and using the formulas from \eqref{8}, it is straightforward to derive the perturbation equation \eqref{perturbationequation} from the confined thin film equation \eqref{FPE}. We will only give intermediate results to help the reader verifying the underlying computations.

First, differentiating \eqref{7} with respect to $x_i$ gives
\begin{align*}
\partial_iv 
=-2\rho\tilde{w}^3z_i + 2\frac{\rho\tilde{w}^3\partial_i\tilde{w}}{h}.
\end{align*}
Differentiating once more and summing over $i$ yields
\begin{align}
\frac{1}2 \laplace v \frac{h}{\tilde w^2} = \left((1-(N+2)\rho\right)h - \L w + p\star R[w]\star \left((\grad w)^{2\star} +\rho\grad w\star\grad^2w\right).
\end{align}
By use of \eqref{8}, we compute for an arbitrary $f(z)$ that
\begin{align}\label{9}\mel
\partial_{x_i}\Big( \Big( \frac{\tilde{w}^2}{h} f\Big)(\Phi)\Big) \\
&= \frac{\tilde{w}}{h}\Big(\partial_i f - z\cdot \nabla \Big(\frac{\partial_i \tilde{w} f}{h}\Big)\Big)\\
&= \frac{\tilde w}h \left(\partial_i f - \frac{z\cdot\grad\partial_iw }h f + \frac{\partial_i w}{h^2}\left(2z\cdot \grad w + z\otimes z:\grad^2 w \right)f - \frac{\partial_iw}h z\cdot\grad f\right).
\end{align}
Hence, differentiating the above identity for $\laplace v$ again yields
\begin{align}
	\frac{1}2 \partial_i\laplace v \frac{h}{\tilde w}  = (N+2)z_i h -\partial_i \L w - N\partial_i w + F[w].
\end{align}
After substracting $\frac{\gamma}{2}x_i \frac{h}{\tilde{w}}= (N+2)z_ih$, we make use of \eqref{7} to obtain
\begin{align}
	\frac{1}2v(\partial_i \laplace v -\gamma x_i) \frac{h}{\tilde w^5}  = -\rho^2 \partial_i\L w - N\rho^2 \partial_i w + \rho^2 F[w].
\end{align}
We  have to take one more spatial derivative, for which we derive the transformation formula
\begin{align}\label{2}
	\partial_{x_i}\Big[ \frac{\tilde{w}^5}{h}f(\Phi)\Big] = \frac{\tilde{w}^4}{h}\Big[ \partial_if + (N+3) \frac{ \partial_i\tilde{w}f}{h} -   \nabla\cdot \Big(z\frac{\partial_i\tilde{w}f}{h}\Big) \Big]
\end{align}
for an arbitrary function $f(z)$. Applying it to the third order derivatives above gives
\begin{align}
	\frac{1}2\partial_i (v\partial_i\laplace v - \gamma vx_i) \frac{h}{\tilde w^4} = -\partial_i(\rho^2 \partial_i\L w) - N \partial_i (\rho^2 \partial_i w) +\partial_i(\rho^2 F[w]) + \rho^2 F[w].
\end{align}
Dividing by $\rho$ and summing over $i$ finally yields
\begin{align}
	\frac{1}2 \grad\cdot (v\grad\laplace v - \gamma xv) \frac{h}{\rho\tilde w^4} =\L^2w + N\L w + \rho^{-1} \grad\cdot(\rho^2 F[w]) + \rho F[w].
\end{align}
It remains to consider the time derivative. With help of 
\eqref{8} we compute
\begin{align*}
\partial_tv = 2\frac{\rho \tilde{w}^4}{h}\partial_t\tilde{w}.
\end{align*}
With regard to the previous two identities, it is now straightforward to identify the confined thin film equation \eqref{FPE} with the perturbation equation \eqref{perturbationequation}.

\subsection{Inequalities}
In this second appendix we collect some useful inequalities for weighted Sovolev spaces from various references like \cite{CaffarelliKohnNirenberg1984,Kufner1985,KochHabilitation}, 
and \cite{Kienzler13}. For further details on the proofs, see also \cite{Seis15}.

The first estimate is a Sobolev embedding result with weight. We notice that the weight becomes visible in the Sobolev numbers, where the dimension is artificially increased from $N+1$ to $N+2$.

\begin{lemma}[Sobolev inequality]\label{sobolevinequality}
	Let  $1\leq p \leq q<\infty$ be such that 
	\begin{align}
		1-\frac{N+2}{p} = -\frac{N+2}{q}.
	\end{align}
	Then it holds that
	\begin{align}
		\|w\|_{L^q\left(L^q(\rho)\right)} \lesssim \|w\|_{L^p\left(L^p(\rho)\right)} +\|\partial_t w\|_{L^p\left(L^p(\rho)\right)}+\|\nabla w\|_{L^p\left(L^p(\rho)\right)}.
	\end{align}
\end{lemma}

Our second estimate is a Hardy inequality. We will use it with different exponents on the weight function.

\begin{lemma}[Hardy inequality]\label{hardytypeinequality}
	For any $p\in \left(1,\infty\right)$ and $\sigma > -1/p$ it holds that
	\begin{align}
			\|w\|_{L^p\left(\rho^\sigma \right)} \lesssim \|\rho w\|_{L^p(\rho^{\sigma})} +\|\rho\nabla w\|_{L^p(\rho^{\sigma})}.
	\end{align}
	In particular, for $\sigma =0$ and $p=2$ we obtain
	\begin{align}
		\|w\|_{L^2} \lesssim \|\rho w \|_{L^2} + \|\rho \nabla w \|_{L^2} \leq \|w\|_H.
	\end{align}
\end{lemma}

Next, we quote an interpolation inequality. Notice that, typical for interpolation inequalities, the dimension will not enter into the dimensional relation of the integrability exponents. As the weight ``increases'' the dimension of the underlying space --- as already noticed in our remark on the above Sobolev embedding --- the weight exponent $\sigma$ does not enter this  dimensional relation.

\begin{lemma}[Interpolation inequality]\label{interpolationintequality}
	For any $1\leq p,q,r\leq \infty$ such that
	\begin{align}
		\frac{2}{p}=\frac{1}{q}+\frac{1}{r}
	\end{align}
	and $\sigma > -1/p$ it holds that
	\begin{align}
		\|\nabla w\|_{L^p(\rho^\sigma)}\lesssim \|w\|^{\frac{1}{2}}_{L^q\left(\rho^\sigma\right)}\|\nabla^2w\|^{\frac{1}{2}}_{L^r\left(\rho^\sigma\right)}. 
	\end{align}
	In particular, for some integers $i<m$ we obtain
	\begin{align}\label{408}
		\|\nabla^i\xi\|^m_{L^p(\rho^{\sigma})}\lesssim \|\xi\|^{m-i}_{L^\infty}\|\nabla^m\xi\|^i_{L^r(\rho^\sigma)},
	\end{align}
	provided that $mr=pi$.
\end{lemma}

We complete this collection with an embedding into $L^{\infty}$. 

\begin{lemma}[Morrey inequality]\label{morreyinequality}
For any $q$ large enough it holds that
		\begin{align}
			\| w\|_{L^\infty} \lesssim \|w\|_{L^q\left(\rho^\sigma\right)} + \|\nabla w\|_{L^q\left(\rho^\sigma\right)}.
		\end{align}
		
\end{lemma}

\section*{Acknowledgement} 
The authors thank Linus Kramer for his help in identifying invariant vector spaces via  representation theory and Beomjun Choi for comments on a first draft of this paper.
This work is funded by the Deutsche Forschungsgemeinschaft (DFG, German Research Foundation) under Germany's Excellence Strategy EXC 2044 --390685587, Mathematics M\"unster: Dynamics--Geometry--Structure.

\medskip

\bibliography{mybib}
\bibliographystyle{abbrv}

\end{document}